\documentclass[12pt]{article}

\usepackage[a4paper, top=2.5cm, bottom=2.5cm, left=2.5cm, right=2.5cm]{geometry}

\usepackage[dvipsnames, table, xcdraw]{xcolor}  

\usepackage[toc,page]{appendix}

\usepackage[utf8]{inputenc}        
\usepackage{sans} 
\usepackage{pifont} 

\usepackage{graphicx}              
\usepackage{caption,subcaption}    
\usepackage{hyperref}              
\hypersetup{                       
	colorlinks=true,               
	linkcolor=blue,                
	filecolor=magenta,             
	urlcolor=cyan,                 
	citecolor=ForestGreen          
}
\graphicspath{{Figures/pdf/}{Figures/}}

\usepackage{amsmath}               
\usepackage{amssymb}               
\usepackage{amsfonts}              
\usepackage{amsthm}                
\usepackage{mathtools}
\usepackage{array}
\usepackage{multirow}
\usepackage{booktabs}
\usepackage{longtable}
\usepackage{physics}

\usepackage{enumitem}  
\usepackage[breakable,skins]{tcolorbox}

\theoremstyle{style-one}
\newtheorem{definition}{Definition}[section]  	
\newtheorem{proposition}{Proposition}[section]  
\newtheorem{theorem}{Theorem}[section]  		
\newtheorem{problem}{Problem}[section]          

\newtheorem*{definition*}{Definition}  
\newtheorem{remark}{Remark}[section]   

\newtheorem*{problem*}{Problem}  

\usepackage{algorithm}
\usepackage{algpseudocode}

\newcommand{\uH}{\mathbf{p}}
\newcommand{\sH}{\mathbf{q}}
\newcommand{\eH}{\mathbf{g}}
\newcommand{\lH}{\boldsymbol{\lambda}}
\newcommand{\mH}{\boldsymbol{\mu}}

\newcommand{\stH}{\mathbf{\tilde{q}}}
\newcommand{\etH}{\mathbf{\tilde{g}}}

\newcommand{\svgfig}[3]{%
	\begingroup
	\def\svgwidth{#1}%
	\edef\svgfigscale{\fpeval{(\the\dimexpr#1\relax)/(#2*1.00375)}}%
	\fontsize{\fpeval{round(12*\svgfigscale,2)}}{\fpeval{round(14.4*\svgfigscale,2)}}\selectfont
	\input{Figures/pdf/#3.pdf_tex}%
	\endgroup
}

\usepackage{textcomp, upquote}

\usepackage{listings}

\definecolor{exactshade}{HTML}{E8F0F7}   
\definecolor{exactink}{HTML}{14507A}     
\definecolor{errshade}{HTML}{FCEBEA}     
\definecolor{errink}{HTML}{A62B2B}       

\newcommand{\exhead}[1]{\textcolor{exactink}{\textbf{#1}}}
\newcommand{\exval}[1]{\textcolor{exactink}{#1}}

\newcommand{\errval}[1]{\cellcolor{errshade}\textcolor{errink}{#1}}

\definecolor{arterial}{HTML}{C0392B}
\definecolor{capillary}{HTML}{8E44AD}
\definecolor{venous}{HTML}{2471A3}

\definecolor{codegreen}{rgb}{0,0.6,0}
\definecolor{codegray}{rgb}{0.5,0.5,0.5}
\definecolor{codepurple}{rgb}{0.58,0,0.82}
\definecolor{backcolour}{rgb}{0.95,0.95,0.92}

\lstdefinestyle{mystyle}{
	backgroundcolor=\color{backcolour},
	commentstyle=\color{codegreen},
	keywordstyle=\color{magenta},
	numberstyle=\tiny\color{codegray},
	stringstyle=\color{codepurple},
	basicstyle=\ttfamily\scriptsize,
	breakatwhitespace=false,
	breaklines=true,
	captionpos=b,
	keepspaces=true,
	numbers=none,
	numbersep=5pt,
	showspaces=false,
	showstringspaces=false,
	showtabs=false,
	tabsize=2,
	upquote=true
}

\lstnewenvironment{python}[1][]
{
	\pythonstyle
	\lstset{#1}
}
{}

\newcommand\pythoninline[1]{{\pythonstyle\lstinline!#1!}}

\usepackage{authblk}

\title{A Data-Driven Computational Framework for Incompressible Flow in Hydraulic Networks}

\date{}

\author[1]{Pedro B. Bazon}

\author[2]{Cristian G. Gebhardt}

\author[1*]{Roberto F. Ausas}

\affil[1]{Instituto de Ci\^encias Matem\'aticas e de Computa\c{c}\~ao, Universidade de S\~ao Paulo\\ Av. Trab. S\~ao Carlense 400, 13566-590, São Carlos-SP, Brazil}

\affil[2]{University of Bergen, Geophysical Institute and Bergen Offshore Wind Centre\\ All\'{e}gaten 70, 5007 Bergen, Norway.}

\begin{document}
	
	\maketitle
	
	\let\thefootnote\relax\footnotetext{\textsuperscript{*}Corresponding author email: rfausas@icmc.usp.br} 
	
	\begin{abstract}
       The classical procedure for solving hydraulic networks relies on the assumption of constitutive equations, which state the relationship between the pressure gradient and fluxes along an edge. In this paper, we propose a data-driven framework that bypasses these constitutive models, formulating the incompressible flow problem directly on the graph topology. By assigning discrete measured data points to network edges, the problem is cast as a mixed-integer quadratic optimization over nodal pressures, edgewise states, and data assignments, accommodating both laminar (convex) and turbulent (non-convex) regimes. To solve this, we evaluate three algorithms: a GPU-accelerated Brute Force method, the Alternating Direction Method (ADM), and Deterministic Annealing (DA). The Brute Force method certifies global optima for small networks, establishing a good baseline for the iterative solvers. We demonstrate that ADM is highly sensitive to its initialization, requiring a faithful surrogate model to avoid local minima. In contrast, DA eliminates this dependence through unsupervised clustering. By annealing the data assignment from the centroid to strict nearest-neighbor projections, DA consistently reaches the global optimum without prior manifold reconstruction. Furthermore, numerical experiments reveal that DA is robust to noisy data, remains thermodynamically admissible on all but the coarsest and noisiest datasets, and sustains its convergence rate on larger networks where Brute Force is intractable and ADM degrades. Finally, the framework is successfully validated on complex configurations, including mixed-component networks and a $958$-edge arteriovenous bed featuring a non-Newtonian Carreau--Yasuda model, demonstrating its scalability and practical applicability.
	\end{abstract}
	
	
	
	\section{Introduction}

Continuum mechanics provides a rigorous framework for describing the motion and thermomechanical response of solids and fluids under the \emph{continuum hypothesis}. Its governing equations fall into four categories: (i) kinematics, which describes motion and deformation (e.g., strain-displacement relations); (ii) kinetics, involving equilibrium and conservation laws (e.g., mass and linear momentum); (iii) thermodynamic principles, governing the relationships between heat, work, and system equilibrium; and (iv) constitutive equations, which link kinematic and kinetic variables to characterize the material behavior (e.g., stress-strain relations) \cite{Reddy2013}.

Among these, constitutive equations play a central role. They establish the relationships between primary field variables (like density, pressure, temperature, and velocity) and secondary quantities (such as the stress tensor and heat flux). These fundamental equations characterize intrinsic material behavior and provide the closure required for the well-posedness of the continuous problem. Constitutive equations, however, are not derived from any universal physical principle. Instead, they are typically postulated based on experimental observations and phenomenological assumptions, such as \emph{ad hoc} empirical fits. 

In practice, a functional form is assumed for the constitutive response and is then constrained by invoking a set of axioms \cite{Reddy2013,Oden2011}. This empirical foundation has several important consequences. It introduces subjectivity into an otherwise principled description, and it becomes problematic when modeling heterogeneous, anisotropic, or strongly nonlinear behavior. Furthermore, calibrating and validating these models is costly. Once the empirical model is embedded in a coupled system, the fidelity of the numerical solution is fundamentally limited by the fidelity of the fit, regardless of how accurate the discretization might be.

These limitations have motivated computational strategies that integrate data directly into the solution process, bypassing explicit material modeling altogether. \emph{Data-Driven Computational Mechanics} (DDCM), introduced by Kirchdoerfer and Ortiz \cite{KirchdoerferOrtiz2016}, removes the constitutive equation from the formulation. Rather than fitting a model to experimental data and solving the resulting boundary-value problem, the solver assigns each material point the admissible state closest to a prescribed material dataset, subject to compatibility and equilibrium constraints.\footnote{This computational paradigm has also been referred to as a distance-minimizing scheme.} The problem thereby becomes one of constrained optimization, where a proximity measure quantifies the distance between admissible strain-stress pairs and the measured ones.

Building on these initial developments, Kirchdoerfer and Ortiz \cite{Kirchdoerfer2017a} recast the assignment through a \emph{maximum-entropy} principle. They graded the relevance of each datum using a temperature-dependent weight to handle data contaminated by outliers. This approach recovers the distance-minimizing scheme in the zero-temperature limit through a cooling schedule, which the authors call \emph{simulated} annealing, even though the iteration involves no stochastic sampling. This paradigm was subsequently applied to structural dynamics using Newmark time integration \cite{Kirchdoerfer2017b}, and has since been extended to finite deformations \cite{Nguyen2018} and inelastic behavior \cite{Eggersmann2019}. On the theoretical side, the existence of solutions and convergence as the dataset fills the manifold were established in \cite{ContiMullerOrtiz2018}. A complementary approach works in reverse, reconstructing the constitutive manifold from data using manifold-learning techniques before solving \cite{Ibanez2016,Ibanez2017}. Both approaches were later combined within a nonlinear-optimization framework covering static and dynamic elasticity \cite{Gebhardt2020,Gebhardt2020a}.

With few exceptions, this body of work is formulated on continua discretized by finite elements, predominantly in solid mechanics. Even the truss and beam assemblies commonly used to test the paradigm \cite{Nguyen2018} act as discretizations for an underlying continuum. The resulting discrete-continuous structure has received significant attention, including solvability, structure-specific initialization, and combinatorial constraints arising from contact \cite{Gebhardt2025,Gebhardt2024}. The continuous sub-problem obtained once the data assignment is fixed has also been thoroughly analyzed \cite{Bazon2026,Codina2026}. However, in all these cases, the combinatorial structure is inherited from the \emph{mesh}.

Comparatively little attention has been paid to systems whose governing equations are inherently discrete and whose natural structure is a \emph{graph} rather than a mesh. In these networks, unknowns live on nodes and edges, and the conservation and compatibility conditions are topological statements rather than differential ones. The closest precedent to the present work is \cite{GebhardtRoccia2025}, which develops a data-driven framework for electrical circuits based on discrete-continuous optimization in a dynamical setting.

That reference is a natural point of departure because hydraulic networks and electrical circuits are structurally identical. Nodal pressure plays a role analogous to electric potential, while fluid flux corresponds to electric current flowing through edges. This structural correspondence extends to the constitutive relations: the Hagen-Poiseuille law, relating flux to the pressure gradient in a duct, is structurally analogous to Ohm's law, which links current to the potential difference across a conductor. Table~\ref{tab:hndd_physical_interpretation} summarizes the correspondence between the physical quantities and governing equations of these two systems.

\begin{table}[h!]
	\centering
	\resizebox{0.95\textwidth}{!}{%
	\begin{tabular}{l l l}
		\toprule
		\textbf{Concept} & \textbf{Hydraulic} & \textbf{Electrical} \\
		\midrule
		Nodal Unknown            & Pressure $ p $                    & Voltage $ V $ \\
		Edge Unknown             & Flow rate $ q $                  & Electric current $ I $ \\
		Conductance              & Hydraulic conductance $ K $      & Electrical conductance $ G $ \\
		Constitutive Law    	 & Hagen--Poiseuille law              & Ohm's law \\
		Conservation Law         & Mass conservation                  & Kirchhoff's Current Law (KCL) \\
		Compatibility Condition  & Pressure drops in loops            & Kirchhoff's Voltage Law (KVL) \\
		\bottomrule
	\end{tabular}}
	\caption{Analogy between hydraulic and electrical systems: correspondence between key physical quantities and governing principles.}
	\label{tab:hndd_physical_interpretation}
\end{table}

The hydraulic setting is, however, a uniquely demanding testbed for the data-driven paradigm due to factors that have no counterpart in linear circuit theory. The relation between pressure gradient and flux is linear only in the laminar regime. Across the laminar-to-turbulent transition, it acquires a pronounced S-shape, making the set of admissible states genuinely non-convex. Furthermore, the closure is not always a constitutive equation in the strict sense; in the turbulent regime, resistance is inertial rather than material and relies on empirical correlation. A data-driven formulation, which requires only measured pressure-gradient/flux pairs, is indifferent to this distinction, which is precisely part of its appeal. Finally, realistic networks are strongly heterogeneous. Hydraulic systems assemble pipes with varying diameters connected through fitting elements, demanding careful mathematical and numerical treatment.

\textbf{Contributions.} The present work develops a data-driven formulation for incompressible flow in hydraulic networks, posed directly on a graph. Here, mass balance and compatibility are topological statements rather than differential ones, meaning the discrete-continuous structure is intrinsic to the network rather than inherited from a mesh. The resulting mixed-integer quadratic program covers both laminar and turbulent regimes and admits multiple constitutive types per network, each with its own dataset and metric weight. Beyond the theoretical formulation, our key algorithmic and applied novelties include:

\begin{itemize}
	\item[$\triangleright$] \textbf{A certified global optimum as a good baseline.} On networks small enough to enumerate, a GPU-accelerated Brute Force algorithm evaluates every data assignment to return the exact minimizer of the mixed-integer problem, restricted to thermodynamically admissible assignments. Using this certified optimum to benchmark the iterative solvers provides a quantitative comparison of accuracy previously absent from the literature.

	\item[$\triangleright$] \textbf{A Constitutive Manifold Reconstruction (CMR) initialization for the ADM.} We develop two variants: a linear least-squares fit for linear manifolds, and a $C^1$-continuous piecewise-cubic Hermite projection for non-convex ones. Comparing these variants demonstrates that the Alternating Direction Method (ADM) only reaches the global optimum when initialized with a high-fidelity reconstruction.

	\item[$\triangleright$] \textbf{A Deterministic Annealing (DA) solver on graphs.} By relaxing the data assignment into a soft, temperature-dependent weight, DA removes the initialization dependence that limits alternating-direction schemes. Adapting the maximum-entropy relaxation of \cite{Kirchdoerfer2017a} to networks, we characterize its two temperature limits (the data centroid as $\beta \to 0^{+}$, and the strict nearest-neighbor assignment as $\beta \to \infty$) alongside the exponential recovery rate that dictates the terminal temperature in practice.

	\item[$\triangleright$] \textbf{A stress-test of solvers under degraded data.} We systematically evaluate solver performance against variations in dataset resolution and measurement noise (Gaussian and uniform, $100$ realizations per configuration). We scale the network size across Cartesian grids containing both laminar and turbulent edges. By reporting accuracy and thermodynamic admissibility separately, we distinguish between standard optimization failures and violations of the second law.

	\item[$\triangleright$] \textbf{Integration of lumped elements.} A hub-and-arm decomposition expands fitting nodes into short arm sub-edges governed by a pipe-like resistance law, allowing fittings to use specific dataset keys. This introduces the first multi-key benchmark. Because the pipe and arm datasets differ by nearly a decade in aspect ratio, the per-dataset weight becomes critical. Since the certified optimum is unreachable for this scale, assessment relies on the agreement among independent solvers.

	\item[$\triangleright$] \textbf{Application to hemodynamics.} Blood rheology relies heavily on empirical fits, and microcirculatory fields are difficult to measure \emph{in vivo}. We apply our framework to an arteriovenous bed grown via \emph{Murray's law} ($958$ edges, $31$ dataset keys) featuring a shear-thinning capillary closure. This exercises the full capability of the formulation: managing multiple constitutive types, a nonlinear manifold, and massive conductance spreads. The recovered fields correctly return an arteriole-dominated pressure budget, demonstrating the framework's scalability.
\end{itemize}

\textbf{Outline.} Sec.~\ref{sec:hyd_net_model} reviews the classical hydraulic network model, detailing the graph structure, governing equations, constitutive closures, and matrix formulations. Sec.~\ref{sec:data_driven} develops the data-driven formulation, including the proximity measure, the mixed-integer problem, the continuous sub-problem, and sampling strategies. Sec.~\ref{sec:solution_algorithms} details the three solution algorithms (Brute Force, ADM, and DA) and breaks down their computational costs. Sec.~\ref{sec:numerical_experiments} presents the numerical experiments in order of increasing complexity: a linear manifold, a non-convex manifold, network scaling, and mixed-component graphs, concluding with the arteriovenous application. Finally, Sec.~\ref{sec:conclusion} summarizes our conclusions. Three appendices provide the nomenclature, data-generation algorithms, and pseudo-inverse results.

	\section{Review of the Classical Approach for Solving Hydraulic Networks}\label{sec:hyd_net_model}

This section briefly reviews the classical procedure for solving hydraulic networks. First, we present the graph structure used to describe the topology of the hydraulic system. Next, we introduce the governing equations, which include the mass balance, the compatibility relations, and the constitutive equations for both laminar and turbulent regimes. We then rewrite these equations in matrix-vector form and convert them to a dimensionless format by selecting a careful set of scaling parameters. Finally, we succinctly discuss the classical solution of the linear hydraulic system.

\subsection{Graph Structure}

A hydraulic network consists of a collection of ducts interconnected at junctions. From a topological standpoint, the network can be represented by a graph, where ducts are modeled as edges and junctions as nodes (or vertices). Specifically, we consider a hydraulic network as a connected graph embedded in $\mathbb{R}^2$ (Fig.~\ref{fig:hyd_net}). Accordingly, let $G = (\mathcal{V}, \mathcal{E})$ denote such a graph, where $\mathcal{V} = \{ \mathbf{v}_1, \dots, \mathbf{v}_{n_v} \}$ is the set of nodes and $\mathcal{E} = \{ e_1, \dots, e_{n_e} \}$ is the set of edges. Each node $\mathbf{v}_i$ is associated with a position in the plane through the map

\[
\mathbf{x}:\mathcal{V}\to\mathbb{R}^2, \qquad \mathbf{v}_i \mapsto \mathbf{x}_i = [x_i,\, y_i]^{\top}.
\]

The physical connectivity is undirected because fluid can flow in either direction along a duct. However, we assign an arbitrary but fixed reference orientation to each edge. Connectivity is understood in this undirected sense, ensuring that $G$ forms a single connected component.

\begin{figure}[h!]
	\centering
	\def\svgwidth{0.85\columnwidth}
	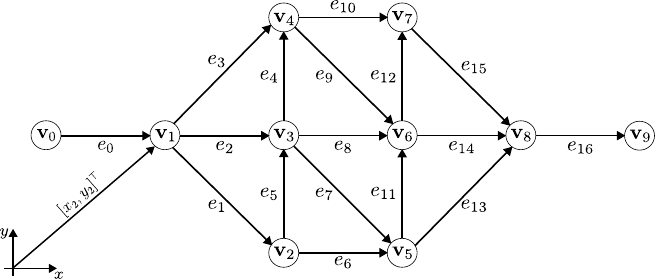
	\caption{Topological representation of a hydraulic network modeled as a connected graph $G = (\mathcal{V}, \mathcal{E})$ embedded in $\mathbb{R}^2$. Physical junctions are denoted by nodes $\mathbf{v}_i \in \mathcal{V}$, while connecting ducts correspond to edges $e_k \in \mathcal{E}$; arrows indicate the reference orientation that fixes the positive flow direction on each edge. A sample position vector is illustrated for node $\mathbf{v}_1$ relative to the global Cartesian coordinate system.}
	\label{fig:hyd_net}
\end{figure}

Following this reference orientation, each edge $e_k \in \mathcal{E}$ is represented as an ordered pair of nodes, $e_k = (\mathbf{v}_i, \mathbf{v}_j)$. Here, $\mathbf{v}_i$ is the initial (tail) node and $\mathbf{v}_j$ is the terminal (head) node, as indicated by the arrows in Fig.~\ref{fig:hyd_net}. The orientation fixes the sign of the flux. A positive flux $q_k > 0$ corresponds to flow from $\mathbf{v}_i$ toward $\mathbf{v}_j$, whereas $q_k < 0$ corresponds to flow in the opposite direction. This orientation is merely a sign convention used to define the incidence matrix $\mathbf{D}$. All governing equations remain invariant if any edge orientation is reversed.

Finally, we denote the total number of nodes and edges in the graph by $|\mathcal{V}| = n_v$ and $|\mathcal{E}| = n_e$, respectively. The set $\mathcal{V}_{D} \subset \mathcal{V}$ denotes the nodes with prescribed pressure. Conversely, $\mathcal{V}_{N} \subset \mathcal{V}$ corresponds to the nodes with an imposed fluid flux. These sets are disjoint, meaning $\mathcal{V}_{D} \cap \mathcal{V}_{N} = \emptyset$.

\subsection{Governing Equations}

To set the stage, let $A_k$, $L_k$, $r_k$, and $D_k := 2r_k$ denote the cross-sectional area, length, inner radius, and diameter of an arbitrary edge $e_k$, respectively. The fluid is assumed incompressible and the flow isothermal, so that any effects induced by temperature variations are neglected.

In this way, three primal fields describe the physical behavior of the internal fluid flow on each edge: the pressure, its gradient, and the flux. Let $p_i$ denote the pressure at node $\mathbf{v}_i$, whereas $g_k$ and $q_k$ correspond to the pressure gradient and the fluid flux on edge $e_k$, respectively. Building upon these definitions, we introduce
\begin{equation}
	\uH = [p_1, \dots, p_{n_v}]^{\top} \in \mathbb{R}^{n_v}, \quad
	\eH = [g_1,\dots, g_{n_e}]^{\top} \in \mathbb{R}^{n_e}, \quad
	\sH = [q_1,\dots, q_{n_e}]^{\top} \in \mathbb{R}^{n_e},
	\label{eq:primal_fields}
\end{equation}
as the vectors of nodal pressures, edge-wise pressure gradients, and fluxes, respectively.

The mathematical description of the flow rests on three equations. First, a node-wise mass balance serves as a physical conservation law. Second, an edge-wise compatibility relation acts as a purely kinematic constraint. Third, a constitutive relation characterizes the intrinsic material behavior. The mass balance and compatibility conditions hold for any fluid and are stated below, while the constitutive relation is addressed in Sec.~\ref{sec:constitutive}. To clarify the ideas, consider the flow through a single edge $e_k$ connecting the tail node $\mathbf{v}_i$ and the head node $\mathbf{v}_j$, as illustrated in Fig.~\ref{fig:hyd_net_edge}. The flow is driven by the pressure difference between the endpoints, $p_i$ and $p_j$. The volumetric flow rate through the edge is given by
\begin{equation}
	q_k = \overline{v}_k\, A_k,
	\label{eq:flux_def}
\end{equation}
where $\overline{v}_k = \frac{1}{A_k}\int_{A_k} v_k \, dA_k$ denotes the cross-sectional average of the axial velocity.
\begin{figure}[h!]
	\centering
	\def\svgwidth{0.25\columnwidth}
	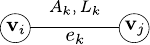
	\caption{Schematic representation of a single directed edge $e_k$ connecting an initial node $\mathbf{v}_i$ to a terminal node $\mathbf{v}_j$.}
	\label{fig:hyd_net_edge}
\end{figure}

In this setting, mass balance is enforced at each node. It requires that the algebraic sum of the flow rates of all incident edges equals a prescribed nodal source or sink term:
\begin{equation}
	\sum_{k \in \mathcal{E}(i)} d_{ik}\, q_k = f_i \qquad \text{for all nodes } i,
	\label{eq:mass_balance_index}
\end{equation}
where

\[
\mathcal{E}(i) := \{\, k \in \{1,\dots,n_e\} \mid \mathbf{v}_i \in e_k \,\}
\]

denotes the set of edges incident to node $\mathbf{v}_i$. The incidence matrix $\mathbf{D}=[d_{ik}]\in\mathbb{R}^{n_v\times n_e}$ is defined as
\begin{equation*}
	d_{ik} =
	\begin{cases}
		+1, & \text{if edge } e_k \text{ leaves node } \mathbf{v}_i,\\
		-1, & \text{if edge } e_k \text{ enters node } \mathbf{v}_i,\\
		0,  & \text{otherwise}.
	\end{cases}
\end{equation*}

The compatibility condition is a purely kinematic constraint. It states that the pressure gradient along each edge equals the pressure difference between its endpoints divided by the edge length:
\begin{equation}
	g_k = \frac{p_j - p_i}{L_k}.
	\label{eq:compat_def}
\end{equation}

Finally, a constitutive equation is required to relate the primary variables (pressure and its gradient) to the secondary variable (flux). This is discussed in the following section.

\subsection{Constitutive Equations}\label{sec:constitutive}

The mass balance and compatibility conditions hold independently of the fluid under study, whereas the constitutive equation does not. Thermodynamic consistency dictates that fluid flows in the direction of decreasing pressure. Guided by this principle, we adopt a general effective edge relation of the form:
\begin{equation}
	g_k = -R_{k}(q_k)\,q_k,
	\label{eq:const_general}
\end{equation}
where $R_{k}(q_k) > 0$ denotes a \textbf{hydraulic resistance} per unit length. Equation~\eqref{eq:const_general} stems from the one-dimensional reduction of a three-dimensional constitutive equation along a single edge. Its dependence on the flux is dictated by the fluid rheology, the pipe geometry, and the flow regime.

The resistance $R_k$ arises through two distinct mechanisms. In the laminar regime, it is derived from a genuine constitutive equation for a generalized Newtonian fluid reduced to a single duct. In the turbulent regime, the flow is inertial rather than material. Therefore, $R_k$ is supplied by an empirical closure that reuses the same mathematical form.

\begin{remark}[Role of the constitutive models]\label{rem:role_closures}
In all cases, the governing system is closed once $R_k$ is specified. This specification is exactly the ingredient that the data-driven formulation bypasses. The specific closures developed below are introduced for three limited purposes. First, they generate the synthetic data pairs that act as a substitute for an experimental campaign (Sec.~\ref{sec:dataset_generation}). Second, they provide the classical reference solutions against which the data-driven results are measured (Sec.~\ref{sec:classical_sol}). The solvers discussed in Sec.~\ref{sec:solution_algorithms} never evaluate these closures, relying entirely on the sampled pairs. Third, they are used to provide an initial guess for some of the solvers introduced later.
\end{remark}

\subsubsection{Laminar Regime}

We first consider the laminar regime, which is characteristic of low-Reynolds-number flows. Here, the resistance follows the first mechanism described above. It is obtained by reducing a genuine three-dimensional constitutive equation to a single duct. We state that reduction once for a general shear-rate-dependent viscosity, and then specialize it to the three fluids used in this work.

\paragraph{$\bullet$ Generalized Newtonian Fluids (GNF).}\label{sec:laminar_fluid_flows}

Many fluids of interest are non-Newtonian. Their viscosity is not constant and varies with the local shear rate. This is captured by the \emph{generalized Newtonian fluid} model \cite{FaithMorrison2001,Bird1987}:
\begin{equation}
	\boldsymbol{\tau} = -\eta(\dot\gamma)\,\dot{\boldsymbol{\gamma}},
	\qquad
	\dot\gamma = \sqrt{\tfrac{1}{2}\,\dot{\boldsymbol{\gamma}}\!:\!\dot{\boldsymbol{\gamma}}} ,
	\label{eq:gnf_constitutive_eq}
\end{equation}
where $\boldsymbol{\tau}$ is the viscous (extra) stress, $\dot{\boldsymbol{\gamma}} = \nabla\mathbf{v} + (\nabla\mathbf{v})^{\top}$ the rate-of-deformation tensor, $\dot\gamma$ the (scalar) shear rate, and $\eta(\dot\gamma)$ the shear-rate-dependent viscosity. We adopt the rheological sign convention throughout, where $\boldsymbol{\tau}$ enters the total momentum flux as $\boldsymbol{\pi} = p\,\mathbf{I} + \boldsymbol{\tau}$.

Consider a fully developed, axisymmetric, and inertialess flow through a straight circular pipe of diameter $D_k$. Equilibrium alone fixes the shear stress, which varies linearly from zero at the axis to $\tau_w$ at the wall for any fluid. The constitutive equation enters only through $\dot\gamma(\tau)$, the inverse of the flow curve $T(\dot\gamma) := \eta(\dot\gamma)\,\dot\gamma$. Integrating the axial velocity over the cross-section then yields the \emph{Rabinowitsch--Mooney} relation \cite{Rabinowitsch1929, Mooney1931, ChhabraRichardson2008}:
\begin{equation}
	q_k = \frac{8\pi}{|g_k|^{3}}\int_{0}^{\tau_w}\!\tau^{2}\,\dot\gamma(\tau)\,\mathrm{d}\tau ,
	\qquad
	\tau_w = \frac{D_k}{4}\,|g_k| ,
	\label{eq:rabinowitsch}
\end{equation}
which is a monotone, odd map through the origin. Its inversion recovers the edge relation~\eqref{eq:const_general}, with $R_k(q_k) = |g_k|/|q_k| > 0$. The functional form of $\eta(\dot\gamma)$ characterizes distinct constitutive fluid models and determines whether~\eqref{eq:rabinowitsch} can be integrated in closed form.

\paragraph{$\bullet$ Newtonian Fluid.} For a constant viscosity, $\eta \equiv \mu_f$, the flow curve is linear and $\dot\gamma(\tau) = \tau/\mu_f$. Equation~\eqref{eq:rabinowitsch} then integrates directly, yielding a resistance of
\begin{equation}
	R_{k} = \frac{128\,\mu_f}{\pi D_k^4},
	\label{eq:hyd_resistance_laminar}
\end{equation}
which is independent of the flux. The constitutive relation~\eqref{eq:const_general} therefore becomes linear, recovering the well-known \emph{Hagen--Poiseuille equation}:
\begin{equation}
	q_k = -K_k\,g_k, \quad \text{where}\; K_k := \frac{1}{R_{k}} = \frac{\pi D_k^{4}}{128\,\mu_f}
	\label{eq:hagen_poiseuille}
\end{equation}
is the \emph{hydraulic conductance} of edge $e_k$.

\paragraph{$\bullet$ Power-Law Fluid.} For a power-law fluid, $\eta = m\,\dot\gamma^{\,n-1}$, the flow curve is again a pure power, meaning $\dot\gamma(\tau) = (\tau/m)^{1/n}$ and~\eqref{eq:rabinowitsch} still integrates in closed form. The resistance now depends on the flux:
\begin{equation}
	R_{k}(q_k) = \frac{4m}{D_{k}}\left[\frac{8(3n+1)}{\pi n D_{k}^3}\right]^{n}|q_k|^{\,n-1} ,
	\label{eq:hyd_resistance_power_law}
\end{equation}
which means that the resistance decreases with the flux for \emph{shear-thinning} fluids $(n<1)$ and increases for \emph{shear-thickening} (or \emph{dilatant}) fluids $(n>1)$. The Newtonian case is recovered for $n=1$ and $m=\mu_f$.

\paragraph{$\bullet$ Carreau--Yasuda Fluid.} The power-law viscosity is unbounded. It diverges at rest and vanishes at high shear, making it unsuitable for fluids such as blood. The experimentally measured viscosity $\eta(\dot\gamma)$ is written as $\mu_{f,\mathrm{app}}$ from here on because it acts as an apparent viscosity. A formulation that captures this behavior is the \emph{Carreau--Yasuda} model \cite{Carreau1972,Yasuda1981}:
\begin{equation}
	\mu_{f,\mathrm{app}}(\dot\gamma) = \mu_{f,\infty} + (\mu_{f,0} - \mu_{f,\infty})\big[1 + (\lambda\,\dot\gamma)^{a}\big]^{\frac{n-1}{a}},
	\label{eq:carreau_yasuda}
\end{equation}
where $\mu_{f,0}$ and $\mu_{f,\infty}$ are the zero- and high-shear asymptotic viscosities, $\lambda$ a characteristic time, $a$ the transition sharpness, and $n<1$ the power-law index. The model interpolates between two Newtonian plateaus to bound the viscosity at both ends. For this model, however, the flow curve $T(\dot\gamma) = \mu_{f,\mathrm{app}}(\dot\gamma)\,\dot\gamma$ has no closed-form inverse, and~\eqref{eq:rabinowitsch} can no longer be integrated in elementary form.

The reduction nevertheless remains exact. Changing the variable of integration in~\eqref{eq:rabinowitsch} from the stress to the shear rate keeps every factor explicit:
\begin{equation}
	q_k(\dot\gamma_w) = \frac{\pi D_k^{3}}{8\,T(\dot\gamma_w)^{3}}
	\int_{0}^{\dot\gamma_w}\! T(\dot\gamma)^{2}\,\dot\gamma\;T'(\dot\gamma)\,\mathrm{d}\dot\gamma ,
	\qquad
	|g_k| = \frac{4\,T(\dot\gamma_w)}{D_k} ,
	\label{eq:carreau_yasuda_quadrature}
\end{equation}
so that sweeping the wall shear rate $\dot\gamma_w$ traces the constitutive curve, requiring a single quadrature and no root-finding. The resistance is bounded by the two Hagen--Poiseuille branches associated with the plateaus, meaning $128\,\mu_{f,\infty}/(\pi D_k^{4}) \le R_k(q_k) \le 128\,\mu_{f,0}/(\pi D_k^{4})$.

\subsubsection{Turbulent Regime}

Most flows encountered in hydraulic systems occur in the transitional-to-turbulent regime, where inertia dominates and the laminar reduction no longer applies. This represents the second mechanism discussed earlier. Turbulence is a property of the flow rather than the fluid itself, meaning the resistance cannot be derived from a constitutive equation. Instead, the pressure losses along a duct are described empirically using the \emph{Darcy--Weisbach} equation.

\paragraph{$\bullet$ Darcy--Weisbach Equation.}\label{sec:turbulent_fluid_flows}\label{sec:darcy_weisbach}

The \emph{Darcy--Weisbach} equation expresses the magnitude of the pressure gradient as
\begin{equation}
	|g_k| = \frac{f_k}{D_k}\frac{\rho\,\overline{v}_k^{2}}{2},
	\label{eq:darcy_weisbach}
\end{equation}
where $f_k$ is the \emph{friction factor} and $\rho$ the fluid density. Recalling $q_k = \overline{v}_k A_k$ from Eq.~\eqref{eq:flux_def} and writing the quadratic term as $|q_k|q_k$ to preserve the sign under flow reversal, the directional convention of Eq.~\eqref{eq:const_general} gives
\begin{equation}
	g_k = -f_k \frac{8\rho}{\pi^2 D_k^5}\,|q_k|q_k,
	\label{eq:darcy_weisbach_expanded}
\end{equation}
which identifies the hydraulic resistance as
\begin{equation}
	R_{k}(q_k) = f_k \frac{8\rho}{\pi^2 D_k^5}\,|q_k|,
	\label{eq:hyd_resistance}
\end{equation}
and reduces the turbulent closure to a single quantity, the friction factor $f_k$. This empirical form remains consistent with the laminar branch. For laminar pipe flow,
\begin{equation}
	f_k = \frac{64}{\mathrm{Re}_k}, \qquad \mathrm{Re}_k = \frac{\rho\,\overline{v}_k\,D_k}{\mu_f},
	\label{eq:friction_laminar}
\end{equation}
with $\mathrm{Re}_k$ serving as the \emph{Reynolds number} of edge $e_k$. Substituting~\eqref{eq:friction_laminar} into~\eqref{eq:hyd_resistance} returns the \emph{Hagen--Poiseuille} resistance~\eqref{eq:hyd_resistance_laminar}. In the transitional-to-turbulent regime, by contrast, $f_k$ depends on both $\mathrm{Re}_k$ and the relative roughness $\varepsilon_k/D_k$. This dependence cannot be derived from theory and has instead been established experimentally \cite{Cengel2014}. This leads to the \emph{Colebrook--White} equation:
\begin{equation}
	\frac{1}{\sqrt{f_k}} = -2.0\,\log_{10}\!\left(\frac{\varepsilon_k/D_k}{3.7} + \frac{2.51}{\mathrm{Re}_k\,\sqrt{f_k}}\right),
	\label{eq:colebrook_white}
\end{equation}
which is classically visualized through the \emph{Moody diagram}. Equation~\eqref{eq:colebrook_white} is \emph{implicit} in $f_k$, meaning that evaluating the friction factor requires iteration. Several explicit correlations have been proposed in its place, each valid over a restricted range \cite{Moody1944,Eck1973,SwameeJain1976,Haaland1983}. Throughout this work, we adopt the \emph{Churchill} correlation \cite{Churchill1977}:
{\small
\begin{equation}
	f_k = 8\left[\left(\frac{8}{\mathrm{Re}_k}\right)^{12} + \frac{1}{(\Theta_k+\Phi_k)^{1.5}}\right]^{1/12},
	\quad 
	\begin{cases}
	\begin{aligned}
	\Phi_k &= \left(\frac{37530}{\mathrm{Re}_k}\right)^{16},\\
	\Theta_k &= \left[-2.457\,\ln\left(\left(\frac{7}{\mathrm{Re}_k}\right)^{0.9}+0.27\,\frac{\varepsilon_k}{D_k}\right)\right]^{16},
	\end{aligned}
	\end{cases}	
	\label{eq:churchill}
\end{equation}}

This correlation is continuous across the laminar, transitional, and fully turbulent regimes. It possesses no singularity, even for $\mathrm{Re}_k < 1$. Finally, because $f_k = f_k(q_k)$ is flow-dependent, both the resistance $R_k(q_k)$ and the conductance $K_k(q_k)$ become nonlinear functions of the flux. The constitutive relation~\eqref{eq:const_general} therefore turns \emph{nonlinear and implicit}:
\begin{equation}
	q_k = -K_k(q_k)\, g_k,
	\label{eq:const_nonlinear_flux}
\end{equation}
which cannot be solved directly for $q_k$ and must instead be treated as the condition $F(q_k, g_k) := q_k + K_k(q_k)\,g_k = 0$.

\subsection{Matrix Form}\label{sec:matrix_form}

It is convenient to collect the governing equations in a standard matrix-vector notation using the primal fields of Eq.~\eqref{eq:primal_fields}. The mass balance from Eq.~\eqref{eq:mass_balance_index} can be written as
\begin{equation}
	\mathbf{D}\,\sH = \mathbf{f},
	\label{eq:mass_balance_matrix}
\end{equation}
where $\mathbf{f} \in \mathbb{R}^{n_v}$ represents the prescribed nodal source or sink term, and $\mathbf{D} \in \mathbb{R}^{n_v \times n_e}$ denotes the discrete divergence (incidence) operator. Similarly, the compatibility condition from Eq.~\eqref{eq:compat_def} reads
\begin{equation}
	\eH = \mathbf{L}^{-1}\,\mathbf{G}\,\uH,
	\label{eq:compat_eq_matrix}
\end{equation}
where $\mathbf{G} \in \mathbb{R}^{n_e \times n_v}$ is the discrete gradient operator, and $\mathbf{L} \in \mathbb{R}^{n_e \times n_e}$ is a diagonal matrix of edge lengths. If all the edges have the same length, then $\mathbf{L} = L\,\mathbf{I}$, with $\mathbf{I} \in \mathbb{R}^{n_e \times n_e}$. It is worth mentioning that the duality relation between the topological incidence and the gradient operator takes the form $\mathbf{G} = -\mathbf{D}^{\top}$.

Finally, for the linear (laminar) closure of Eq.~\eqref{eq:hagen_poiseuille}, the constitutive relation in matrix form reads
\begin{equation}
	\sH = -\mathbf{K}\,\eH,
	\label{eq:const_eq_matrix}
\end{equation}
where $\mathbf{K} \in \mathbb{R}^{n_e \times n_e}$ is a diagonal matrix collecting the fluid-intrinsic hydraulic conductances of the edges. Once again, if the conductance is uniform throughout the network, then $\mathbf{K} = K\,\mathbf{I}$, with $\mathbf{I} \in \mathbb{R}^{n_e \times n_e}$ denoting the identity matrix.

\subsection{Dimensionless Form}\label{sec:dimensionless}
It is convenient to recast the governing equations in dimensionless form. We first introduce the scaling parameters for pressure and length, which are geometric and topological constants independent of the flow regime. The pressure scale is the spread of the prescribed nodal pressures:
\begin{equation}
	\Delta p_{c} := p_{\mathrm{in}} - p_{\mathrm{out}}, \quad
	p_{\mathrm{in}} := \max_{i \in \mathcal{V}_D} p_i, \quad
	p_{\mathrm{out}} := \min_{i \in \mathcal{V}_D} p_i,
\end{equation}
which presupposes at least two distinct prescribed pressures so that $|\mathcal{V}_D| \geq 2$ and $\Delta p_c > 0$. For a single prescribed pressure, the scale is instead built from the imposed nodal fluxes. The length scale is taken as either the mean over all edges or the maximum:
\begin{equation*}
	L_c := \frac{1}{n_e}\sum_{k=1}^{n_e} L_k
	\qquad \text{or} \qquad
	L_c := \max_{k} L_k,
\end{equation*}
and the characteristic diameter $D_c$ is defined analogously from the edge diameters $D_k$. The two together fix the characteristic pressure gradient:
\begin{equation*}
	g_c := \frac{\Delta p_c}{L_c},
\end{equation*}
which is the scale against which the edge-wise gradients are normalized.

The characteristic conductance $K_c$ and flux $q_c$ depend on the constitutive behavior and must be adapted to the flow regime. In both regimes below, $K_c$ and the characteristic diameter $D_c$ are each obtained from the edge-wise conductances $K_k$ and diameters $D_k$ through one of three rules:
\begin{equation}
	K_c :=
	\begin{cases}
		\dfrac{1}{n_e}\displaystyle\sum_{k=1}^{n_e} K_k, & \text{(mean)}, \\[10pt]
		\displaystyle\max_{k} K_k, & \text{(max)}, \\[8pt]
		\left(\dfrac{1}{n_e}\displaystyle\sum_{k=1}^{n_e} \dfrac{1}{K_k}\right)^{-1}, & \text{(harmonic)}.
	\end{cases}\qquad
	D_c :=
	\begin{cases}
		\dfrac{1}{n_e}\displaystyle\sum_{k=1}^{n_e} D_k, & \text{(mean)}, \\[10pt]
		\displaystyle\max_{k} D_k, & \text{(max)}, \\[8pt]
		\left(\dfrac{1}{n_e}\displaystyle\sum_{k=1}^{n_e} \dfrac{1}{D_k}\right)^{-1}, & \text{(harmonic)}.
	\end{cases}
	\label{eq:agg_rule}
\end{equation}
The mean and max rules are dominated by the network's largest-valued edges, whereas the harmonic rule is dominated by the smallest-valued ones. The length scale $L_c$ has no such distinction and is always taken as the mean or the maximum of $L_k$, irrespective of which rule is used for $K_c$. Throughout the examples of Sec.~\ref{sec:numerical_experiments}, the mean rule is adopted for all three quantities ($L_c, D_c, K_c$) unless stated otherwise.

\begin{remark}[Characteristic conductance and diameter]
It is worth pointing out that the mean or max rule is preferable when either $K_k$ or $D_k$ is roughly homogeneous across the network (spanning at most one or two orders of magnitude). The harmonic rule is preferable when it is highly heterogeneous.
\end{remark}

\paragraph{$\bullet$ Laminar regime.}
Here the constitutive relation is linear, so a constant conductance $K_c$ is naturally available and taken as in Eq.~\eqref{eq:agg_rule}. The flux scale then follows as
\begin{equation}
	q_c := K_c \, \frac{\Delta p_c}{L_c}.
\end{equation}

\paragraph{$\bullet$ Turbulent regime.}
Here the conductance is flow-dependent (cf.~Sec.~\ref{sec:turbulent_fluid_flows}) and cannot serve as an a priori scale. Instead, we derive a flux scale from the Darcy--Weisbach relation. Inverting Eq.~\eqref{eq:darcy_weisbach} with the characteristic gradient $\Delta p_c/L_c$ and the characteristic diameter $D_c$ from Eq.~\eqref{eq:agg_rule} gives the characteristic velocity and flux:
\begin{equation}
	v_c := \sqrt{\frac{2 D_c \,\Delta p_c}{f_{\mathrm{ref}}\,\rho\,L_c}},
	\qquad
	q_c := v_c\,\frac{\pi D_c^2}{4},
\end{equation}
with $f_{\mathrm{ref}} = 0.02$ acting as a representative friction factor for turbulent pipe flow throughout. An equivalent conductance is then recovered by matching the linear scaling structure, $K_c = q_c\,L_c/\Delta p_c$.

\subsubsection*{Dimensionless Governing Equations}
By construction, the definitions above preserve the relation $q_c = K_c\,\Delta p_c / L_c$ regardless of the flow regime. With them, the dimensionless variables and matrices are introduced as
\begin{equation*}
	\uH^* = \frac{\uH - p_{\mathrm{out}}\mathbf{1}}{\Delta p_c}, \quad
	\sH^* = \frac{\sH}{q_c}, \quad
	\eH^* = \frac{\eH}{g_c}, \quad
	\mathbf{f}^* = \frac{\mathbf{f}}{q_c}, \quad
	\mathbf{K}^* = \frac{\mathbf{K}}{K_c}, \quad
	\mathbf{L}^* = \frac{\mathbf{L}}{L_c},
\end{equation*}
where $\mathbf{1} \in \mathbb{R}^{n_{v}}$ is the vector of all ones. Since the incidence operators $\mathbf{D}$ and $\mathbf{G} = -\mathbf{D}^{\top}$ are intrinsically dimensionless, and the constant pressure offset $p_{\mathrm{out}}\mathbf{1}$ lies in $\ker(\mathbf{G})$, so that $\mathbf{G}\,\uH = \Delta p_c\,\mathbf{G}\,\uH^*$. Substituting these definitions into the governing equations, and using $q_c = K_c\,\Delta p_c / L_c$ to cancel the prefactors, leaves their structure unchanged:
\begin{equation}
	\mathbf{D}\,\sH^* = \mathbf{f}^*, \qquad
	\eH^* = (\mathbf{L}^*)^{-1}\mathbf{G}\,\uH^*, \qquad
	\sH^* = -\mathbf{K}^*\,\eH^*.
	\label{eq:nd_governing}
\end{equation}

\subsection{The Classical Solution}\label{sec:classical_sol}

Before moving to the data-driven formulation, we recall the classical procedure for solving hydraulic network systems. We restrict the analysis to the \emph{linear constitutive regime} involving a Newtonian fluid in laminar flow. In this regime, the \emph{Hagen--Poiseuille} equation~\eqref{eq:hagen_poiseuille} holds with a constant conductance. The nonlinear closures of Sec.~\ref{sec:constitutive} yield a flux-dependent conductance $K_k(q_k)$ and are solved iteratively using methods like Picard or Newton--Raphson upon the same linear structure derived here.

Substituting the compatibility condition~\eqref{eq:compat_eq_matrix} into the constitutive law~\eqref{eq:const_eq_matrix}, and inserting the result into the mass balance~\eqref{eq:mass_balance_matrix} with $\mathbf{G} = -\mathbf{D}^{\top}$, gives the reduced linear system for the nodal pressures:
\begin{equation}
	\mathbf{D}\,\mathbf{K}\mathbf{L}^{-1}\mathbf{D}^{\top}\,\uH = \mathbf{f}.
	\label{eq:lin_sys_class_sol}
\end{equation}
Eliminating $\eH^*$ and $\sH^*$ from~\eqref{eq:nd_governing} yields the same structure in dimensionless form:
\begin{equation}
	\mathbf{G}^{\top}\mathbf{K}^*(\mathbf{L}^*)^{-1}\mathbf{G}\,\uH^* = \mathbf{f}^*
	\quad\Longleftrightarrow\quad
	\mathbf{D}\,\mathbf{K}^*(\mathbf{L}^*)^{-1}\mathbf{D}^{\top}\,\uH^* = \mathbf{f}^*,
	\label{eq:lin_sys_class_sol_nd}
\end{equation}
differing only through the normalized parameter matrices. The system matrix
\begin{equation}
	\mathbf{A} := \mathbf{G}^{\top}\mathbf{K}\mathbf{L}^{-1}\mathbf{G} \equiv \mathbf{D}\mathbf{K}\mathbf{L}^{-1}\mathbf{D}^{\top}
	\in \mathbb{R}^{n_v \times n_v}
\end{equation}
is a \emph{weighted discrete Laplacian} where each edge $e_k$ contributes with weight $K_k/L_k$, representing its effective conductance per unit length.

For $\mathbf{K}$ and $\mathbf{L}$ diagonal with strictly positive entries, $\mathbf{A}$ is symmetric and positive semidefinite. It is positive definite if and only if $\ker(\mathbf{G})$ is trivial. This is the standard property of weighted graph Laplacians, and follows at once from
\begin{equation}
	\mathbf{v}^{\top}\mathbf{A}\,\mathbf{v} = (\mathbf{G}\mathbf{v})^{\top}\mathbf{K}\mathbf{L}^{-1}(\mathbf{G}\mathbf{v}) = \sum_{k=1}^{n_e} \frac{K_k}{L_k}\,(\mathbf{G}\mathbf{v})_k^{2} \; \geq \; 0,
	\label{eq:laplacian_quadratic_form}
\end{equation}
which vanishes if and only if $\mathbf{G}\mathbf{v} = \mathbf{0}$.

For a connected graph, $\ker(\mathbf{G})$ is spanned by the constant vector $\mathbf{1} \in \mathbb{R}^{n_v}$, meaning the unreduced Laplacian $\mathbf{A}$ is only positive \emph{semi}definite and the nodal pressures are fixed up to an additive constant. Prescribing the pressure at one or more nodes removes this freedom. If $\mathcal{V}_D \neq \emptyset$, deleting the rows and columns of the constrained nodes eliminates the constant-pressure nullspace, causing the reduced matrix to become symmetric positive definite and invertible. The resulting reduced system then has a unique solution.

This procedure also fixes the role of each field. When the constitutive relation is known, only the nodal pressures are treated as primary unknowns. The gradient $\eH = \mathbf{L}^{-1}\mathbf{G}\,\uH$ follows from compatibility, and the flux is determined by $\sH = -\mathbf{K}\,\eH$. In energetic terms, $\eH$ and $\sH$ form a conjugate pair, with $\sH$ acting as the dual member. It is a derived quantity rather than an independent unknown. A discrete dataset breaks this hierarchy. Since no analytical law links $\sH$ to $\eH$, the flux becomes an unknown in its own right. Section~\ref{sec:data_driven} computes $(\uH,\eH,\sH)$ simultaneously, treating the constitutive relation as an objective rather than a strict equation. This explains why $\sH$ appears among the \emph{primal fields}. Moving forward, the terms \emph{primal} and \emph{dual} are used in an optimization context, where dual fields refer to the Lagrange multipliers that enforce physical constraints.

	\section{The Data-Driven Formulation}\label{sec:data_driven}

Turning to the \emph{data-driven} approach, we now formulate the continuous--discrete optimization problem considered herein. \emph{Unless otherwise stated, all quantities in this paper are dimensionless in the sense of Sec.~\ref{sec:dimensionless}. We drop the asterisk for readability. That is, we write $\uH, \eH, \sH$ instead of $\uH^*, \eH^*, \sH^*$, and similarly for the remaining quantities.}

The graph edges do not need to share a single constitutive behavior. Each edge $e_k$ is assigned a \emph{dataset key} $\mu(k) \in \mathcal{M}$, where $\mathcal{M}$ denotes the finite set of distinct constitutive types present in the network (e.g., pipes of different diameters or roughness). More generally, this key represents any grouping where edges share the same underlying constitutive relation (cf.~Sec.~\ref{sec:local_constitutive_sampling}). This assignment induces a partition of the edge set:
\begin{equation}
	\mathcal{E} = \bigsqcup_{\mu \in \mathcal{M}} \mathcal{E}_{\mu}, \qquad
	\mathcal{E}_{\mu} := \{e_k \in \mathcal{E} \mid \mu(k) = \mu\}.
\end{equation}
Each dataset key $\mu \in \mathcal{M}$ is associated with its own set of experimentally measured pressure-gradient/flux data pairs,
\begin{equation}
	\mathcal{D}^{(\mu)} := \big\{(g_n^{(\mu)}, q_n^{(\mu)})\big\}_{n=1}^{\mathcal{N}_m^{(\mu)}},
\end{equation}
where $\mathcal{N}_m^{(\mu)}$ is the number of sampled pairs for that dataset. The pooled dataset of the network is $\mathcal{D} := \bigsqcup_{\mu\in\mathcal{M}} \mathcal{D}^{(\mu)}$, with a total size of $|\mathcal{D}| = \sum_{\mu\in\mathcal{M}} \mathcal{N}_m^{(\mu)}$. This union is disjoint because each pair is tagged by its measured dataset key. Consequently, pairs with identical coordinates but different keys remain distinct.

These data pairs are not independent. Within each key, they sample a common \emph{constitutive manifold} $\Gamma^{(\mu)} \subset \mathbb{R}^2$, which represents the physical relation between the pressure gradient and flux for the edges carrying that key. The dataset and the manifold must remain distinct concepts. The continuous manifold $\Gamma^{(\mu)}$ is neither known nor constructed in the data-driven setting. Instead, the dataset $\mathcal{D}^{(\mu)}$ serves as a finite sample of this manifold. Because the data is generally corrupted by measurement error, $\mathcal{D}^{(\mu)} \not\subset \Gamma^{(\mu)}$ for noisy datasets.

The assignment of data pairs to the graph edges is performed, within each dataset key, using \emph{characteristic functions} defined as
\begin{equation}
	\begin{aligned}
		&\chi_n^{(\mu)} : \mathcal{E}_{\mu} \to \{0,1\},
		&&\qquad \mu \in \mathcal{M},\ n = 1,\dots,\mathcal{N}_m^{(\mu)},\\[1mm]
		\text{s.t. }& \sum_{n=1}^{\mathcal{N}_m^{(\mu)}} \chi_n^{(\mu)}(e_k) = 1,
		&&\qquad \text{for } e_k \in \mathcal{E}_{\mu}.
	\end{aligned}
\end{equation}
These functions act as binary selectors to assign a unique data pair $(g_n^{(\mu)}, q_n^{(\mu)}) \in \mathcal{D}^{(\mu)}$ to each edge $e_k \in \mathcal{E}_\mu$. Thus, every edge draws exclusively from the dataset associated with its own key.
Accordingly, the edge-wise data fields $(\etH_k, \stH_k)$ are:
\begin{equation}
	(\etH_k, \stH_k) := \sum_{n=1}^{\mathcal{N}_m^{(\mu(k))}} \chi_n^{(\mu(k))}(e_k)\,
	\big(g_n^{(\mu(k))},q_n^{(\mu(k))}\big),
	\qquad \text{for } k = 1, \dots, n_e.
	\label{eq:data_fields}
\end{equation}

\subsection{The Proximity Measure}\label{sec:proximity}

The data-driven formulation replaces the constitutive equation by a measure of how far a candidate state lies from the data. To this end, we define a weighted quadratic proximity measure $\mathcal{J}(\eH, \sH) : \mathbb{R}^{n_e} \times \mathbb{R}^{n_e} \rightarrow \mathbb{R}$ between the primal pressure-gradient/flux fields $(\eH,\sH)$ and the distributed data fields $(\etH,\stH)$, built from a symmetric positive-definite diagonal matrix $\mathbf{C} \in \mathbb{R}^{n_e \times n_e}$:
\begin{equation}
	\mathcal{J}(\eH, \sH) := \frac{1}{2}(\eH - \etH)^{\top}\mathbf{C}\,(\eH - \etH) + \frac{1}{2}(\sH - \stH)^{\top}\mathbf{C}^{-1}(\sH - \stH).
	\label{eq:proximity_measure}
\end{equation}
Pairing $\mathbf{C}$ with the pressure-gradient term and its inverse $\mathbf{C}^{-1}$ with the flux term reflects the conjugate nature of the two fields. Each diagonal entry $C_k := (\mathbf{C})_{kk} > 0$ stretches the metric along $\eH_k$ and compresses it along $\sH_k$ by the reciprocal factor. This ensures that the per-edge weighting maintains a unit determinant.

Because each edge $e_k$ draws its data from the dataset $\mathcal{D}^{(\mu(k))}$ of its own key, the natural choice for $C_k$ is also key-dependent and remains constant across each group $\mathcal{E}_\mu$. We choose the per-dataset scalar so that deviations in $\eH$ and $\sH$ contribute equally to $\mathcal{J}$ relative to that specific dataset $\mathcal{D}^{(\mu)}$. Accordingly, let $\sigma_{g}^{(\mu)}$ and $\sigma_{q}^{(\mu)}$ denote the standard deviations of the pressure-gradient and flux coordinates over $\mathcal{D}^{(\mu)}$. We then set:
\begin{equation}
	C_k := c^{(\mu(k))}, \qquad c^{(\mu)} := \frac{\sigma_{q}^{(\mu)}}{\sigma_{g}^{(\mu)}},
	\qquad \mu \in \mathcal{M}.
	\label{eq:c_choice}
\end{equation}
This definition ensures that $\sqrt{c^{(\mu)}}\,\sigma_{g}^{(\mu)} = \sigma_{q}^{(\mu)}/\sqrt{c^{(\mu)}} = \sqrt{\sigma_{g}^{(\mu)}\sigma_{q}^{(\mu)}}$. As a result, neither field dominates the proximity measure within a single constitutive dataset.

\begin{remark}[Weighting \& Nondimensionalization]\label{rem:weighting_vs_nondim}
Weighting and nondimensionalization play distinct, complementary roles. Nondimensionalization applies a single global set of characteristic scales to remove physical units and fix the overall magnitude of the fields. However, it does not balance the two coordinates within a dataset. A single set of scales cannot equalize the spreads of $\eH$ and $\sH$ across every dataset, nor can it keep every edge of a highly heterogeneous network at order one. The per-dataset weight $c^{(\mu)}$ directly addresses this residual imbalance on a dataset-by-dataset basis.
\end{remark}

\subsection{Optimization Problem (OP)}\label{sec:op}

With the proximity measure defined, we can state the full data-driven problem. We seek the state $(\uH,\eH,\sH)$ that lies closest to the data among all states satisfying mass balance, compatibility, and the prescribed boundary conditions, while leaving the data assignment free. Therefore, the primal fields and the selectors $\chi_n^{(\mu)}$ are determined jointly. In explicit form, the problem is:
\begin{equation}
	\begin{aligned}
		\min_{\uH,\eH,\sH,\,\{\chi_n^{(\mu)}\}\in\mathcal{S}} \quad
		& \frac{1}{2}(\eH - \etH)^{\top}\mathbf{C}\,(\eH - \etH)
		+ \frac{1}{2}(\sH - \stH)^{\top}\mathbf{C}^{-1}(\sH - \stH) \\
		\text{s.t.}~
		&\begin{cases}
			\begin{aligned}
				\mathbf{D}\,\sH &= \mathbf{f}, &&\text{in } \mathcal{V} \\
				\eH &= \mathbf{L}^{-1}\mathbf{G}\,\uH, &&\text{in } \mathcal{E} \\
				\mathbf{B}\,\uH &= \uH_{D}, &&\text{on } \mathcal{V}_D,
			\end{aligned}
		\end{cases}
	\end{aligned}
	\label{eq:op}
\end{equation}
where the data fields $(\etH,\stH)$ are tied to the selectors through Eq.~\eqref{eq:data_fields}, and the admissible set of assignments is
\begin{equation}
	\mathcal{S} := \left\{ \{\chi_n^{(\mu)}\} \;\middle|\; \chi_n^{(\mu)}(e_k) \in \{0,1\}, \, \sum_{n=1}^{\mathcal{N}_{m}^{(\mu(k))}}\chi_n^{(\mu(k))}(e_k) = 1, \, \text{for } k = 1, \dots, n_e \right\}.
	\label{eq:admissible_set}
\end{equation}
Here, $\uH_D \in \mathbb{R}^{|\mathcal{V}_D|}$ groups the prescribed pressure values. The matrix $\mathbf{B} \in \mathbb{R}^{\vert{}\mathcal{V}_D\vert{} \times n_v}$ is a selection matrix that extracts the pressure degrees of freedom associated with the nodes in $\mathcal{V}_D$:
\begin{equation*}
	b_{ij} =
	\begin{cases}
		1, & \text{if } j = \nu_i, \\
		0, & \text{otherwise},
	\end{cases}
	\qquad
	\begin{aligned}
		&i=1,\dots,|\mathcal{V}_D|,\\
		&j=1,\dots,n_v,
	\end{aligned}
\end{equation*}
where $\nu_i$ denotes the global index of the $i$-th node in the set $\mathcal{V}_D$.

\begin{remark}[Boundary Conditions]
Prescribed nodal pressures can be imposed using Lagrange multipliers for the algebraic constraints, or by directly modifying the discrete algebraic system derived from the optimality conditions. Both approaches are algebraically equivalent, but we adopt the latter strategy for simplicity.
\label{rem:dirichlet_bc}
\end{remark}

To recast Eq.~\eqref{eq:op} as a saddle-point problem, we introduce two vectors of Lagrange multipliers:
\begin{equation}
	\lH = [\lambda_1,\dots, \lambda_{n_v}]^{\top} \in \mathbb{R}^{n_v}, \qquad
	\mH = [\mu_1,\dots, \mu_{n_e}]^{\top}\in \mathbb{R}^{n_e},
\end{equation}
The former enforces mass balance, and the latter ensures the compatibility condition. For notational convenience, we also define the primal and dual spaces as:
\begin{equation}
	\mathbb{X} := \mathbb{R}^{n_v} \times \mathbb{R}^{n_e} \times \mathbb{R}^{n_e}, \qquad
	\mathbb{Y} := \mathbb{R}^{n_v} \times \mathbb{R}^{n_e},
\end{equation}
such that $(\uH,\eH,\sH) =: \mathbf{x} \in \mathbb{X}$ and $(\lH,\mH) =: \mathbf{y} \in \mathbb{Y}$. With these definitions at hand, the OP reads:

\begin{tcolorbox}[colback=black!3!white, colframe=black!45!white, halign=justify, boxrule=0.6pt, title = OP]
\begin{problem}
\label{prob:op}
Find the primal fields $\mathbf{x} \in \mathbb{X}$, the dual fields $\mathbf{y} \in \mathbb{Y}$, and the data-assignment variables $\{\chi_n^{(\mu)}\} \in \mathcal{S}$ such that
\begin{equation}
			(\mathbf{x},\mathbf{y},\{\chi_n^{(\mu)}\})
			=
			\arg\min_{\substack{\mathbf{x}\in \mathbb{X}\\
					\{\chi_n^{(\mu)}\} \in \mathcal{S}}}
			\;\sup_{\mathbf{y}\in \mathbb{Y}}
			\mathcal{L}(\mathbf{x},\mathbf{y}, \{\chi_n^{(\mu)}\}),
		\end{equation}
where the Lagrangian $\mathcal{L}:\mathbb{X}\times\mathbb{Y}\times\mathcal{S}\to\mathbb{R}$ is defined by
\begin{equation}
			\mathcal{L}(\mathbf{x},\mathbf{y},\{\chi_n^{(\mu)}\}) := \mathcal{J}(\eH,\sH) + \lH^{\top}(\mathbf{D}\,\sH-\mathbf{f}) + \mH^{\top}(\eH-\mathbf{L}^{-1}\mathbf{G}\,\uH),
			\label{eq:lag_op}
		\end{equation}
with $\mathcal{J}$ the proximity measure of Eq.~\eqref{eq:proximity_measure} and the data fields $(\etH,\stH)$ given by Eq.~\eqref{eq:data_fields}. The dependence of $\mathcal{L}$ on the assignment is carried entirely by $(\etH,\stH)$.
\end{problem}
\end{tcolorbox}

Problem~\ref{prob:op} presents important challenges. First, the data-assignment search space grows combinatorially. Each edge $e_k$ selects one datum from its constitutive dataset $\mathcal{D}^{(\mu(k))}$ of size $\mathcal{N}_m^{(\mu(k))}$, meaning the total number of possible data distributions across the graph is:
\begin{equation*}
	N_{\mathrm{comb}} := \prod_{k=1}^{n_e} \mathcal{N}_m^{(\mu(k))}
	= \prod_{\mu\in\mathcal{M}} \big(\mathcal{N}_m^{(\mu)}\big)^{|\mathcal{E}_{\mu}|},
\end{equation*}
This massive combination renders exhaustive enumeration intractable beyond small networks. Second, the problem mixes discrete and continuous variables. The discrete binary indicators assign a unique data pair to each edge, while the continuous variables represent the primal and dual unknowns. Because the proximity functional is quadratic and the physical constraints are linear, this is a Mixed-Integer Quadratic Programming (MIQP) problem \cite{DelPiaDeyMolinaro2017}. This class is known to be NP-hard \cite{GareyJohnson1979}, as it contains integer programming as a special case.

\begin{remark}[Thermodynamic Consistency]\label{rem:thermo_consistency}
Neither Problem~\ref{prob:op} nor its restriction (Problem~\ref{prob:sop}) enforces thermodynamic consistency between the pressure-gradient and flux fields. Because dissipation is a local statement, the \emph{Clausius--Duhem inequality} \cite{Reddy2013} is evaluated edge-wise. A pair $(\eH,\sH)$ is considered \emph{thermodynamically admissible} if:
\begin{equation}
		g_k\,q_k \leq 0, \qquad k = 1, \dots, n_e,
		\label{eq:clausius_duhem}
	\end{equation}
This means the flux on every edge must be directed down its own pressure gradient. Summing over the edges yields the network-wide form $\eH^{\top}\sH \leq 0$. Each of these $n_e$ inequalities is non-convex and bilinear, so we do not impose them as constraints. Instead, we check for admissibility \emph{a posteriori} at each iterate.
\end{remark}

\subsection{Sub-optimization Problem (SOP)}\label{sec:sop}

The combinatorial nature of the assignment $\{\chi_n^{(\mu)}\}$ makes the OP difficult to solve. If this assignment is prescribed in advance, the data fields $(\etH,\stH)$ become known. The minimization in Eq.~\eqref{eq:op} then applies only to the continuous variables, resulting in a quadratic objective with linear equality constraints. We refer to this restricted formulation as the Sub-optimization Problem (SOP).

\begin{tcolorbox}[colback=black!3!white, colframe=black!45!white, halign=justify, boxrule=0.6pt, title = SOP]
\begin{problem}
\label{prob:sop}
Given an assignment $\{\chi_n^{(\mu)}\} \in \mathcal{S}$, and hence the data fields $(\etH,\stH)$, find the primal fields $\mathbf{x} \in \mathbb{X}$ and the dual fields $\mathbf{y} \in \mathbb{Y}$ such that
\begin{equation}
			(\mathbf{x},\mathbf{y})
			=
			\arg\min_{\mathbf{x}\in \mathbb{X}}
			\;\sup_{\mathbf{y}\in \mathbb{Y}}
			\mathcal{L}(\mathbf{x},\mathbf{y}),
		\end{equation}
where $\mathcal{L}$ is the Lagrangian of Eq.~\eqref{eq:lag_op} at the prescribed assignment,
\begin{equation}
			\mathcal{L}(\mathbf{x},\mathbf{y})
			=
			\mathcal{J}(\eH,\sH)
			+ \lH^{\top}(\mathbf{D}\,\sH - \mathbf{f})
			+ \mH^{\top}(\eH - \mathbf{L}^{-1}\mathbf{G}\,\uH).
			\label{eq:lag_sop}
		\end{equation}
\end{problem}
\end{tcolorbox}

From the Lagrangian in Eq.~\eqref{eq:lag_sop}, one readily derives the following optimality conditions:
\begin{equation}
	\begin{array}{rcrcl}
		\nabla_{\uH}\mathcal{L}(\mathbf{x},\mathbf{y}) &=& -\mathbf{G}^{\top}\mathbf{L}^{-1} \mH &=& \mathbf{0}, \\
		\nabla_{\eH}\mathcal{L}(\mathbf{x},\mathbf{y}) &=& \mathbf{C}(\eH - \etH) + \mH &=& \mathbf{0}, \\
		\nabla_{\sH}\mathcal{L}(\mathbf{x},\mathbf{y}) &=& \mathbf{C}^{-1}(\sH - \stH) + \mathbf{D}^{\top} \lH &=& \mathbf{0}, \\
		\nabla_{\lH}\mathcal{L}(\mathbf{x},\mathbf{y}) &=& \mathbf{D} \sH - \mathbf{f} &=& \mathbf{0}, \\
		\nabla_{\mH}\mathcal{L}(\mathbf{x},\mathbf{y}) &=& \eH - \mathbf{L}^{-1}\mathbf{G} \uH &=& \mathbf{0}.
	\end{array}
\end{equation}

These conditions yield the following linear system:
\begin{equation}
	\begin{cases}
		\begin{aligned}
			-\mathbf{G}^{\top}\mathbf{L}^{-1} \mH ~&=~ \mathbf{0}, \\
			\mathbf{C}\,\eH + \mH ~&=~ \mathbf{C}\,\etH, \\
			\mathbf{C}^{-1}\sH + \mathbf{D}^{\top} \lH ~&=~ \mathbf{C}^{-1}\stH, \\
			\mathbf{D} \sH ~&=~ \mathbf{f}, \\
			\eH - \mathbf{L}^{-1}\mathbf{G} \uH ~&=~ \mathbf{0},
		\end{aligned}
	\end{cases}
	\label{eq:opt_cond_sys}
\end{equation}
which should be solved together with the following set of boundary conditions:
\begin{equation}
	\begin{aligned}
		\mathbf{B}\,\uH &= \uH_{D}, \\
		\mathbf{B}\,\lH &= \mathbf{0}_{D},
	\end{aligned} \qquad \mbox{on } \mathcal{V}_{D}.
\end{equation}

By ordering the variables as $(\uH, \eH, \sH, \lH, \mH)$, we can recast the system into block-matrix form. To enforce nodal boundary conditions directly within this structure, we introduce a diagonal mask matrix $\mathbf{I}_D \in \mathbb{R}^{n_v \times n_v}$. Its $i$-th entry is $1$ if $\mathbf{v}_i \in \mathcal{V}_D$ and $0$ otherwise:
\begin{equation}
	\begin{bmatrix}
		\mathbf{I}_D     & \mathbf{0}       & \mathbf{0}       & \mathbf{0}        & -\widehat{\mathbf{G}}^{\top}\mathbf{L}^{-1} \\
		\mathbf{0}       & \mathbf{C}       & \mathbf{0}       & \mathbf{0}        & \mathbf{I}       \\
		\mathbf{0}       & \mathbf{0}       & \mathbf{C}^{-1}  & \mathbf{D}^{\top} & \mathbf{0}       \\
		\mathbf{0}       & \mathbf{0}       & \widehat{\mathbf{D}} & \mathbf{I}_D  & \mathbf{0}       \\
		-\mathbf{L}^{-1}\mathbf{G} & \mathbf{I}       & \mathbf{0}      & \mathbf{0}        & \mathbf{0}       \\
	\end{bmatrix}
	\begin{bmatrix}
		\uH \\
		\eH \\
		\sH \\
		\lH \\
		\mH \\
	\end{bmatrix}
	=
	\begin{bmatrix}
		\hat{\uH} \\
		\mathbf{C}\,\etH \\
		\mathbf{C}^{-1}\stH \\
		\hat{\mathbf{f}} \\
		\mathbf{0} \\
	\end{bmatrix}.
	\label{eq:block_matrix}
\end{equation}
We obtain the modified matrices $\widehat{\mathbf{G}}^{\top}$ and $\widehat{\mathbf{D}}$ by zeroing the rows associated with constrained nodes ($\mathbf{v}_i \in \mathcal{V}_D$). This decoupling allows the diagonal matrix $\mathbf{I}_D$ to strictly enforce the boundary conditions. The right-hand side vectors are adjusted accordingly. At constrained nodes, $\hat{\uH}$ assumes the prescribed pressures $\uH_D$, and the corresponding entries in $\hat{\mathbf{f}}$ are zeroed. Conversely, at free nodes ($\mathbf{v}_i \notin \mathcal{V}_D$), $\hat{\uH}$ is set to $\mathbf{0}$ while $\hat{\mathbf{f}}$ retains its original value ($\hat{\mathbf{f}} \leftarrow \mathbf{f}$).

\begin{remark}
The system of equations \eqref{eq:opt_cond_sys} naturally decouples into two independent subsystems. The first involves $(\sH, \lH)$ and corresponds to the third and fourth equations. The second involves $(\uH, \eH, \mH)$ and corresponds to the first, second, and fifth equations.
\end{remark}

\begin{remark}[The single-dataset case]\label{rem:single_dataset}
For networks with only one constitutive type ($\vert{}\mathcal{M}\vert{} = 1$), we drop the key superscript entirely. In this case, the network draws from a single dataset $\mathcal{D} = \{(g_n,q_n)\}_{n=1}^{\mathcal{N}_m}$. The selectors simplify to $\chi_n(e_k)$, the metric weight of Eq.~\eqref{eq:c_choice} becomes $\mathbf{C} = c\,\mathbf{I}$, and the assignment count reduces to $N_{\mathrm{comb}} = \mathcal{N}_m^{n_e}$. This simplification does not lose generality. The edges couple only through the mass-balance and compatibility constraints, which are independent of the dataset keys, while the data assignment acts independently on each edge. We recover the general case by allowing each edge to draw from its specific dataset $\mathcal{D}^{(\mu(k))}$ with its respective weight $c^{(\mu(k))}$.
\end{remark}

\subsection{On the Dataset Generation}\label{sec:dataset_generation}

The datasets $\mathcal{D}^{(\mu)}$ are built using either \emph{Global Graph Sampling} or \emph{Local Constitutive Sampling}. Both strategies generate dimensionless pairs according to the scales defined in Sec.~\ref{sec:dimensionless}.

Because no experimental database of pressure-gradient/flux pairs is available for these networks, both strategies rely on the closures from Sec.~\ref{sec:constitutive} (as described in Remark~\ref{rem:role_closures}). Sampling a known closure enables strict verification. The classical solution provides a ground truth, ensuring that any discrepancies stem from the solver rather than the data. This approach also allows us to directly control dataset resolution, noise levels, and the sampled operating range. We investigate these three quantities thoroughly in Sec.~\ref{sec:numerical_experiments}.

\subsubsection{Global Graph Sampling}\label{sec:global_graph_sampling}

This strategy solves the dimensionless classical linear problem (Eq.~\eqref{eq:lin_sys_class_sol_nd}) on the target graph to extract the pair $(g_k, q_k)$ from each edge, yielding $\mathcal{N}_m = n_e$. This is the natural way to validate a solver. The exact solution is known \emph{a priori}, and the dataset perfectly matches the pairs the solver is expected to recover (Fig.~\ref{fig:global_graph_sampling}). Consequently, any error is isolated strictly to the solver rather than the data.

\begin{figure}[h!]
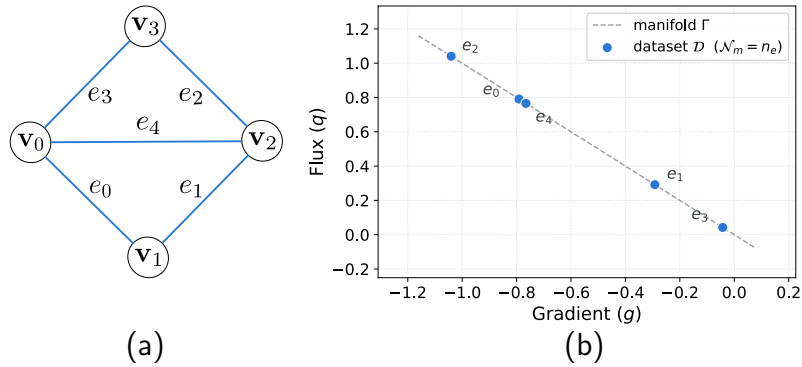

	\centering
	\svgfig{0.665\columnwidth}{301.91485383}{global_graph_sampling}
	\caption{(a) \emph{Wheatstone bridge} graph. (b) The pair $(g_k, q_k)$ read off each edge after solving the laminar linear problem under the prescribed boundary conditions, labeled by edge; the dashed line is the constitutive manifold $\Gamma$.}
	\label{fig:global_graph_sampling}
\end{figure}

However, this strategy presents two limitations. First, the dataset provides only one point per edge, resulting in a sparse sample of $\Gamma^{(\mu)}$. This sparsity is benign in the laminar regime, where pairs of a given key form a straight line, but it becomes severe on the curved manifolds of the transitional and turbulent branches. Second, it relies on a baseline solution that is rarely available in practice. For these reasons, we use this strategy solely as a reference construction. Every dataset analyzed in Sec.~\ref{sec:numerical_experiments} is built using the local strategy described below.

\subsubsection{Local Constitutive Sampling}\label{sec:local_constitutive_sampling}

In contrast, this strategy samples the closure directly, which resolves both limitations. It produces arbitrarily many points per curve in any flow regime and does not require a baseline solution. The constitutive curve of a pipe depends entirely on its closure and the parameters $(D, \varepsilon)$. It is independent of the pipe length, which only influences the discrete gradient operator. Each unique combination of parameters defines a distinct dataset key $\mu \in \mathcal{M}$. For example, in the $3 \times 3$ grid of Fig.~\ref{fig:local_constitutive_sampling}(a), the horizontal edges carry parameters $(D_0, \varepsilon_0)$ while the vertical ones carry $(D_1, \varepsilon_1)$. This results in $\mathcal{M} = \{\mu_0, \mu_1\}$ and the partition $\mathcal{E} = \mathcal{E}_{\mu_0} \sqcup \mathcal{E}_{\mu_1}$. Sampling each key separately yields the two datasets shown in Fig.~\ref{fig:local_constitutive_sampling}(b).

\begin{figure}[h!]
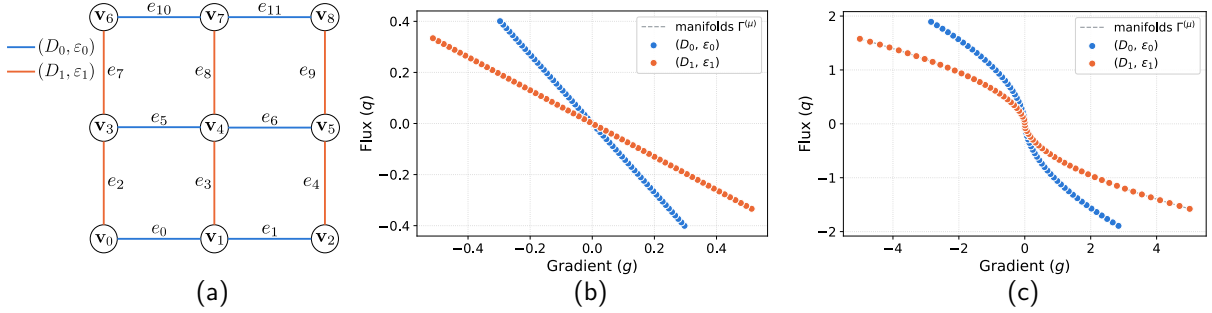

	\centering
	\svgfig{1.0\columnwidth}{548.02395487}{local_constitutive_sampling}
	\caption{(a) $3 \times 3$ Cartesian grid with horizontal $(D_0, \varepsilon_0)$ and vertical $(D_1, \varepsilon_1)$ pipe parameters. Dimensionless pressure-gradient--flux pairs on a uniform grid in $q$ over $[-q_{\max}, +q_{\max}]$, for the (b) laminar and (c) turbulent regimes; dashed curves are the constitutive manifolds $\Gamma^{(\mu)}$.}
	\label{fig:local_constitutive_sampling}
\end{figure}

\begin{sloppypar}
For each key, $\mathrm{Re}_{\max}$ dictates the maximum flux to be sampled, $q_{\max} = (\pi/4)\,(\mu_f/\rho)\,D\,\mathrm{Re}_{\max}$. We then select $\mathcal{N}_m^{(\mu)}$ values of $q$ evenly distributed across a uniform grid over $[-q_{\max}, +q_{\max}]$. Sampling both positive and negative values honors the odd symmetry of the \emph{Darcy--Weisbach} relation. Next, each $q$ is mapped to its gradient using the closure: $q$ yields $\mathrm{Re}$, $\mathrm{Re}$ yields $f$ (Eq.~\eqref{eq:churchill}), and $f$ yields $g$ (Eq.~\eqref{eq:darcy_weisbach_expanded}). Finally, the resulting pair is nondimensionalized by $(g_c, q_c)$. Algorithm~\ref{alg:data_generation_churchill} in Appendix~\ref{appendix:algorithms} outlines this entire procedure.
\end{sloppypar}

The value of $\mathrm{Re}_{\max}$ determines which flow regime the dataset represents. Below the transition threshold ($\mathrm{Re}_{\max} \leq 2300$), the constitutive curve is linear. As a result, the uniform sweep spaces the points evenly along the line, as shown in Fig.~\ref{fig:local_constitutive_sampling}(b). Higher Reynolds numbers introduce the transitional and turbulent branches, creating the S-shaped curve seen in Fig.~\ref{fig:local_constitutive_sampling}(c). Because this curve is steep near the origin, the uniform sweep in $q$ naturally concentrates points in that region. In all cases, $\mathrm{Re}_{\max}$ must adequately cover the target network's operating range to prevent the solver from extrapolating beyond the sampled data.

	\section{Solution Algorithms}\label{sec:solution_algorithms}

We now turn to the computational approaches for solving the \emph{data-driven} optimization Problem~\ref{prob:op}. To this end, three computational strategies are considered: the Brute Force approach, the Alternating Direction Method (ADM) with different initialization strategies, and the Deterministic Annealing (DA). These are discussed in detail below.

\begin{remark}
Throughout this section we adopt the single-dataset notation, which entails no loss of generality.
\end{remark}

\subsection{Brute Force}\label{sec:brute_force}

The first approach consists of solving the block-matrix linear system in Eq.~\eqref{eq:block_matrix} (which stems from the optimality conditions of the SOP) while exhaustively testing all possible data assignments. In this way, the approach amounts to solving all possible instances of the Sub-optimization Problem (SOP), rather than solving Problem~\ref{prob:op} at once, and subsequently selecting the one that yields the global minimum. Since all instances are at hand, those violating~\eqref{eq:clausius_duhem} are discarded, so that the minimum is taken over the thermodynamically admissible assignments alone. Three aspects of this strategy are worth highlighting.

First, for small graphs with a limited number of data pairs, this approach may be both effective and computationally efficient for identifying the global optimum. However, applying the same procedure to larger networks quickly becomes computationally infeasible due to its combinatorial nature.

Second, from Eq.~\eqref{eq:block_matrix}, one readily observes that the left-hand block matrix remains unchanged across all data assignments. Therefore, repeatedly solving a linear system for each updated right-hand side would introduce significant overhead. To mitigate this, the inverse of the block matrix is factored once and subsequently applied to all possible right-hand side vectors, thereby reducing the overall computational cost.

Third, to reduce the computational time, this work leverages GPU acceleration. Instead of applying the precomputed inverse to each right-hand side individually, the matrix-matrix multiplications are performed in \emph{batches}, with the batch size being primarily determined by the available GPU memory limits.

In this way, the Brute Force approach requires solving $N_{\mathrm{comb}}$ instances of the SOP. Its cost therefore grows combinatorially with either the number of edges or the dataset size, which restricts the method to small networks and moderate datasets; the detailed accounting is given in Sec.~\ref{sec:complexity}.

\subsection{Alternating Direction Method}\label{sec:adm}

To overcome the combinatorial complexity of the Brute Force approach, we consider an iterative method based on the concept of \emph{alternating directions}, which decouples the enforcement of physical constraints from the selection of constitutive data. This is the solution strategy originally proposed for data-driven computational mechanics in~\cite{KirchdoerferOrtiz2016}, here transposed to the hydraulic network setting.

To set the stage and simplify the notation, we introduce the combined dimensionless vectors of primal and data fields, $\mathbf{z} := (\eH,\sH)$ and $\mathbf{\tilde{z}} := (\etH, \stH)$, together with the $\mathbf{C}$-weighted squared distance induced by the proximity measure of Eq.~\eqref{eq:proximity_measure},
\begin{equation}
	\|\mathbf{z} - \mathbf{\tilde{z}}\|_{\mathbf{C}}^{2} := (\eH - \etH)^{\top}\mathbf{C}\,(\eH - \etH) + (\sH - \stH)^{\top}\mathbf{C}^{-1}(\sH - \stH),
	\label{eq:c_norm}
\end{equation}
so that $\mathcal{J} = \tfrac{1}{2}\|\mathbf{z} - \mathbf{\tilde{z}}\|_{\mathbf{C}}^{2}$. With this notation, Problem~\ref{prob:op} (OP) can be recast as a distance minimization problem between two sets:
\begin{equation}
	\min_{\mathbf{z} \in \mathcal{A},\, \mathbf{\tilde{z}} \in \mathcal{Z}}
	\|\mathbf{z} - \mathbf{\tilde{z}}\|_{\mathbf{C}}^{2}.
	\label{eq:adm_op}
\end{equation}

Here, we define the physically admissible set $\mathcal{A}$ and the constitutive-data set $\mathcal{Z}$ as follows.
\begin{definition}[Physically admissible set]
\[
	\mathcal{A} := \Big\{(\eH,\sH) \in \mathbb{R}^{n_e} \times \mathbb{R}^{n_e} \;\big|\; \exists \uH \in \mathbb{R}^{n_v} \text{ s.t. } \mathbf{D}\,\sH = \mathbf{f}, \; \eH = \mathbf{L}^{-1}\mathbf{G}\,\uH, \;\mathbf{B}\,\uH = \uH_D \Big\}.
	\]
\end{definition}

\begin{definition}[Constitutive-data set]
\begin{multline*}
		\mathcal{Z} :=  \Big\{(\etH,\stH) \in \mathbb{R}^{n_e}\times\mathbb{R}^{n_e} \;\big|\; \exists \{\chi_n\}_{n=1}^{\mathcal{N}_m} \in \mathcal{S} \\
		\text{ s.t. } (\etH_k,\stH_k) = \sum_{n=1}^{\mathcal{N}_m} \chi_n(e_k)\,(g_n,q_n), \;\text{for } k = 1, \dots, n_e \Big\}.
	\end{multline*}
\end{definition}

Being the solution set of linear equality constraints, $\mathcal{A}$ is an affine subspace of $\mathbb{R}^{n_e} \times \mathbb{R}^{n_e}$ (nonempty whenever the problem is feasible) and therefore convex and closed. By contrast, $\mathcal{Z}$ is discrete and non-convex. Given these sets, we now define the following projection operators:
\begin{equation}
	P_{\mathcal{A}}(\mathbf{\tilde{z}}) :=
	\arg\min_{\mathbf{z}\in\mathcal{A}} \|\mathbf{z} - \mathbf{\tilde{z}}\|_{\mathbf{C}}^2,
	\qquad
	P_{\mathcal{Z}}(\mathbf{z}) :=
	\arg\min_{\mathbf{\tilde{z}}\in\mathcal{Z}} \|\mathbf{\tilde{z}} - \mathbf{z}\|_{\mathbf{C}}^2.
\end{equation}

In this case, projecting onto $\mathcal{A}$ corresponds to solving the Sub-optimization Problem~\ref{prob:sop} (SOP). This requires solving the associated optimality conditions, which take the form of the linear saddle-point system in Eq.~\eqref{eq:block_matrix}. On the other hand, the projection onto the data set $\mathcal{Z}$ admits a fully decoupled structure. This decoupling is possible due to the separability of both the $\mathbf{C}$-weighted norm (where $\mathbf{C}$ is diagonal) and the definition of $\mathcal{Z}$, allowing the projection to be carried out independently on each edge. Therefore, the projection is obtained by solving, for each $k = 1, \dots, n_e$,
\begin{equation}
	(\etH_k,\stH_k)
	=
	\arg\min_{(g_n,q_n)\in\mathcal{D}}
	\Big[\, C_k\,(g_n - \eH_k)^2 + C_k^{-1}(q_n - \sH_k)^2 \,\Big].
\end{equation}

In this way, $P_{\mathcal{Z}}$ reduces to a collection of independent nearest-neighbor searches in the dataset $\mathcal{D}$, one for each edge. In practice, these queries can be efficiently performed using spatial data structures such as KD-trees. Thus, we can state the alternating direction algorithm in its most fundamental form as follows:

\begin{algorithm}[H]
	\caption{Alternating Direction Method (ADM)} 
	\label{alg:alt_proj}
	\begin{algorithmic}[1]
		\Require Initial data assignment $\mathbf{\tilde{z}}_0 \in \mathcal{Z}$, maximum iterations $J_{\max}$
		\Ensure Best physically admissible state $\mathbf{z}^{\star}$
		\State $\mathcal{J}^{\star} \gets +\infty$
		\For{$j = 0, 1, 2, \dots, J_{\max}-1$}
		\State $\mathbf{z}_{j+1} \gets P_{\mathcal{A}}(\mathbf{\tilde{z}}_j)$ \Comment{Physical projection (KKT solve, Eq.~\eqref{eq:block_matrix})}
		\State $\mathcal{J}_j \gets \tfrac{1}{2}\big\|\mathbf{z}_{j+1} - P_{\mathcal{Z}}(\mathbf{z}_{j+1})\big\|_{\mathbf{C}}^2$ \Comment{Hard distance objective for monitoring}
		\If{$\mathbf{z}_{j+1}$ satisfies~\eqref{eq:clausius_duhem} \textbf{and} $\mathcal{J}_j < \mathcal{J}^{\star}$} \Comment{Check thermodynamic admissibility}
		\State $\mathbf{z}^{\star} \gets \mathbf{z}_{j+1}, \quad \mathcal{J}^{\star} \gets \mathcal{J}_j$
		\EndIf
		\State $\mathbf{\tilde{z}}_{j+1} \gets P_{\mathcal{Z}}(\mathbf{z}_{j+1})$ \Comment{Hard nearest-neighbor data projection}
		\EndFor
	\end{algorithmic}
\end{algorithm}
Each iteration of Algorithm~\ref{alg:alt_proj} consists of one linear solve followed by $n_e$ independent nearest-neighbor queries over $\mathcal{D}$. Both are cheap to repeat: only the right-hand side of the block system changes between iterations, so the factorization is computed once and reused, and $\mathcal{D}$ being fixed, the KD-tree is likewise built only once. The resulting cost is far below that of the Brute Force for all but the smallest networks (cf.~Sec.~\ref{sec:complexity}).

This speed comes at the price of weaker optimality guarantees. Each projection minimizes the same distance over one of its two arguments, so the sequence $\|\mathbf{z}_j - \mathbf{\tilde{z}}_j\|_{\mathbf{C}}$ never increases; and since $\mathcal{Z}$ is finite, the assignment $\{\chi_n\}$ takes only finitely many values, so the iteration reaches a fixed point of $P_{\mathcal{Z}} \circ P_{\mathcal{A}}$ after finitely many steps. The only exception is a tie in the nearest-neighbor search, which may cause cycling between states of equal objective. That fixed point, however, is only a \emph{partial} minimum: it is optimal in $\mathbf{z}$ for the current assignment, and optimal in $\mathbf{\tilde{z}}$ for the current state, but not necessarily optimal in both at once. Because $\mathcal{Z}$ is discrete and non-convex, it may lie far from the global optimum. Escaping this local minimum would require reassigning several edges simultaneously, which the alternating structure cannot do. The ADM is therefore a local method applied to a discrete problem, and the fixed point it reaches depends strongly on the initialization $\mathbf{\tilde{z}}_0$. To reduce this dependence, we consider three initialization strategies, detailed below: Constitutive Manifold Reconstruction (CMR), Null-space, and Random.

\subsubsection{Constitutive Manifold Reconstruction (CMR) Initialization}\label{sec:cmr}

This initialization mimics the procedure one would follow in a real application: the dataset is approximated by a fitting model, yielding a reconstructed representation of the constitutive manifold. Two strategies are considered: (a) a linear least-squares fit and (b) a projection onto Hermite polynomials.

\paragraph{(a) Linear Least-Squares Fit.}\label{sec:cmr_linear_ls}
The dataset is approximated by a single linear constitutive equation. To this end, let $\mathbf{g}_{\mathcal{D}}, \mathbf{q}_{\mathcal{D}} \in \mathbb{R}^{\mathcal{N}_m}$ collect the pressure-gradient and flux components of the $\mathcal{N}_m$ data pairs in $\mathcal{D}$,
\begin{equation}
	\mathbf{g}_{\mathcal{D}} := [g_1, \dots, g_{\mathcal{N}_m}]^{\top}, \qquad \mathbf{q}_{\mathcal{D}} := [q_1, \dots, q_{\mathcal{N}_m}]^{\top},
\end{equation}
so that a candidate conductance $k$ predicts the fluxes $-k\,\mathbf{g}_{\mathcal{D}}$. We seek the single value that best represents the constitutive equation underlying the dataset.

\begin{tcolorbox}[colback=black!3!white, colframe=black!45!white, halign=justify, boxrule=0.6pt, title = Least-Squares Fitting]
\begin{problem}
Find the effective conductance $k_{\text{fit}} \in \mathbb{R}$ such that
\begin{equation}
			k_{\text{fit}} = \arg\min_{k \in \mathbb{R}}\,\mathcal{J}_{\mathrm{LS}}(k), 
		\end{equation}
where the functional $\mathcal{J}_{\mathrm{LS}}: \mathbb{R} \to \mathbb{R}$ is defined as:
\begin{equation*}
			\mathcal{J}_{\mathrm{LS}}(k) := \|k\,\mathbf{g}_{\mathcal{D}} + \mathbf{q}_{\mathcal{D}}\|_2^2.
		\end{equation*}
\end{problem}
\end{tcolorbox}

Whenever $\mathbf{g}_{\mathcal{D}} \neq \mathbf{0}$, $\mathcal{J}_{\mathrm{LS}}$ is a strictly convex quadratic and therefore has a unique global minimizer, characterized by the stationarity condition $\mathcal{J}_{\mathrm{LS}}'(k) = 0$ (the \emph{normal system}):
\begin{equation}
	(\mathbf{g}_{\mathcal{D}}^{\top}\mathbf{g}_{\mathcal{D}})\, k_{\text{fit}} = -\mathbf{g}_{\mathcal{D}}^{\top}\mathbf{q}_{\mathcal{D}},
	\label{eq:normal_system}
\end{equation}
whose solution (Appendix~\ref{appendix:pseudoinverse}) reduces in this scalar case to
\begin{equation}
	k_{\text{fit}} = - \frac{\mathbf{g}_{\mathcal{D}}^{\top}\mathbf{q}_{\mathcal{D}}}{\|\mathbf{g}_{\mathcal{D}}\|_2^2}.
	\label{eq:k_fit_ls}
\end{equation}

The fit is dissipative, $k_{\text{fit}} > 0$, whenever $\mathbf{g}_{\mathcal{D}}^{\top}\mathbf{q}_{\mathcal{D}} < 0$; this holds by construction for noise-free datasets, but is no longer guaranteed once noise is injected. In addition, the fitted value defines a linear constitutive equation in the form of Eq.~\eqref{eq:const_eq_matrix}, with uniform conductance $\mathbf{K} = k_{\text{fit}}\,\mathbf{I}$. Solving Eq.~\eqref{eq:lin_sys_class_sol} for the nodal pressures $\uH_0$ recovers the edge-wise fields $\eH_0 = \mathbf{L}^{-1}\mathbf{G}\uH_0$ and $\sH_0 = -k_{\text{fit}}\eH_0$, so that the state $\mathbf{z}_0 = (\eH_0, \sH_0)$ is thermodynamically admissible edge by edge, $(\eH_0)_k\,(\sH_0)_k = -k_{\text{fit}}\,(\eH_0)_k^{2} \leq 0$ for every $k$, whenever the fit is dissipative. The initial data assignment $\mathbf{\tilde{z}}_0 = P_{\mathcal{Z}}(\mathbf{z}_0)$ inherits this property when the data pairs are themselves dissipative.

\paragraph{(b) Piecewise-Cubic Hermite Least-Squares Projection.}\label{sec:cmr_hermite_projection}

For datasets arising from nonlinear constitutive behavior -- such as those generated by turbulent flow models -- a single effective conductance fails to capture the geometry of the constitutive manifold. The second variant instead constructs a globally $C^1$-continuous approximation $q = \hat{q}_{h}(g)$ directly from $\mathcal{D}$.

To this end, the gradient domain is partitioned into a one-dimensional finite element mesh constructed symmetrically about the origin, with a mesh node placed exactly at $g = 0$ to preserve thermodynamic consistency. Over each element, the local constitutive relation is approximated using cubic Hermite polynomials, expressed in vector form as
\begin{equation}
	\hat{q}_h^{(E)}(g) = \mathbf{N}(\xi)\,\mathbf{c}_E,
\end{equation}
where $\xi \in [0,1]$ is the local element coordinate, $\mathbf{N}(\xi) \in \mathbb{R}^{1 \times 4}$ contains the four Hermite shape functions, and $\mathbf{c}_E \in \mathbb{R}^4$ collects the two nodal degrees of freedom at each endpoint -- i.e., the flux value and its derivative with respect to $g$.

The global degrees of freedom vector $\mathbf{c}$ is determined by solving the normal equations of the $L_2$-projection of $\mathcal{D}$ onto the finite element space:
\begin{equation}
	\left(\mathbf{P}^{\top}\mathbf{P} \right)\mathbf{c} = \mathbf{P}^{\top}\mathbf{q}_{\mathcal{D}},
	\label{eq:hermite_normal_system}
\end{equation}
where $\mathbf{P}\in \mathbb{R}^{\mathcal{N}_m \times 2n_{\mathrm{nodes}}}$ is the global observation matrix assembling the Hermite shape functions evaluated at the data points. Provided the data populate every element of the mesh, $\mathbf{P}$ has full column rank and the projection is again well posed, with unique solution $\mathbf{c} = \mathbf{P}^{\dagger}\mathbf{q}_{\mathcal{D}}$ (Appendix~\ref{appendix:pseudoinverse}). The null-flux constraint $\hat{q}_h(0) = 0$ is then strongly enforced by modifying the system at the nodal flux degree of freedom corresponding to the origin. Note that the derivative degree of freedom at the origin remains free, allowing the approximated model to correctly represent a constitutive equation with nonzero slope at $g = 0$.

With $\hat{q}_h$ established, the initial state $(\uH_0, \eH_0, \sH_0)$ is computed by solving the nonlinear hydraulic problem via a Picard iteration (Algorithm~\ref{alg:picard_solver}). At each step $j$, the effective conductance on each edge is updated via the secant approximation,
\begin{equation}
	K_{kk}^{(j)} = -\frac{\hat{q}_h\!\left(\eH_k^{(j)}\right)}{\eH_k^{(j)}}, \qquad k = 1, \dots, n_e,
\end{equation}
and the resulting pressure system is solved until convergence. The $C^1$-continuity of $\hat{q}_h$, together with the imposed null-flux condition $\hat{q}_h(0)=0$, is essential here: it ensures that the secant conductance remains bounded as $\eH_k^{(j)} \to 0$, with limit $-\hat{q}_h'(0)$. Provided the surrogate is strictly dissipative at the origin ($-\hat{q}_h'(0) > 0$), as expected from thermodynamic consistency, each $K_{kk}^{(j)}$ is strictly positive, so the pressure system remains symmetric positive-definite throughout the iteration (cf.~Eq.~\eqref{eq:lin_sys_class_sol} and Eq.~\eqref{eq:laplacian_quadratic_form}). Upon convergence, the initial data assignment is finalized via $\mathbf{\tilde{z}}_0 = P_{\mathcal{Z}}(\mathbf{z}_0)$, with $\mathbf{z}_0 = (\eH_0,\sH_0)$, providing a physically admissible and thermodynamically consistent starting point for the ADM.

The Picard scheme is preferred over Newton--Raphson because it requires only the secant conductance $-\hat{q}_h(g)/g$, which is directly available from the surrogate. By contrast, Newton's method would additionally require $\mathrm{d}\hat{q}_h/\mathrm{d}g$, which is a quantity more sensitive to noise and potentially ill-behaved near the origin.

\subsubsection{Null-space Initialization}\label{sec:null_space_init}
This strategy computes an initial state $(\uH_0, \eH_0, \sH_0)$ that strictly satisfies the physical constraints, without imposing any constitutive assumption. The Dirichlet data are built in through the same masking device as in Eq.~\eqref{eq:block_matrix}: the diagonal mask $\mathbf{I}_D$ and the row-zeroed incidence operator $\widehat{\mathbf{D}}$ replace the mass-balance row of every constrained node by its pressure constraint. This yields the under-determined linear system
\begin{equation}
	\underbrace{
		\begin{bmatrix}
			\mathbf{I}_D & \mathbf{0} & \widehat{\mathbf{D}} \\
			-\mathbf{L}^{-1}\mathbf{G} & \mathbf{I} & \mathbf{0}
	\end{bmatrix}}_{=:\,\mathbf{M}}
	\underbrace{
		\begin{bmatrix}
			\uH \\
			\eH \\
			\sH
	\end{bmatrix}}_{=:\,\mathbf{x}} =
	\underbrace{
		\begin{bmatrix}
			\hat{\mathbf{f}} \\
			\mathbf{0}
	\end{bmatrix}}_{=:\,\mathbf{r}}
	\quad \implies \quad \mathbf{M}\mathbf{x} = \mathbf{r},
	\label{eq:under_determined_syst}
\end{equation}
with $\hat{\mathbf{f}}$ carrying the prescribed pressures at the constrained nodes and the imposed nodal fluxes elsewhere. Every solution therefore satisfies mass balance $\mathbf{D}\,\sH = \mathbf{f}$ at the free nodes, the boundary conditions $\mathbf{B}\,\uH = \uH_D$ on $\mathcal{V}_D$, and compatibility $\eH = \mathbf{L}^{-1}\mathbf{G}\,\uH$ on every edge.

Let $m$ and $n$ denote the number of rows and columns of $\mathbf{M}$, respectively. With $m < n$ and $\mathrm{rank}(\mathbf{M}) = m$, the solutions of~\eqref{eq:under_determined_syst} form an affine subspace of dimension $n - m$; among them we select the one of minimum Euclidean norm, which carries no arbitrary component from $\ker(\mathbf{M})$.

\begin{tcolorbox}[colback=black!3!white, colframe=black!45!white, halign=justify, boxrule=0.6pt, title = Minimum-Norm Solution]
\begin{problem}
Find the primal fields $\mathbf{x}_0 \in \mathbb{R}^n$ such that
\begin{equation}
			\mathbf{x}_0 = \arg\min_{\mathbf{x} \in \mathbb{R}^n} \mathcal{J}_{\mathrm{MN}}(\mathbf{x}) , \quad \text{s.t. }\mathbf{M}\mathbf{x} = \mathbf{r},
		\end{equation}
where the functional $\mathcal{J}_{\mathrm{MN}} : \mathbb{R}^n \to \mathbb{R}$ is defined as
\begin{equation*}
			\mathcal{J}_{\mathrm{MN}}(\mathbf{x})  := \frac{1}{2}\|\mathbf{x}\|_2^2.
		\end{equation*}
\end{problem}
\end{tcolorbox}

Since $\mathbf{M}$ has full row rank, this solution exists and is unique, and is given by the right Moore--Penrose pseudo-inverse (Appendix~\ref{appendix:pseudoinverse})
\begin{equation}
	\mathbf{x}_0 = \mathbf{M}^{\dagger}\mathbf{r} = \mathbf{M}^{\top}(\mathbf{M}\mathbf{M}^{\top})^{-1}\mathbf{r}.
\end{equation}

Once the initial fields $\mathbf{x}_0 = (\uH_0, \eH_0, \sH_0)$ are computed, the initial data assignment follows as $\mathbf{\tilde{z}}_0 = P_{\mathcal{Z}}(\mathbf{z}_0)$, with $\mathbf{z}_0 = (\eH_0, \sH_0)$. Carrying no constitutive information, this state is physically admissible but not thermodynamically consistent by construction.

\subsubsection{Random Initialization}
The random initialization is the simplest strategy considered. It assigns an arbitrary data pair $(g_n, q_n) \in \mathcal{D}$ independently to each edge, without any physical or constitutive guidance. As a consequence, the initial state $\mathbf{\tilde{z}}_0$ is unlikely to be physically admissible, and the ADM is exposed to the full non-convexity of $\mathcal{Z}$ from the very first iteration.

\subsection{Deterministic Annealing}\label{sec:da}

The initializations above steer the ADM toward better solutions, but none removes its fundamental limitation: the hard projection $P_{\mathcal{Z}}$ commits to a single data pair per edge at every iteration, so the method may lock onto a suboptimal assignment $\{\chi_n\}$ with no mechanism to revise it. We therefore relax the combinatorial assignment into a continuous one, governed by an inverse temperature $\beta > 0$: the physical projection $P_{\mathcal{A}}$ is left unchanged, while $P_{\mathcal{Z}}$ gives way to a \emph{soft} counterpart $P_{\mathcal{Z}}^{\beta}$. The inverse temperature $\beta$ is raised gradually, carrying the iterates from a unique, initialization-independent basin at low $\beta$ to the original non-convex problem at high $\beta$.

This is the \emph{Deterministic Annealing} (DA), proposed in the statistical-mechanics treatment of clustering~\cite{RoseGurewitzFox1990, Rose1998} and brought into data-driven computational mechanics by Kirchdoerfer and Ortiz~\cite{Kirchdoerfer2017a} under the name of \emph{simulated} annealing. We keep the former designation, which is standard in the clustering literature. The temperature enters through the maximum-entropy weights rather than through random trial moves, ensuring the iteration is deterministic and its outcome reproducible. Its construction and properties, transposed here to the hydraulic network setting, are detailed below.

Recall that each edge is assigned a single data pair through the binary selectors $\chi_n(e_k) \in \{0,1\}$ subject to $\sum_{n=1}^{\mathcal{N}_m} \chi_n(e_k) = 1$ (the admissible set $\mathcal{S}$). The DA relaxes these indicators into soft membership weights $w_n(e_k) \in [0,1]$ living on the probability simplex, i.e., we replace $\mathcal{S}$ by its convex hull
\begin{equation}
	\overline{\mathcal{S}} :=
	\left\{ \{w_n\}_{n=1}^{\mathcal{N}_m} \;\middle|\;
	w_n(e_k) \geq 0, \;
	\sum_{n=1}^{\mathcal{N}_m} w_n(e_k) = 1, \;\text{for } k = 1,\dots,n_e
	\right\}.
\end{equation}
Accordingly, the edge-wise data fields become convex combinations of the measured pairs,
\begin{equation}
	(\etH_k, \stH_k) = \sum_{n=1}^{\mathcal{N}_m} w_n(e_k)\,(g_n,q_n),
	\qquad k = 1,\dots,n_e,
	\label{eq:da_soft_fields}
\end{equation}
which range over the per-edge convex hull of $\mathcal{D}$,
\begin{equation}
	\overline{\mathcal{Z}} :=
	\Big\{(\etH,\stH) \;\big|\;
	\exists\, \{w_n\} \in \overline{\mathcal{S}} \text{ s.t. }
	(\etH_k,\stH_k) = \textstyle\sum_{n} w_n(e_k)(g_n,q_n) \Big\}
	= \operatorname{conv}(\mathcal{Z}).
\end{equation}
The vertices of $\overline{\mathcal{Z}}$, recovered when each $w_n(e_k) \in \{0,1\}$, are exactly the points of $\mathcal{Z}$; for any $0 < w_n(e_k) < 1$, the target in Eq.~\eqref{eq:da_soft_fields} is not an actual measurement, but a weighted average of several data pairs.

Given a physically admissible state $\mathbf{z} = (\eH,\sH) \in \mathcal{A}$, the hard projection $P_{\mathcal{Z}}$ selects, on each edge, the single nearest data pair. The DA instead assigns weights through a \emph{maximum-entropy} principle~\cite{Rose1998, Kirchdoerfer2017a}: among all soft assignments achieving a prescribed mean proximity to the data, the least committal (highest-entropy) one is chosen. Introducing the per-edge $\mathbf{C}$-weighted squared distances
\begin{equation}
	d_{k,n}^2 := C_k\,(\eH_k - g_n)^2 + C_k^{-1}(\sH_k - q_n)^2,
\end{equation}
(which carry no factor $\tfrac{1}{2}$, so that $\mathcal{J} = \tfrac{1}{2}\sum_{k}\sum_{n}\chi_n(e_k)\,d_{k,n}^2$ at a hard assignment, and the inverse temperature below absorbs a factor two with respect to the convention of \cite{Kirchdoerfer2017a}), and the Shannon entropy of the edge weight vector $w_k := \big(w_1(e_k), \dots, w_{\mathcal{N}_m}(e_k)\big)$,
\begin{equation}
	\mathcal{H}(w_k) := -\sum_{n=1}^{\mathcal{N}_m} w_n(e_k)\,\ln w_n(e_k),
\end{equation}
the soft projection is defined edge-wise as the minimizer of a \emph{free-energy} functional balancing proximity against entropy.

\begin{tcolorbox}[colback=black!3!white, colframe=black!45!white, halign=justify, boxrule=0.6pt, title = Soft Projection $P_{\mathcal{Z}}^{\beta}$ (per edge)]
\begin{problem}
Given $(\eH_k,\sH_k)$ and an inverse temperature $\beta > 0$, find the weights $w_k \in \Delta_{\mathcal{N}_m}$ (the probability simplex) such that
\begin{equation}
			w_k = \arg\min_{w_k \in \Delta_{\mathcal{N}_m}} \mathcal{F}_k^{\beta}(w_k),
		\end{equation}
where the free-energy functional $\mathcal{F}_k^{\beta} : \Delta_{\mathcal{N}_m} \to \mathbb{R}$ is
\begin{equation*}
			\mathcal{F}_k^{\beta}(w_k)
			:= \underbrace{\sum_{n=1}^{\mathcal{N}_m} w_n(e_k)\, d_{k,n}^2}_{\text{mean proximity}}
			\;-\; \frac{1}{\beta}\,\underbrace{\mathcal{H}(w_k)}_{\text{entropy}}.
		\end{equation*}
\end{problem}
\end{tcolorbox}

Since $\mathcal{H}(w_k) = -\sum_n w_n(e_k)\ln w_n(e_k)$, the free energy reads
\begin{equation}
	\mathcal{F}_k^{\beta}(w_k) = \sum_{n=1}^{\mathcal{N}_m} w_n(e_k)\,d_{k,n}^2
	+ \frac{1}{\beta}\sum_{n=1}^{\mathcal{N}_m} w_n(e_k)\,\ln w_n(e_k),
	\label{eq:da_free_energy_expanded}
\end{equation}
whose first term is linear in $w_k$, while the second (the negative entropy) has a diagonal Hessian with entries $1/\big(\beta\,w_n(e_k)\big) > 0$. Hence $\mathcal{F}_k^{\beta}$ is strictly convex for every $\beta > 0$ and, $\Delta_{\mathcal{N}_m}$ being compact and convex, the minimizer exists, is unique, and is characterized by stationarity alone. It therefore suffices to enforce the equality constraint, through the Lagrangian
\begin{equation}
	\mathcal{L}(w_k,\lambda) := \mathcal{F}_k^{\beta}(w_k)
	+ \lambda\Big(\textstyle\sum_{n=1}^{\mathcal{N}_m} w_n(e_k) - 1\Big).
\end{equation}
Stationarity with respect to each weight gives
\begin{equation}
	\frac{\partial \mathcal{L}}{\partial w_n(e_k)}
	= d_{k,n}^2 + \frac{1}{\beta}\big[\ln w_n(e_k) + 1\big] + \lambda = 0
	\quad\Longrightarrow\quad
	w_n(e_k) = \exp\!\big(-\beta\, d_{k,n}^2\big)\,\exp\!\big(-\beta\lambda - 1\big),
\end{equation}
where the second factor is independent of $n$ and is thus fixed by the constraint itself,
\begin{equation}
	\sum_{n=1}^{\mathcal{N}_m} w_n(e_k) = 1
	\quad\Longrightarrow\quad
	\exp(-\beta\lambda - 1) = \frac{1}{Z_k^{\beta}},
	\qquad
	Z_k^{\beta} := \sum_{m=1}^{\mathcal{N}_m} \exp\!\big(-\beta\,d_{k,m}^2\big),
\end{equation}
which yields the \emph{Gibbs (softmax) distribution}
\begin{equation}
	w_n(e_k) =
	\frac{\exp\!\big(-\beta\, d_{k,n}^2\big)}
	{\displaystyle\sum_{m=1}^{\mathcal{N}_m} \exp\!\big(-\beta\, d_{k,m}^2\big)}.
	\label{eq:da_gibbs_weights}
\end{equation}
The exponential of the entropy, $N_{\mathrm{eff}} := \exp\mathcal{H}(w_k)$, counts the data points effectively active on edge $e_k$: it equals $\mathcal{N}_m$ for uniform weights and $1$ once the assignment has hardened.

The operator $P_{\mathcal{Z}}^{\beta}$ returns the soft data fields of Eq.~\eqref{eq:da_soft_fields}, assembled from these weights. The inverse temperature $\beta$ continuously deforms the soft projection between two extremes. At high temperatures ($\beta \to 0^{+}$), the weights become uniform and the data fields collapse onto the centroid of the dataset. This ensures that the composite map $P_{\mathcal{Z}}^{\beta} \circ P_{\mathcal{A}}$ is smooth and possesses an essentially unique attraction basin, thereby yielding an initialization-independent algorithm. At low temperatures ($\beta \to \infty$), on the other hand, the hard projection is recovered. These two limiting properties are established by the following proposition.

\begin{proposition}[Temperature limits]\label{prop:dda_limits}
Let $(\eH_k,\sH_k)$ be fixed and let the weights be given by the Gibbs distribution~\eqref{eq:da_gibbs_weights}. Then:
\begin{enumerate}[label=(\roman*)]
\item As $\beta \to 0^{+}$, $w_n(e_k) \to 1/\mathcal{N}_m$ for all $n$, so the soft target converges to the \emph{data centroid}
\[
		(\etH_k,\stH_k) \to
		\frac{1}{\mathcal{N}_m}\sum_{n=1}^{\mathcal{N}_m}(g_n,q_n),
		\]
which is independent of $(\eH_k,\sH_k)$ and hence of the initialization.
\item If the nearest neighbor $n^{\ast} := \arg\min_n d_{k,n}^2$ is unique, then as $\beta \to \infty$, $w_{n^{\ast}}(e_k) \to 1$ and $w_n(e_k) \to 0$ for $n \neq n^{\ast}$, so
\[
		P_{\mathcal{Z}}^{\beta}(\mathbf{z}) \to P_{\mathcal{Z}}(\mathbf{z}),
		\]
i.e., the soft projection recovers the hard nearest-neighbor projection of the ADM. Moreover, writing $d_{(1)}^2 \le d_{(2)}^2$ for the two smallest distances on edge $e_k$, the approach to the hard limit is governed by their gap,
\[
		1 - w_{n^{\ast}}(e_k) \;=\; \mathcal{O}\!\Big(\exp\!\big[-\beta\,\big(d_{(2)}^2 - d_{(1)}^2\big)\big]\Big),
		\]
so that the terminal temperature required to resolve the assignment grows as the two closest candidates come together.
\end{enumerate}
\end{proposition}

\begin{proof}
(i) As $\beta \to 0^{+}$, $\exp(-\beta d_{k,n}^2) \to 1$ for every $n$, so each weight tends to $1/\mathcal{N}_m$ and Eq.~\eqref{eq:da_soft_fields} reduces to the arithmetic mean of $\mathcal{D}$.
(ii) To show that $w_{n^{\ast}}(e_k) \to 1$, we divide both the numerator and the denominator of Eq.~\eqref{eq:da_gibbs_weights} by the dominant term $\exp(-\beta d_{k,n^{\ast}}^2)$, yielding:
$$
	w_{n^{\ast}}(e_k) = \frac{1}{1 + \sum_{m \neq n^{\ast}} \exp\!\Big(-\beta \big(d_{k,m}^2 - d_{k,n^{\ast}}^2\big)\Big)}.
	$$
Since $n^{\ast}$ is the unique minimum, the difference $d_{k,m}^2 - d_{k,n^{\ast}}^2 > 0$ for all $m \neq n^{\ast}$. As $\beta \to \infty$, the exponent $-\beta(d_{k,m}^2 - d_{k,n^{\ast}}^2) \to -\infty$, so every term in the sum vanishes. The same expression yields the stated rate: the sum is dominated by its largest term, namely the one built on the runner-up distance $d_{(2)}^2$, whence $1 - w_{n^{\ast}}(e_k) = \mathcal{O}\big(\exp[-\beta(d_{(2)}^2 - d_{(1)}^2)]\big)$ with $d_{(1)}^2 = d_{k,n^{\ast}}^2$. Thus, $w_{n^{\ast}}(e_k) \to 1$. Because the weights sum to unity, it follows immediately that $w_n(e_k) \to 0$ for all $n \neq n^{\ast}$. Substituting these limits into Eq.~\eqref{eq:da_soft_fields} gives $(\etH_k,\stH_k) \to (g_{n^{\ast}}, q_{n^{\ast}})$, which is precisely the $k$-th component of $P_{\mathcal{Z}}(\mathbf{z})$.
\end{proof}

In summary, annealing $\beta$ from small to large realizes a \emph{homotopy} (graduated non-convexity) that tracks the global minimizer from the convex surrogate into the original non-convex problem. Finally, for completeness, we discuss the algorithmic aspects of DA.

The inverse temperature $\beta$ is increased geometrically across iterations according to
\begin{equation}
	\beta_{j+1} = \min\!\big(\gamma\,\beta_j,\; \beta_{\max}\big),
	\qquad \beta_0 \text{ given},\;\; \gamma > 1,
	\label{eq:dda_schedule}
\end{equation}
where $\gamma$ is the cooling ratio and $\beta_{\max}$ is a final value large enough that $P_{\mathcal{Z}}^{\beta_{\max}}$ is, for all practical purposes, indistinguishable from $P_{\mathcal{Z}}$. Given the values of $\beta_0$ and $\beta_{\max}$, together with the maximum number of iterations $J_{\max}$, the geometric ratio is computed as
\begin{equation}
	\gamma = \left(\frac{\beta_{\max}}{\beta_0}\right)^{1/J_{\max}}.
	\label{eq:da_gamma_auto}
\end{equation}

It is worth mentioning that the cooling rate plays an important role in the algorithm, since it governs the method's ability to reach the global optimum. Cooling too fast sharpens the assignment before the trajectory settles into the global basin. This risks the same trapping as the ADM. With these definitions at hand, the DA proceeds exactly as the ADM (Algorithm~\ref{alg:alt_proj}), substituting $P_{\mathcal{Z}}^{\beta_j}$ for $P_{\mathcal{Z}}$ and cooling the temperature after each step.

\begin{algorithm}[H]
	\caption{Deterministic Annealing (DA)}
	\label{alg:dda}
	\begin{algorithmic}[1]
		\Require Schedule $(\beta_0, \beta_{\max}, J_{\max})$; cooling ratio $\gamma$ from Eq.~\eqref{eq:da_gamma_auto}
		\Ensure Best physically admissible state $\mathbf{z}^{\star}$
		\State $\mathbf{\tilde{z}}_0 \gets \dfrac{1}{\mathcal{N}_m}\sum_{n=1}^{\mathcal{N}_m}(g_n,q_n)$ on every edge
		\Comment{centroid: $\beta \to 0$ limit, initialization-independent}
		\State $\beta \gets \beta_0$, \quad $\mathcal{J}^{\star} \gets +\infty$
		\For{$j = 0,1,2,\dots,J_{\max}-1$}
		\State $\mathbf{z}_{j+1} \gets P_{\mathcal{A}}(\mathbf{\tilde{z}}_j)$
		\Comment{physical projection (KKT solve, Eq.~\eqref{eq:block_matrix})}
		\State $\mathcal{J}_j \gets \tfrac{1}{2}\big\|\mathbf{z}_{j+1} - P_{\mathcal{Z}}(\mathbf{z}_{j+1})\big\|_{\mathbf{C}}^2$
		\Comment{\emph{hard} objective, for monitoring/comparability}
		\If{$\mathbf{z}_{j+1}$ satisfies~\eqref{eq:clausius_duhem} \textbf{and} $\mathcal{J}_j < \mathcal{J}^{\star}$}
		\State $\mathbf{z}^{\star} \gets \mathbf{z}_{j+1}$, \quad $\mathcal{J}^{\star} \gets \mathcal{J}_j$
		\EndIf
		\State $\mathbf{\tilde{z}}_{j+1} \gets P_{\mathcal{Z}}^{\beta}(\mathbf{z}_{j+1})$
		\Comment{soft projection, Eqs.~\eqref{eq:da_soft_fields}--\eqref{eq:da_gibbs_weights}}
		\State $\beta \gets \min(\gamma\,\beta,\;\beta_{\max})$
		\Comment{cool down}
		\EndFor
	\end{algorithmic}
\end{algorithm}

Regarding the implementation of Algorithm~\ref{alg:dda}, it is worth highlighting that although the iterations are driven by the soft targets $\mathbf{\tilde{z}}_j \in \overline{\mathcal{Z}}$, the monitored objective is the hard distance $\tfrac{1}{2}\|\mathbf{z}_{j+1} - P_{\mathcal{Z}}(\mathbf{z}_{j+1})\|_{\mathbf{C}}^2$ to the nearest actual data pair. This keeps the objective function on the same footing as the Brute Force and ADM objectives, which measure proximity to genuine measurements in $\mathcal{D}$ rather than to convex combinations in $\overline{\mathcal{Z}}$.

Lastly, each DA iteration performs one physical projection $P_{\mathcal{A}}$, exactly as in the ADM, followed by the soft projection. The latter is the only structural difference between the two methods: since every datum carries a nonzero weight, evaluating Eq.~\eqref{eq:da_gibbs_weights} requires all $\mathcal{N}_m$ distances per edge, whereas $P_{\mathcal{Z}}$ retrieves a single nearest neighbor from the KD-tree. The per-iteration cost of the two methods nevertheless remains comparable. What sets them apart is the iteration budget. The DA needs a sufficiently slow cooling schedule to track the global optimum. Both points are quantified in Sec.~\ref{sec:complexity}.

\subsection{Computational Cost: Offline and Online}\label{sec:complexity}

We now collect the computational cost of the three algorithms. The aim is not a rigorous operation count. Constants, memory traffic, and the GPU batching discussed in Sec.~\ref{sec:brute_force} are all left aside. Instead, the focus is on estimating how each solver scales with the network size, the dataset size, and the iteration budget. This scaling, rather than the constants, is what separates the three methods. Two classes of expense are worth keeping apart. \emph{Offline} costs are paid once, before the iterative loop, and depend on the dataset and on the network but not on the number of iterations. \emph{Online} costs are paid inside the loop, once per iteration. For the Brute Force approach, these are paid once per candidate assignment. Table~\ref{tab:complexity} collects both, with $N = 2n_v + 3n_e$ the number of degrees of freedom of the saddle-point system in Eq.~\eqref{eq:block_matrix} and $n_{\mathrm{dof}}$ that of the Hermite CMR surrogate.

\begin{table}[h!]
	\centering
	\resizebox{0.98\textwidth}{!}{%
	\begin{tabular}{l l l l}
		\toprule
		& \textbf{Stage} & \textbf{Cost} & \textbf{Required by} \\
		\midrule
		\multirow{4}{*}{\rotatebox{90}{\textbf{Offline}}}
		& Dataset generation (sampling of $\mathcal{D}$)      & $\mathcal{O}(|\mathcal{D}|)$                & all \\
		& Factorization of the block matrix~\eqref{eq:block_matrix} & $\mathcal{O}(N^3)$                     & all \\
		& KD-tree construction over $\mathcal{D}$             & $\mathcal{O}(\mathcal{N}_m \log \mathcal{N}_m)$ & ADM, DA \\
		& CMR surrogate fit (assembly $+$ normal equations)   & $\mathcal{O}(\mathcal{N}_m + n_{\mathrm{dof}}^3)$ & ADM-CMR \\
		\midrule
		\multirow{3}{*}{\rotatebox{90}{\textbf{Online}}}
		& Physical projection $P_{\mathcal{A}}$ (back-substitution) & $\mathcal{O}(N^2)$                     & all \\
		& Hard data projection $P_{\mathcal{Z}}$ (KD-tree queries)  & $\mathcal{O}(n_e \log \mathcal{N}_m)$  & ADM, DA\textsuperscript{$\dagger$} \\
		& Soft data projection $P_{\mathcal{Z}}^{\beta}$ (all distances) & $\mathcal{O}(n_e \mathcal{N}_m)$   & DA \\
		\midrule
		\multirow{3}{*}{\rotatebox{90}{\textbf{Total}}}
		& Brute Force  & $\mathcal{O}\!\left(N^3 + N_{\mathrm{comb}}\,N^2\right)$ & \\
		& ADM          & $\mathcal{O}\!\left(N^3 + J_{\max}(N^2 + n_e \log \mathcal{N}_m)\right)$ & \\
		& DA           & $\mathcal{O}\!\left(N^3 + J_{\max}(N^2 + n_e \mathcal{N}_m)\right)$ & \\
		\bottomrule
	\end{tabular}}
	\caption{Offline and online computational costs of the three solvers. \textsuperscript{$\dagger$}The DA evaluates $P_{\mathcal{Z}}$ only to monitor the hard objective, not to drive the iteration.}
	\label{tab:complexity}
\end{table}

Three consequences deserve comment. First, the cubic factorization dominates the offline cost but is paid only once: the left-hand side of Eq.~\eqref{eq:block_matrix} is the same for every data assignment and every iteration, so all three solvers reuse a single factorization. This is what makes the Brute Force tractable at all, each enumerated assignment costing $\mathcal{O}(N^2)$ instead of $\mathcal{O}(N^3)$.

Second, the online costs differ in where they grow. The Brute Force is combinatorial in $N_{\mathrm{comb}}$, and hence limited to small networks and moderate datasets, whereas the ADM and the DA are linear in the iteration budget $J_{\max}$. Between the latter two, the only structural difference is the data projection: the hard one queries the KD-tree at $\mathcal{O}(\log \mathcal{N}_m)$ per edge, while the soft one evaluates all $\mathcal{N}_m$ exponentials, since every datum carries a nonzero weight. The DA thus trades a logarithmic query for a linear sweep. This is a modest penalty, which is dominated at these network sizes by the $\mathcal{O}(N^2)$ back-substitution.

Third, and more consequential than the per-iteration arithmetic, the DA needs a larger iteration budget. The cooling schedule must be slow enough for the trajectory to settle into the global basin before the weights sharpen.

	\section{Numerical Experiments}\label{sec:numerical_experiments}
In this section, we evaluate three methods introduced previously: the Brute Force method, the Alternating Direction Method (ADM), and Deterministic Annealing (DA). We apply these methods to the Optimization Problem~\ref{prob:op} (OP) to assess their accuracy, robustness, and computational efficiency across various scenarios. All datasets are synthetic. We report both the datasets and the recovered fields in dimensionless variables, unless we state physical units explicitly.

\subsection{Benchmarks on Hydraulic Networks}\label{sec:benchmark_networks}

We first assess the methods on four benchmarks of increasing structural complexity. These include a linear constitutive manifold, a nonlinear manifold, variations in network size and topology, and a network with several distinct constitutive types. The first two benchmarks are small enough for Brute Force enumeration. This provides a certified global optimum for Problem~\ref{prob:op} to serve as a reference for the ADM and DA approximations. For the last two benchmarks, the high number of candidate assignments precludes this enumeration, similar to the hemodynamics application.

\subsubsection{Linear Constitutive Manifolds}\label{sec:laminar_fluid_flow}\label{sec:adm_da_baseline}

We begin with the laminar regime. Here, the constitutive relation is linear and the dataset, although discrete, samples a convex manifold. This is perhaps the simplest setting to expose the solver's behavior, and it serves as the baseline against which we later contrast the non-convex turbulent case.

\paragraph{$\bullet$ Comparative algorithmic performance and initialization dependence.}

In this first experiment, we compare the initialization-independent DA and the ADM against the Brute Force solution on a simple topology that replicates the classic \emph{Wheatstone bridge} of electrical circuit theory. The ADM is evaluated under three initialization strategies: CMR, Null-space, and Random.

Figure~\ref{fig:wheatstone_topology} illustrates this network alongside the prescribed boundary conditions. The duct geometry and the properties of the working fluid (water at 20$^\circ$C) are uniform across all edges, and Table~\ref{tab:network_parameters} summarizes every numerical value.

\begin{figure}[h!]
	\centering
	\begin{minipage}[c]{0.45\textwidth}
		\centering
		\def\svgwidth{\linewidth}
		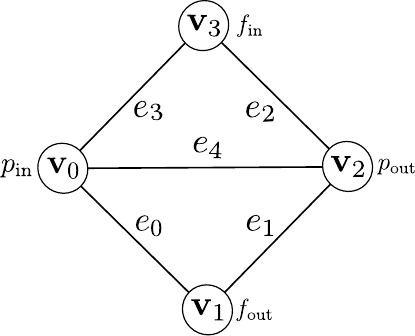
		\caption{Schematic illustration of the \emph{Wheatstone bridge} hydraulic network topology.}
		\label{fig:wheatstone_topology}
	\end{minipage}\hfill
	\begin{minipage}[b]{0.525\textwidth}
		\centering
		\captionof{table}{Geometric properties, fluid physical parameters, and boundary conditions for the \emph{Wheatstone bridge} hydraulic network.}
		\label{tab:network_parameters}
		\vspace{0.15cm}
		\resizebox{\linewidth}{!}{%
			\begin{tabular}{llcl}
				\toprule
				\textbf{Parameter} & \textbf{Symbol} & \textbf{Value} & \textbf{Unit} \\
				\midrule
				\multicolumn{4}{c}{\textit{Graph Topology \& Geometry}} \\
				\midrule
				Number of nodes & $n_v$ & 4 & -- \\
				Number of edges & $n_e$ & 5 & -- \\
				Pipe length, arms (Edges 0--3) & $L_{0\text{--}3}$ & $\sqrt{2}$ & $\mathrm{m}$ \\
				Pipe length, bridge (Edge 4) & $L_{4}$ & $2$ & $\mathrm{m}$ \\
				Pipe diameter   & $D$   & $1.0 \times 10^{-2}$ & $\mathrm{m}$ \\
				Pipe roughness  & $\varepsilon$ & $1.0 \times 10^{-5}$ & $\mathrm{m}$ \\
				\midrule
				\multicolumn{4}{c}{\textit{Fluid Properties (Water at 20$\,^\circ$C)}} \\
				\midrule
				Fluid density     & $\rho$ & $998.0$ & $\mathrm{kg/m^3}$ \\
				Dynamic viscosity & $\mu_f$  & $1.0 \times 10^{-3}$ & $\mathrm{Pa \cdot s}$ \\
				\midrule
				\multicolumn{4}{c}{\textit{Boundary Conditions}} \\
				\midrule
				Inlet pressure (Node 0)  & $p_{\text{in}}$  & $101.35 \times 10^3$ & $\mathrm{Pa}$ \\
				Outlet pressure (Node 2) & $p_{\text{out}}$ & $101.30 \times 10^3$ & $\mathrm{Pa}$ \\
				Sink flux (Node 1)       & $f_{\text{out}}$     & $-4.0 \times 10^{-6}$ & $\mathrm{m^3/s}$ \\
				Source flux (Node 3)     & $f_{\text{in}}$     & $8.0 \times 10^{-6}$  & $\mathrm{m^3/s}$ \\
				\bottomrule
			\end{tabular}
		}
	\end{minipage}
\end{figure}

Since the hydraulic conductance is uniform throughout the network, a single constitutive dataset is constructed using the \emph{Local Constitutive Sampling} strategy, with $\mathcal{N}_{m} = 30$ data points up to a maximum Reynolds number of $\mathrm{Re}_{\max} = 1500$. This upper bound falls entirely within the laminar regime ($\mathrm{Re} < 2300$), yielding the linear constitutive dataset illustrated by the blue markers in Fig.~\ref{fig:wheatstone_exp_adm_init_linear}(c)--(h). No artificial noise is introduced at this stage, so the analysis focuses on the algorithms' accuracy and their sensitivity to the initial guess.

In the numerical setup, the CMR strategy employs the linear least-squares model of Sec.~\ref{sec:cmr}(a), while the Random initialization relies on a fixed seed for reproducibility. The ADM is executed for a fixed limit of 50 iterations with no early-termination criterion. For comparison, the DA solver is run over the same 50-iteration budget, cooling the soft assignment from an initial inverse temperature $\beta_0 = 0.02$ to a terminal value of $\beta_{\max} = 10^{3}$. This choice of $\beta_{\max}$ is justified \emph{a posteriori} by the dedicated cooling study of Fig.~\ref{fig:wheatstone_da_cooling}.

\begin{figure}[h!]
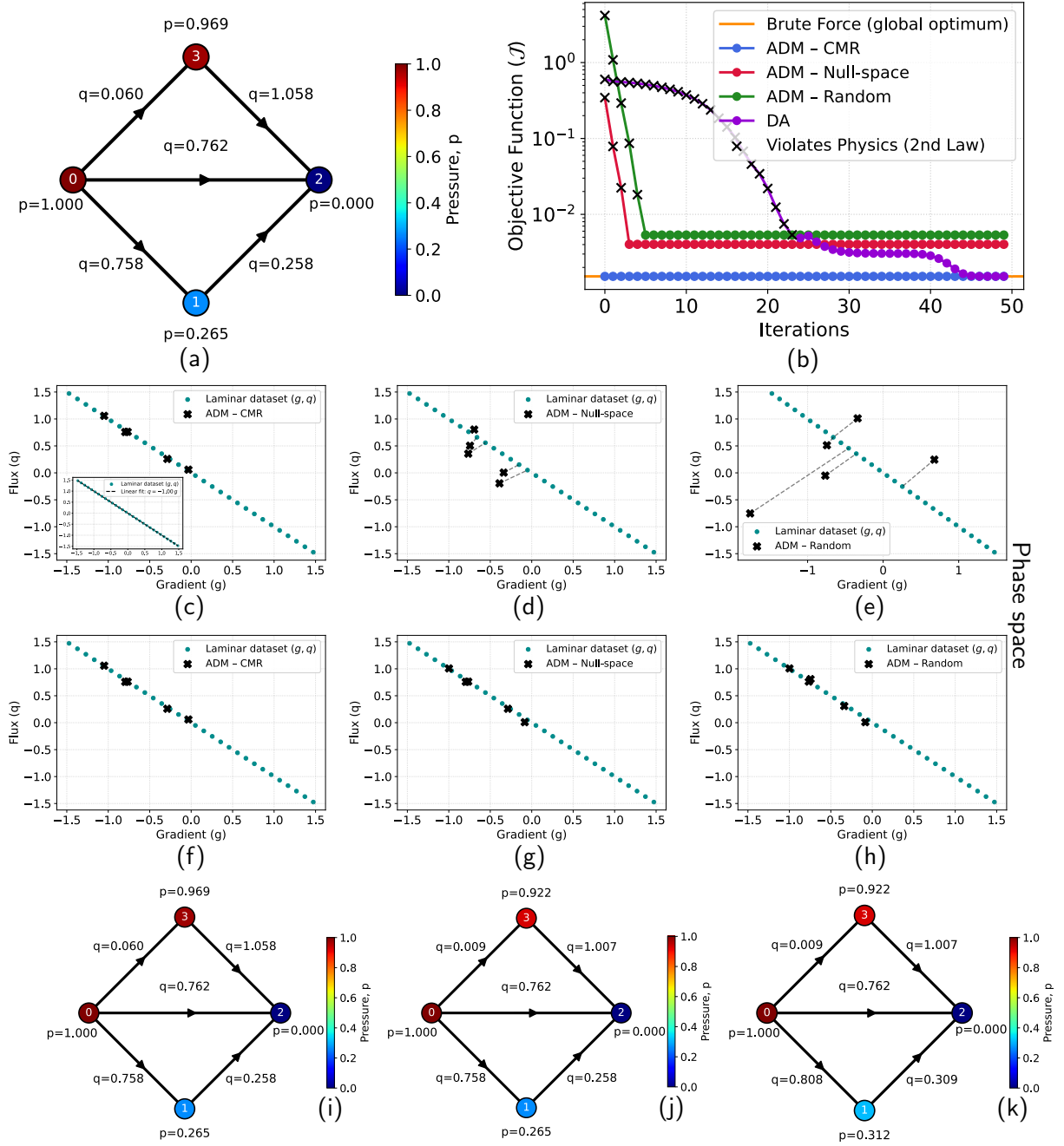

	\centering
	\svgfig{1.0\columnwidth}{505.92945009}{wheatstone_exp_adm_init_linear}
	\caption{Performance and convergence of the ADM and DA solvers on the \emph{Wheatstone bridge} topology under laminar flow regime. (a) Global optimum solution obtained via the Brute Force method, where the nodal colormap denotes the dimensionless pressure field ($p$) and arrows indicate flux ($q$) magnitudes and directions. (b) Evolution of the objective function ($\mathcal{J}$) across iterations; cross markers ($\times$) indicate states that violate the second law of thermodynamics. Panels (c--e) and (f--h) show the phase-space states at the first and final ADM iterations, respectively, with columns corresponding to the CMR (c, f), Null-space (d, g), and Random (e, h) initializations. Panels (i--k) display the corresponding final pressure and flux fields predicted by each initialization strategy.}
	\label{fig:wheatstone_exp_adm_init_linear}
\end{figure}

The exact solution reaches at most $\mathrm{Re} \approx 1060$ across the five edges, so the network remains strictly laminar, and the Brute Force enumeration returns the global optimum $\mathcal{J}_{\text{bf}} = 1.52 \times 10^{-3}$ as the reference. Against this reference, the ADM shows strong sensitivity to the initial guess (Fig.~\ref{fig:wheatstone_exp_adm_init_linear}(b)). The CMR initialization attains the optimum on the very first iteration while preserving thermodynamic admissibility throughout. In contrast, the Null-space and Random initializations violate the second law through Iterations 0 to 2 and 0 to 4, respectively. They subsequently stagnate well above the global optimum (Table~\ref{tab:adm_initialization_results}). The DA successfully reaches the optimum by the end of its cooling schedule.

The phase-space panels of Fig.~\ref{fig:wheatstone_exp_adm_init_linear}(c)--(h) interpret these profiles geometrically, with the datum selected by each solver listed in Table~\ref{tab:adm_selected_data}. The CMR initial state already lies on the dataset. Therefore, the nearest-datum projection assigns the optimal point on all five edges simultaneously, making the final state identical. The DA assignment also coincides with this optimal state. The Null-space initialization starts near the dataset but strictly off it, while the Random initialization scatters far away. Both are eventually pulled onto the dataset, but they settle on a different assignment. Specifically, the high-flux Edge 2 is matched to the neighboring datum (index 24 instead of 25), which explains their higher objective values. The Random run mis-selects an additional datum on Edge 1, leading to correspondingly worse performance.

This mis-selection on Edge 2, shared by both uninformed initializations, is then amplified by the network physics. The resulting error on Edge 2 itself is minor (approximately $5\%$). However, mass conservation transfers this imbalance to the low-flux Edge 3 ($\mathrm{Re}\approx 43$). Its flux collapses from $0.060$ to $0.009$, representing an $85\%$ drop. This is the largest deviation in the flux block of Table~\ref{tab:adm_initialization_results}. Furthermore, its gradient substantially deviates on an edge that otherwise retained its optimal datum (Table~\ref{tab:adm_selected_data}).

Table~\ref{tab:adm_initialization_results} also includes the exact solution of the classical problem. The dataset samples this problem but does not contain the exact solution itself. This comparison measures the inherent cost of using a finite dataset. The Brute Force optimum is mathematically unimprovable, yet it still departs from the exact solution by $0.64\%$ in nodal pressure, $0.90\%$ in gradient, and $3.46\%$ in flux. The deviation again falls on Edge 3, but for a different reason. The assignment is correct, but the dataset simply lacks a point near its operating flux. Consequently, the nearest available datum lies $21\%$ above the exact value, and the recovered field lies $41\%$ above it. This limitation is intrinsic to the finite dataset and not a failure of the solvers.

\begin{table}[h!]
	\centering
	\caption{Constitutive data point selected by each solver on every edge of the laminar
		\emph{Wheatstone bridge}. The index $n$ labels the shared set of $\mathcal{N}_m = 30$
		data pairs, ordered by increasing flux; $(g_{n^\star}, q_{n^\star})$ are the coordinates of
		the globally-optimal datum. Indices in red differ
		from the Brute Force optimum $n^\star$.}
	\label{tab:adm_selected_data}
	\vspace{0.2cm}
	\resizebox{0.75\textwidth}{!}{%
		\begin{tabular}{@{}lcccccc@{}}
			\toprule
			\textbf{Edge} & $n^\star$ (BF) & $(g_{n^\star},\,q_{n^\star})$ & CMR & Null-space & Random & DA \\
			\midrule
			0 \ ($0\!\to\!1$) & 22 & $(-0.762,\ 0.762)$ & 22 & 22 & 22 & 22 \\
			1 \ ($1\!\to\!2$) & 17 & $(-0.254,\ 0.254)$ & 17 & 17 & \errval{18} & 17 \\
			2 \ ($3\!\to\!2$) & 25 & $(-1.067,\ 1.067)$ & 25 & \errval{24} & \errval{24} & 25 \\
			3 \ ($0\!\to\!3$) & 15 & $(-0.051,\ 0.051)$ & 15 & 15 & 15 & 15 \\
			4 \ ($0\!\to\!2$) & 22 & $(-0.762,\ 0.762)$ & 22 & 22 & 22 & 22 \\
			\bottomrule
	\end{tabular}}
\end{table}

\begin{table}[h!]
	\centering
	\caption{Comparison of the global optimum (Brute Force) against the ADM and DA approximations for the \emph{Wheatstone bridge} network. Values highlighted in red indicate a relative absolute error greater than $10\%$ with respect to the global optimum. The \emph{Exact} column is the solution of the classical hydraulic problem, which the dataset samples but does not contain; no objective is associated with it, since it is not an assignment of data points.}
	\label{tab:adm_initialization_results}
	\vspace{0.2cm}
	\resizebox{0.85\textwidth}{!}{
		\begin{tabular}{@{}l >{\columncolor{exactshade}}c ccccc@{}}
			\toprule
			\textbf{Node} & \exhead{Exact} & \textbf{Brute Force} & \textbf{CMR} & \textbf{Null-space} & \textbf{Random} & \textbf{DA} \\
			\midrule
			0 & \exval{$1.00000$} & $1.00000$ & $1.00000$ & $1.00000$ & $1.00000$ & $1.00000$ \\
			1 & \exval{$0.26952$} & $0.26545$ & $0.26545$ & $0.26545$ & \errval{$0.31236$} & $0.26545$ \\
			2 & \exval{$0.00000$} & $0.00000$ & $0.00000$ & $0.00000$ & $0.00000$ & $0.00000$ \\
			3 & \exval{$0.96096$} & $0.96909$ & $0.96909$ & $0.92218$ & $0.92218$ & $0.96909$ \\
			\midrule
			\textbf{Objective} $\mathcal{J}$ & \exval{--} & $1.52\times 10^{-3}$ & $1.52\times 10^{-3}$ & $4.03\times 10^{-3}$ & $5.34\times 10^{-3}$ & $1.52\times 10^{-3}$ \\
			\bottomrule
		\end{tabular}
	}

	\vspace{0.35cm}

	\resizebox{0.85\textwidth}{!}{
		\begin{tabular}{@{}l >{\columncolor{exactshade}}c ccccc@{}}
			\toprule
			\textbf{Edge} & \exhead{Exact} & \textbf{Brute Force} & \textbf{CMR} & \textbf{Null-space} & \textbf{Random} & \textbf{DA} \\
			\midrule
			\multicolumn{7}{@{}l}{\textit{Gradient} ($g$)} \\
			0 \ ($0 \!\to\! 1$) & \exval{$-0.79100$} & $-0.79540$ & $-0.79540$ & $-0.79540$ & $-0.74460$ & $-0.79540$ \\
			1 \ ($1 \!\to\! 2$) & \exval{$-0.29185$} & $-0.28745$ & $-0.28745$ & $-0.28745$ & \errval{$-0.33824$} & $-0.28745$ \\
			2 \ ($3 \!\to\! 2$) & \exval{$-1.04057$} & $-1.04937$ & $-1.04937$ & $-0.99858$ & $-0.99858$ & $-1.04937$ \\
			3 \ ($0 \!\to\! 3$) & \exval{$-0.04227$} & $-0.03347$ & $-0.03347$ & \errval{$-0.08426$} & \errval{$-0.08426$} & $-0.03347$ \\
			4 \ ($0 \!\to\! 2$) & \exval{$-0.76569$} & $-0.76569$ & $-0.76569$ & $-0.76569$ & $-0.76569$ & $-0.76569$ \\
			\addlinespace[0.15cm]
			\multicolumn{7}{@{}l}{\textit{Flux} ($q$)} \\
			0 \ ($0 \!\to\! 1$) & \exval{$0.79100$} & $0.75753$ & $0.75753$ & $0.75753$ & $0.80832$ & $0.75753$ \\
			1 \ ($1 \!\to\! 2$) & \exval{$0.29185$} & $0.25838$ & $0.25838$ & $0.25838$ & \errval{$0.30917$} & $0.25838$ \\
			2 \ ($3 \!\to\! 2$) & \exval{$1.04057$} & $1.05790$ & $1.05790$ & $1.00710$ & $1.00710$ & $1.05790$ \\
			3 \ ($0 \!\to\! 3$) & \exval{$0.04227$} & $0.05960$ & $0.05960$ & \errval{$0.00880$} & \errval{$0.00880$} & $0.05960$ \\
			4 \ ($0 \!\to\! 2$) & \exval{$0.76569$} & $0.76193$ & $0.76193$ & $0.76193$ & $0.76193$ & $0.76193$ \\
			\bottomrule
		\end{tabular}
	}
\end{table}

Finally, Fig.~\ref{fig:wheatstone_exp_da_linear} shows the annealing at work: the effective number of active data points per edge, $N_{\mathrm{eff}} = \exp\mathcal{H}$, decreases monotonically from $\mathcal{N}_m$ to $1$ as the system is cooled, and the per-edge weight trajectories show the probability mass, initially spread over the whole dataset, concentrating onto a single datum at the coldest temperature.

\begin{figure}[h!]
	\centering
	\svgfig{0.925\columnwidth}{551.91793266}{wheatstone_exp_da_linear}
	\caption{Mechanistic evolution of the DA solver on the noise-free laminar \emph{Wheatstone bridge} ($\mathcal{N}_m = 30$). (a) Active data points per edge, $N_{\mathrm{eff}} = \exp\mathcal{H}$, versus inverse temperature $\beta$, showing the contraction from uniform assignment ($\beta \to 0$) to a single datum ($\beta \to \infty$). (b) Final membership weights $w_n$ at $\beta \approx 805$, illustrating data localization. (c--g) Weight evolution $w_n(\beta)$ for edges 0--4, respectively, where the diffuse probability mass concentrates onto the target datum (horizontal dotted lines) as $\beta$ increases.}
	\label{fig:wheatstone_exp_da_linear}
\end{figure}

\paragraph{$\bullet$ Influence of dataset resolution.}

We now refine the dataset from $\mathcal{N}_m = 20$ to $160$, leaving the network and the boundary conditions unchanged (Fig.~\ref{fig:wheatstone_exp_adm_iter_linear}).

\begin{figure}[h!]
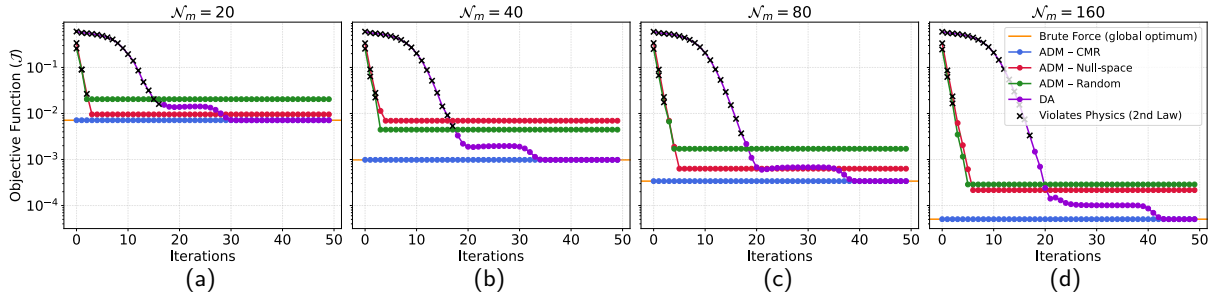

	\centering
	\svgfig{1.0\columnwidth}{579.13890293}{wheatstone_exp_adm_iter_linear}
	\caption{Influence of dataset resolution on the ADM and DA convergence for the \emph{Wheatstone bridge} topology. Evolution of the objective function $\mathcal{J}$ over 50 iterations for (a) $\mathcal{N}_m = 20$, (b) $\mathcal{N}_m = 40$, (c) $\mathcal{N}_m = 80$, and (d) $\mathcal{N}_m = 160$. The solid orange horizontal line marks the Brute Force global optimum baseline. Cross markers ($\times$) indicate iterations where the solution violates thermodynamic admissibility constraints.}
	\label{fig:wheatstone_exp_adm_iter_linear}
\end{figure}

As expected, the final ADM objective decreases with resolution for every initialization. A denser dataset represents the constitutive manifold more accurately and consequently lowers the projection error. This effect is largest where there is the most room for improvement. For instance, the Random initialization improves by nearly two orders of magnitude across the sweep, and the Null-space initialization approaches the global optimum by $\mathcal{N}_m = 80$. However, their ordering is not fixed. The gap between the two narrows with resolution and occasionally reverses (Fig.~\ref{fig:wheatstone_exp_adm_iter_linear}(b)). This suggests that a denser manifold offers the ADM more routes out of a local minimum. However, the number of iterations needed to restore thermodynamic admissibility (indicated by the cross markers) remains unchanged. This is unsurprising because the constraint is never strictly imposed during the iterations (Remark~\ref{rem:thermo_consistency}), meaning that refining the dataset provides no mechanism to enforce it.

In contrast, the CMR initialization and the DA are insensitive to resolution, as both recover the global optimum at every density. The DA nevertheless requires a colder terminal temperature $\beta_{\max}$ as the dataset is refined, which Proposition~\ref{prop:dda_limits}(ii) accurately predicts. The transition from a soft to a hard assignment is governed by the gap $d_{(2)}^2 - d_{(1)}^2$ between the two closest candidates. This gap shrinks as the dataset is packed more densely, so a lower temperature is needed to isolate the nearest neighbor rather than average two nearly equidistant points.

\begin{figure}[h!]
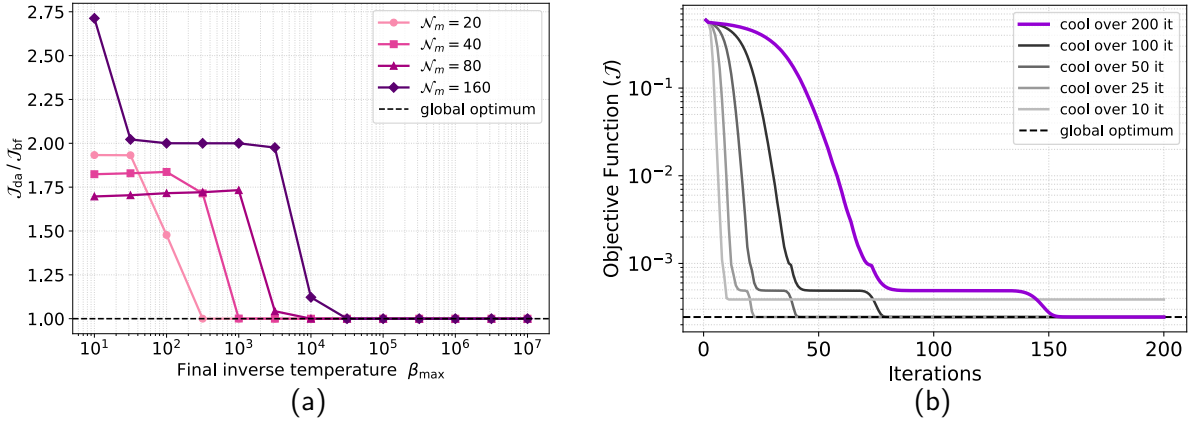

	\centering
	\svgfig{1.0\columnwidth}{486.12540756}{wheatstone_da_cooling}
	\caption{Annealing of the DA solver on the laminar \emph{Wheatstone bridge}.
		(a) Ratio $\mathcal{J}_{\text{da}}/\mathcal{J}_{\text{bf}}$ versus terminal temperature
		$\beta_{\max}$ (50-iteration budget) for four densities $\mathcal{N}_m$: each curve drops
		from a stagnation plateau onto the optimum (dashed line) at a threshold $\beta_{\max}$ that
		grows with density. (b) Convergence histories for $\mathcal{N}_m = 90$ and $\beta_{\max}=10^5$,
		sweeping the cooling duration ($J_{\mathrm{cool}}$ iterations from $\beta_0$ to $\beta_{\max}$):
		all schedules reach the optimum except the near-instantaneous quench, showing that the cooling
		\emph{rate} is immaterial in the laminar regime.}
	\label{fig:wheatstone_da_cooling}
\end{figure}

Figure~\ref{fig:wheatstone_da_cooling}(a) confirms this behavior. Each curve sits on a plateau at $\mathcal{J}_{\text{da}} \approx 2\,\mathcal{J}_{\text{bf}}$ and drops onto the optimum only after passing a critical $\beta_{\max}$. This threshold rises from approximately $3\times10^{2}$ at $\mathcal{N}_m = 20$ to roughly $3\times10^{4}$ at $\mathcal{N}_m = 160$. The value $\beta_{\max} = 10^{5}$ used in this parameter sweep lies safely beyond every threshold, and the $\beta_{\max} = 10^{3}$ used in the prior comparative study safely clears the threshold at its specific density of $\mathcal{N}_m = 30$.

The cooling \emph{rate}, by contrast, barely matters here. Fixing $\mathcal{N}_m = 90$ and stretching the schedule from a near-instantaneous quench to a slow anneal, the DA recovers the optimum in almost every case, remaining close to it with as few as ten cooling iterations (Fig.~\ref{fig:wheatstone_da_cooling}(b)). This rate-independence is a property of the convex laminar manifold and does not survive into the turbulent regime of Sec.~\ref{sec:turb_adm_da_baseline}.

\paragraph{$\bullet$ Influence of measurement noise.}

So far the datasets have been noise-free, which isolated the effects of dataset resolution and initialization. Experimental measurements, however, always carry uncertainty and acquisition errors, so we now examine how the solvers behave when the constitutive data are noisy.

Let $\mathcal{D} = \{(g_n, q_n)\}_{n=1}^{\mathcal{N}_m}$ be the reference, noise-free dataset. A perturbed dataset $\mathcal{D}^{\xi}$ is built by adding independent noise to both the pressure-gradient and the flux components:
\begin{equation}
	\mathcal{D}^{\xi} = \left\{(g_n^{\xi}, q_n^{\xi})\right\}_{n=1}^{\mathcal{N}_m}, 
	\quad \text{with} \quad
	\left\{
	\begin{aligned}
		g_n^{\xi} &= g_n + \xi_{n}^{g}, \\
		q_n^{\xi} &= q_n + \xi_{n}^{q},
	\end{aligned}
	\right.
\end{equation}
where $\xi_{n}^{g}$ and $\xi_{n}^{q}$ are random perturbations. Two noise models are considered:
\begin{itemize}
	\item[$\clubsuit$] \textbf{Gaussian noise:} the perturbations are sampled from a normal distribution with zero mean and standard deviation $\eta$, i.e., $\xi_{n}^{g}, \xi_{n}^{q} \sim \mathcal{N}(0, \eta^2)$.
	\item[$\clubsuit$] \textbf{Uniform noise:} the perturbations are drawn from a zero-mean uniform distribution, $\xi_{n}^{g}, \xi_{n}^{q} \sim \mathcal{U}(-\sqrt{3}\eta, \sqrt{3}\eta)$.
\end{itemize}

\begin{remark}
	This choice of support gives the uniform noise the same variance as its Gaussian counterpart. The two models can therefore be compared at equal noise energy, differing only in the shape of the distribution.
\end{remark}

To assess robustness, the dataset size $\mathcal{N}_m$ is increased while the noise level $\eta$ is reduced, so that resolution and fidelity improve together, with $100$ independent realizations per pair $(\eta, \mathcal{N}_m)$. Panels (c)--(f) of Figs.~\ref{fig:wheatstone_conv_linear_gaussian} and \ref{fig:wheatstone_conv_linear_uniform} illustrate the resulting datasets.

We separate the two questions a noisy dataset raises: how accurate a solver is when it returns a solution, and how often it returns an admissible one at all. The former is reported by the convergence trends and ECDFs of Figs.~\ref{fig:wheatstone_conv_linear_gaussian} and \ref{fig:wheatstone_conv_linear_uniform}; the latter, separately, by the inadmissibility bar chart of Fig.~\ref{fig:wheatstone_noise_admissibility}. Throughout, a realization is counted as inadmissible when the \emph{final} iterate of the algorithm violates~\eqref{eq:clausius_duhem}, regardless of whether earlier iterates satisfied it.

Figures~\ref{fig:wheatstone_conv_linear_gaussian}(a) and \ref{fig:wheatstone_conv_linear_uniform}(a) report the median objective $\mathcal{J}$ along with its interquartile range. Because the solvers fail at different rates, these statistics are computed solely on the realizations where every method successfully satisfies the second law. This ensures all curves are compared on identical data. As data quality improves from $(\eta, \mathcal{N}_m) = (0.05, 20)$ to $(0.00625, 160)$, the median objective falls monotonically for every solver. A more accurate manifold yields lower projection errors. The Brute Force method, which enumerates every candidate assignment, provides a per-realization lower bound, and its median acts as the floor of each panel. The DA and ADM-CMR track that floor closely, with the DA being essentially coincident with it. Conversely, the Null-space and Random initializations sit well above this optimal bound. Notably, both noise distributions yield essentially the same trends.

Figures~\ref{fig:wheatstone_conv_linear_gaussian}(b) and \ref{fig:wheatstone_conv_linear_uniform}(b) assess the same realizations distributionally through empirical cumulative distribution functions (ECDFs) aggregated across all configurations. Curves shifted to the left indicate superior performance. The Brute Force reference stochastically dominates every iterative solver. The DA curve lies essentially on top of it, occupying the same low-objective region with a comparably steep transition. The ADM-CMR curve sits marginally to its right, proving that the initialization-free DA reaches the global optimum just as reliably as the best-initialized ADM. The Null-space and Random initializations shift significantly to the right, confirming their tendency to settle into higher-objective local minima or to fail outright when working with coarse data.

\begin{figure}[h!]
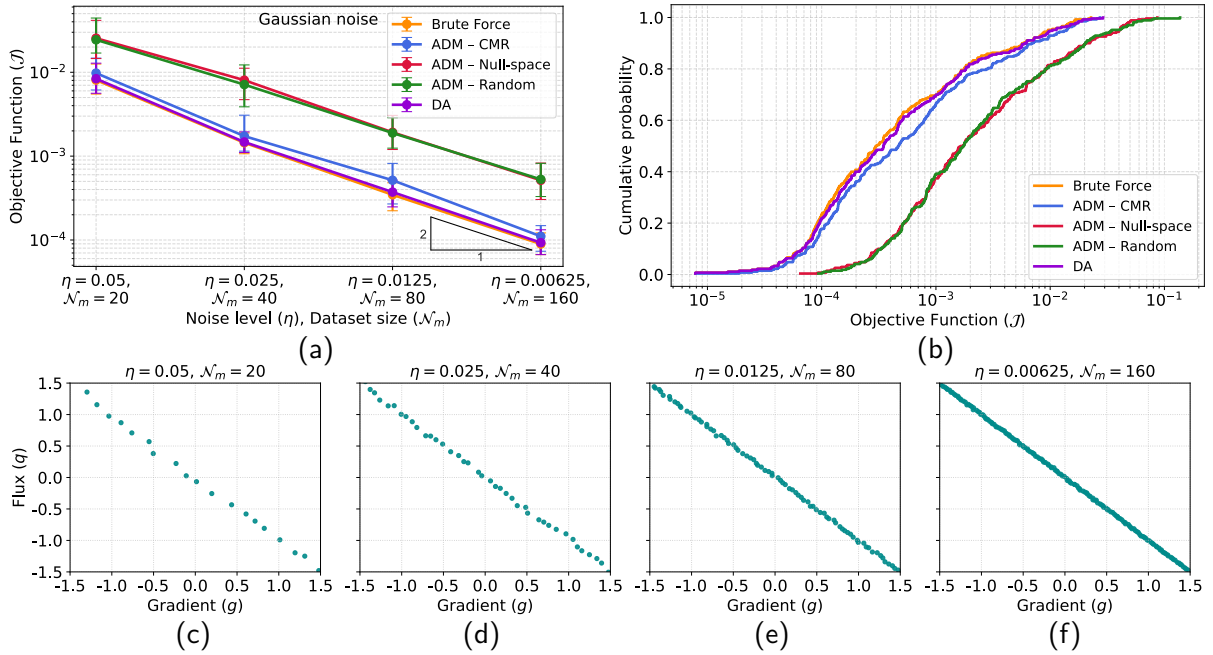

	\centering
	\svgfig{1.0\columnwidth}{488.11641144}{wheatstone_conv_linear_gaussian}
	\caption{Statistical performance comparison of solver approaches for the \emph{Wheatstone} network under Gaussian noise perturbations. (a) Median objective function value $\mathcal{J}$ with interquartile range, computed over the realizations solved admissibly by every method, illustrating convergence behavior as data quality improves. (b) Empirical cumulative distribution functions (ECDFs) of $\mathcal{J}$ over the same realizations, aggregated across all configurations. (c)--(f) Representative synthetic Gaussian datasets, showing the progression from sparse, high-noise measurements to dense, low-noise representations.}
	\label{fig:wheatstone_conv_linear_gaussian}
\end{figure}

\begin{figure}[h!]
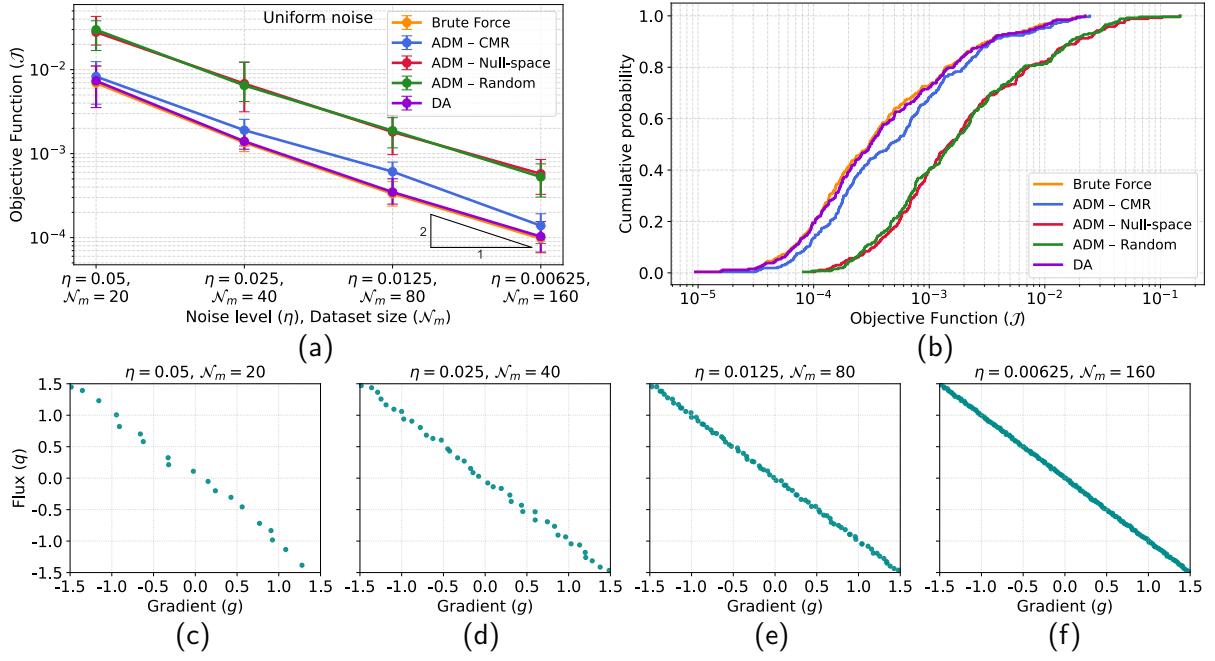

	\centering
	\svgfig{1.0\columnwidth}{483.00476025}{wheatstone_conv_linear_uniform}
	\caption{Statistical performance comparison of solver approaches for the \emph{Wheatstone} network under Uniform noise perturbations. (a) Median objective function value $\mathcal{J}$ with interquartile range, computed over the realizations solved admissibly by every method. (b) Empirical cumulative distribution functions (ECDFs) of $\mathcal{J}$ over the same realizations, aggregated across all configurations. (c)--(f) Representative synthetic Uniform datasets, showing the progression from sparse, high-noise measurements to dense, low-noise representations.}
	\label{fig:wheatstone_conv_linear_uniform}
\end{figure}

We now evaluate robustness by examining how often each solver returns a thermodynamically admissible solution (Fig.~\ref{fig:wheatstone_noise_admissibility}). Consistent with the noise-free findings, the reliability of the ADM depends heavily on its initialization. The Null-space and Random strategies are the least reliable, proving inadmissible in up to roughly $51\%$ and $54\%$ of the coarsest, noisiest realizations ($\mathcal{N}_m = 20$). Even where they do converge, they settle into higher-objective local minima with significant variability. The ADM-CMR is by far the most robust, landing near-optimal on almost every realization, with only about $12$ to $16\%$ inadmissible at the coarsest setting.

The DA solver starts from the initialization-independent data centroid and is second only to the ADM-CMR. Its final solution is inadmissible in approximately $18\%$ of the coarsest, noisiest realizations ($\mathcal{N}_m = 20$). However, on the realizations it solves admissibly, its median objective and interquartile range sit essentially on the Brute-Force lower bound across the entire $(\eta, \mathcal{N}_m)$ sequence. This matches, and marginally improves on, the ADM-CMR performance, importantly without requiring the construction of any constitutive surrogate.

\begin{figure}[h!]
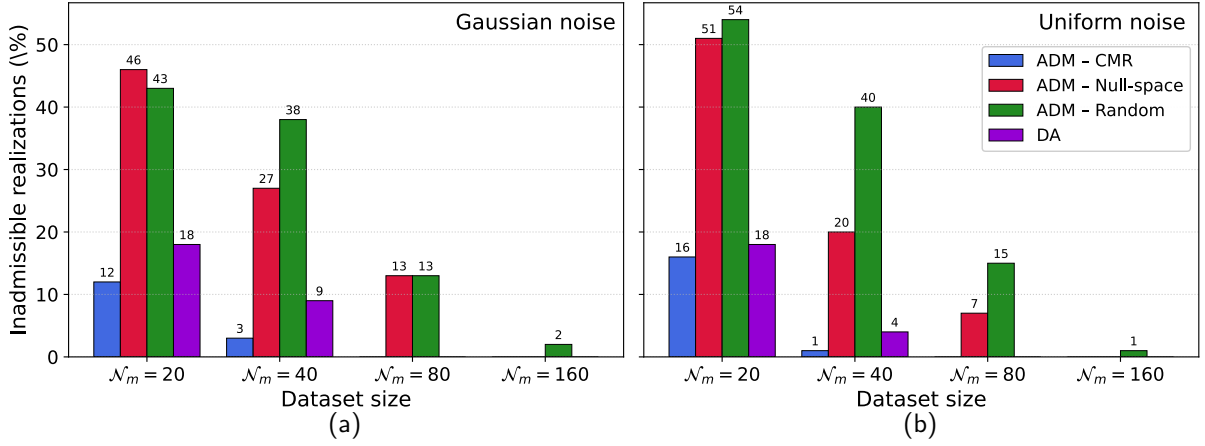

	\centering
	\svgfig{1.0\columnwidth}{562.68931021}{wheatstone_noise_admissibility}
	\caption{Fraction of the 100 stochastic realizations in which each solver's final solution is thermodynamically inadmissible, as a function of the dataset size $\mathcal{N}_m$, under (a) Gaussian and (b) Uniform noise. At $\mathcal{N}_m = 20$ the Null-space and Random initializations are inadmissible in up to $\approx 51\%$ and $\approx 54\%$ of the realizations, the initialization-free DA in $\approx 18\%$, and the ADM-CMR in only $12$--$16\%$; all solvers become reliable as the data is refined, with ADM-CMR the most robust and DA a close second throughout.}
	\label{fig:wheatstone_noise_admissibility}
\end{figure}

These analyses, however, are confined to laminar flows where the constitutive manifold is linear. We now drop that assumption and evaluate the framework on nonlinear constitutive behavior.

\subsubsection{Nonlinear Constitutive Manifolds}\label{sec:turbulent_fluid_flow}\label{sec:turb_adm_da_baseline}

We now repeat the study of Sec.~\ref{sec:adm_da_baseline} under non-convex conditions. The solvers, the three initialization strategies, and the Brute Force reference are exactly the same. Only the underlying flow regime changes.

\paragraph{$\bullet$ Comparative algorithmic performance and initialization dependence.}

The network topology, the fluid, and the geometry remain identical to the laminar study (Fig.~\ref{fig:wheatstone_topology} and Table~\ref{tab:network_parameters}). First, the boundary conditions are raised to drive the network into the turbulent regime. The prescribed pressures at nodes 0 and 2 become $105.3~\mathrm{kPa}$ and $101.3~\mathrm{kPa}$, while the nodal fluxes at nodes 1 and 3 become $0.0005~\mathrm{m^3/s}$ and $-0.00025~\mathrm{m^3/s}$. This causes the two nodes to exchange the source and sink roles they carry in Table~\ref{tab:network_parameters}. Second, the dataset is sampled up to $\mathrm{Re}_{\max} = 4\times10^4$ with $\mathcal{N}_{m} = 90$ points. This range covers the S-shaped, non-convex transition seen in Fig.~\ref{fig:wheatstone_exp_adm_init_nonlinear}(c)--(h). It is again noise-free, isolating non-convexity as the only factor under study. Third, the CMR uses the Hermite surrogate of Sec.~\ref{sec:cmr}(b) rather than the linear fit. Furthermore, both iterative solvers receive a 500-iteration budget, which is ten times the laminar one. The DA needs this extended budget. Only a slow enough schedule allows the iterates to pass the S-shaped region before the assignment weights sharpen. The DA cools from $\beta_0 = 0.02$ to $\beta_{\max} = 10^3$, a choice justified \emph{a posteriori} in Fig.~\ref{fig:wheatstone_da_cooling_turb}.

\begin{figure}[h!]
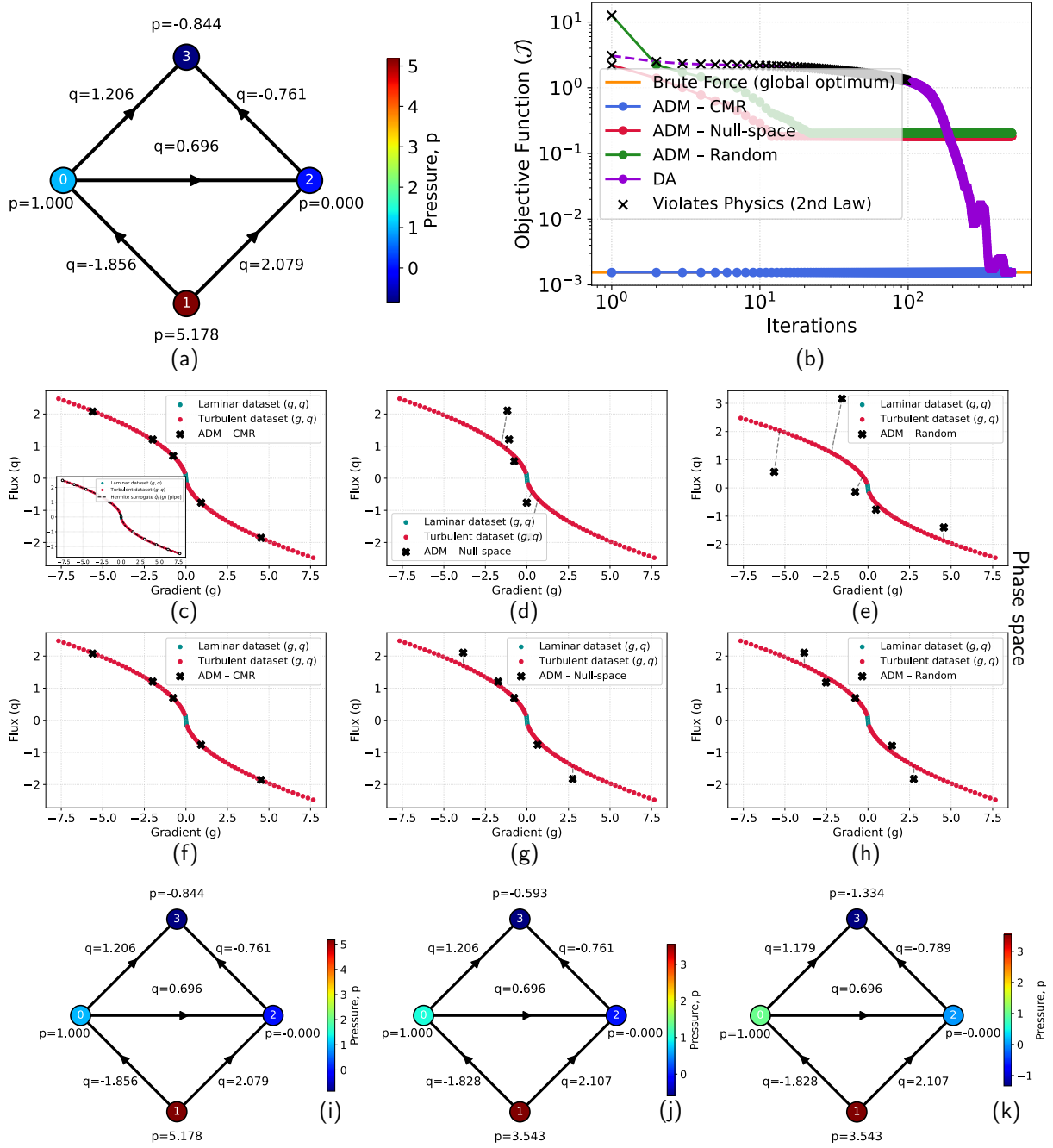

	\centering
	\svgfig{1.0\columnwidth}{541.60664236}{wheatstone_exp_adm_init_nonlinear}
	\caption{Performance and convergence of the ADM and DA solvers on the \emph{Wheatstone bridge} under turbulent flow. Panel layout as in Fig.~\ref{fig:wheatstone_exp_adm_init_linear}: (a) Brute Force global optimum (nodal pressure colormap, flux arrows); (b) objective $\mathcal{J}$ across iterations, cross markers ($\times$) flagging second-law violations; (c--h) phase-space states at the first (c--e) and final (f--h) ADM iterations for the CMR, Null-space, and Random initializations; (i--k) final pressure and flux fields per initialization.}
	\label{fig:wheatstone_exp_adm_init_nonlinear}
\end{figure}

The Brute Force optimum is $\mathcal{J}_{\text{bf}} = 1.54 \times 10^{-3}$ (Fig.~\ref{fig:wheatstone_exp_adm_init_nonlinear}(a)). The ADM is again highly sensitive to the initial guess. Here, it is much more so, since the S-shaped geometry creates several basins that easily trap the iteration. The CMR initialization avoids them, reaching the optimum on the first iteration with every iterate thermodynamically admissible. The Null-space and Random initializations both fail. Their first iterate violates the second law. The alternating projection then restores admissibility, but the iteration stalls two orders of magnitude above the optimum (Table~\ref{tab:turb_adm_initialization_results}). The DA, starting from the plain data centroid, escapes these traps. The annealing process carries the iterates past the non-convexity and recovers $\mathcal{J}_{\text{bf}}$ exactly (Fig.~\ref{fig:wheatstone_exp_adm_init_nonlinear}(b)).

The phase-space panels of Fig.~\ref{fig:wheatstone_exp_adm_init_nonlinear}(c)--(h) interpret these profiles geometrically. Table~\ref{tab:turb_adm_selected_data} lists the datum selected by each solver. The CMR guess already lies on the true solution, so its final state is identical to the initial one. The DA assignment also coincides with it, and both select the optimal datum on all five edges. The Null-space initialization pushes the states toward the origin due to its minimum-norm bias. The Random initialization scatters them far from the constitutive curve. Both settle on a different assignment, mis-assigning the four loaded edges (Edges 0 through 3) and recovering only Edge 4.

That pattern is not accidental. Each loaded edge touches a node with a free pressure, specifically node 1 or node 3. This is exactly the pressure that the stalled ADM gets wrong. The gradient of the incident edge drifts, and with it the nearest datum, while mass conservation still keeps the fluxes within a few percent of the optimum. Edge 4 bridges the two pressure-prescribed nodes, meaning its gradient is fixed entirely by the Dirichlet data. It cannot drift, and every method recovers its datum exactly. Table~\ref{tab:turb_adm_initialization_results} shows the same split in the recovered fields, observing deviations above $10$\% on the nodal pressures and the edge gradients, but none on the fluxes.

Table~\ref{tab:turb_adm_initialization_results} also reports the exact solution of the classical nonlinear problem, obtained by Picard iteration on the Churchill closure. It confirms that the Brute Force optimum stands in faithfully for the physical behavior here. The two differ by $0.77\%$, $0.83\%$, and $0.65\%$ in pressure, gradient, and flux. Furthermore, no individual edge exceeds $2\%$. The agreement is sharper than in the laminar experiment. The reason lies in the sampling rather than the closure. The five edges operate over $\mathrm{Re} \approx 1.1$ to $3.4\times10^{4}$, representing a spread of barely a factor of three. Therefore, a single $\mathcal{N}_m = 90$ dataset resolves every operating point equally well.

\begin{table}[h!]
	\centering
	\caption{Constitutive data point selected by each solver on every edge of the turbulent
		\emph{Wheatstone bridge} ($\mathcal{N}_m = 90$). The index $n$ labels the shared set of data pairs, ordered
		by increasing flux; $(g_{n^\star}, q_{n^\star})$ are the coordinates of the globally-optimal datum.
		Indices in red differ from the Brute Force optimum $n^\star$.}
	\label{tab:turb_adm_selected_data}
	\vspace{0.2cm}
	\resizebox{0.75\textwidth}{!}{%
		\begin{tabular}{@{}lcccccc@{}}
			\toprule
			\textbf{Edge} & $n^\star$ (BF) & $(g_{n^\star},\,q_{n^\star})$ & CMR & Null-space & Random & DA \\
			\midrule
			0 \ ($0\!\to\!1$) & 11 & $(+4.543,\ -1.865)$ & 11 & \errval{19} & \errval{19} & 11 \\
			1 \ ($1\!\to\!2$) & 82 & $(-5.588,\ +2.088)$ & 82 & \errval{75} & \errval{75} & 82 \\
			2 \ ($3\!\to\!2$) & 31 & $(+0.883,\ -0.752)$ & 31 & \errval{33} & \errval{27} & 31 \\
			3 \ ($0\!\to\!3$) & 66 & $(-2.029,\ +1.197)$ & 66 & \errval{64} & \errval{69} & 66 \\
			4 \ ($0\!\to\!2$) & 57 & $(-0.770,\ +0.696)$ & 57 & 57 & 57 & 57 \\
			\bottomrule
	\end{tabular}}
\end{table}

\begin{table}[h!]
	\centering
	\caption{Comparison of the global optimum (Brute Force) against the ADM and DA approximations for the \emph{turbulent} \emph{Wheatstone bridge} network. Values highlighted in red indicate a relative absolute error greater than $10\%$ with respect to the global optimum. The \emph{Exact} column is the solution of the classical nonlinear problem with the Churchill closure, obtained by Picard iteration; no objective is associated with it, since it is not an assignment of data points.}
	\label{tab:turb_adm_initialization_results}
	\vspace{0.2cm}
	\resizebox{0.85\textwidth}{!}{%
		\begin{tabular}{@{}l >{\columncolor{exactshade}}c ccccc@{}}
			\toprule
			\textbf{Node} & \exhead{Exact} & \textbf{Brute Force} & \textbf{CMR} & \textbf{Null-space} & \textbf{Random} & \textbf{DA} \\
			\midrule
			0 & \exval{$1.00000$}  & $1.00000$  & $1.00000$  & $1.00000$                   & $1.00000$                   & $1.00000$ \\
			1 & \exval{$5.14082$}  & $5.17810$  & $5.17810$  & \errval{$3.54265$}  & \errval{$3.54265$}  & $5.17810$ \\
			2 & \exval{$0.00000$}  & $0.00000$  & $0.00000$  & $0.00000$                   & $0.00000$                   & $0.00000$ \\
			3 & \exval{$-0.86102$} & $-0.84449$ & $-0.84449$ & \errval{$-0.59283$} & \errval{$-1.33362$} & $-0.84451$ \\
			\midrule
			\textbf{Objective} $\mathcal{J}$ & \exval{--} & $1.54{\times}10^{-3}$ & $1.54{\times}10^{-3}$ & $1.82{\times}10^{-1}$ & $2.03{\times}10^{-1}$ & $1.54{\times}10^{-3}$ \\
			\bottomrule
		\end{tabular}
	}

	\vspace{0.35cm}

	\resizebox{0.85\textwidth}{!}{
		\begin{tabular}{@{}l >{\columncolor{exactshade}}c ccccc@{}}
			\toprule
			\textbf{Edge} & \exhead{Exact} & \textbf{Brute Force} & \textbf{CMR} & \textbf{Null-space} & \textbf{Random} & \textbf{DA} \\
			\midrule
			\multicolumn{7}{@{}l}{\textit{Gradient} ($g$)} \\
			0 \ ($0 \!\to\! 1$) & \exval{$4.48386$}  & $4.52423$  & $4.52423$  & \errval{$2.75329$}  & \errval{$2.75329$}  & $4.52423$ \\
			1 \ ($1 \!\to\! 2$) & \exval{$-5.56670$} & $-5.60707$ & $-5.60707$ & \errval{$-3.83613$} & \errval{$-3.83613$} & $-5.60707$ \\
			2 \ ($3 \!\to\! 2$) & \exval{$0.93235$}  & $0.91445$  & $0.91445$  & \errval{$0.64194$}  & \errval{$1.44410$}  & $0.91447$ \\
			3 \ ($0 \!\to\! 3$) & \exval{$-2.01520$} & $-1.99730$ & $-1.99730$ & \errval{$-1.72478$} & \errval{$-2.52694$} & $-1.99731$ \\
			4 \ ($0 \!\to\! 2$) & \exval{$-0.76569$} & $-0.76569$ & $-0.76569$ & $-0.76569$                  & $-0.76569$                  & $-0.76569$ \\
			\addlinespace[0.15cm]
			\multicolumn{7}{@{}l}{\textit{Flux} ($q$)} \\
			0 \ ($0 \!\to\! 1$) & \exval{$-1.85173$} & $-1.85620$ & $-1.85620$ & $-1.82837$ & $-1.82837$ & $-1.85620$ \\
			1 \ ($1 \!\to\! 2$) & \exval{$2.08337$}  & $2.07890$  & $2.07890$  & $2.10673$  & $2.10673$  & $2.07890$ \\
			2 \ ($3 \!\to\! 2$) & \exval{$-0.77511$} & $-0.76108$ & $-0.76108$ & $-0.76108$ & $-0.78892$ & $-0.76109$ \\
			3 \ ($0 \!\to\! 3$) & \exval{$1.19244$}  & $1.20647$  & $1.20647$  & $1.20647$  & $1.17863$  & $1.20646$ \\
			4 \ ($0 \!\to\! 2$) & \exval{$0.69359$}  & $0.69591$  & $0.69591$  & $0.69591$  & $0.69591$  & $0.69591$ \\
			\bottomrule
		\end{tabular}
	}
\end{table}

Finally, Fig.~\ref{fig:wheatstone_exp_da_nonlinear} reports the annealing mechanism for the turbulent dataset. It mirrors the mechanism already described in the laminar case (cf.~Fig.~\ref{fig:wheatstone_exp_da_linear}), and both are interpreted in the exact same manner.

\begin{figure}[h!]
	\centering
	\svgfig{1.0\columnwidth}{559.53774198}{wheatstone_exp_da_nonlinear}
	\caption{Mechanistic evolution of the DA solver on the noise-free turbulent \emph{Wheatstone bridge} ($\mathcal{N}_m = 90$). (a) Active data points per edge, $N_{\mathrm{eff}} = \exp\mathcal{H}$, versus inverse temperature $\beta$, showing the contraction from uniform assignment ($\beta \to 0$) to a single datum ($\beta \to \infty$). (b) Final membership weights $w_n$ at $\beta \approx 979$, illustrating data localization. (c--g) Weight evolution $w_n(\beta)$ for edges 0--4, respectively, where the diffuse probability mass concentrates onto the target datum (horizontal dotted lines) as $\beta$ increases.}
	\label{fig:wheatstone_exp_da_nonlinear}
\end{figure}

It remains to justify the terminal temperature and the slow schedule used above, which Fig.~\ref{fig:wheatstone_da_cooling_turb} addresses sequentially. Panel (a) sweeps $\beta_{\max}$ at a fixed slow cooling rate for four dataset densities $\mathcal{N}_m \in \{20, 40, 80, 160\}$. Every curve descends from a warm-temperature plateau onto the global optimum (dashed line) at a threshold that grows monotonically with the dataset density. The densities bracketing the $\mathcal{N}_m = 90$ value used above reach this optimum by $\beta_{\max} \sim 10^3$. This indicates that the terminal temperature of the main runs is cold enough to resolve the nearest datum on every edge.

Panel (b) addresses the cooling rate. The terminal temperature is now fixed colder than required ($\beta_{\max} = 10^5$) with a generous 500-iteration budget, ensuring that stagnation can never be blamed on the terminal temperature itself. The runs differ only in the number of iterations $J_{\mathrm{cool}}$ used to raise $\beta$ from $\beta_0 = 0.02$ to $\beta_{\max}$, after which $\beta$ is held fixed. The slow schedule ($J_{\mathrm{cool}} = 500$) successfully descends all the way to the global optimum. Faster schedules stagnate at progressively higher plateaus, specifically: $\mathcal{J}_{\text{da}} \approx 1.36 \times 10^{-2}$ at $J_{\mathrm{cool}} = 150$, $6.40 \times 10^{-2}$ at $50$, and $1.82 \times 10^{-1}$ at the fastest quench, $J_{\mathrm{cool}} = 15$.

\begin{figure}[h!]
	\centering
	\svgfig{1.0\columnwidth}{542.01370224}{wheatstone_da_cooling_turb}
	\caption{Annealing of the DA solver on the turbulent \emph{Wheatstone bridge}.
		(a) Ratio $\mathcal{J}_{\text{da}}/\mathcal{J}_{\text{bf}}$ versus terminal temperature
		$\beta_{\max}$ (at a fixed slow cooling rate) for four densities $\mathcal{N}_m$: each curve
		drops from a stagnation plateau onto the optimum (dashed line) at a threshold $\beta_{\max}$
		that grows with density. (b) Convergence histories for $\mathcal{N}_m = 90$ and
		$\beta_{\max}=10^5$, sweeping the cooling duration ($J_{\mathrm{cool}}$ iterations from $\beta_0$
		to $\beta_{\max}$): only the slow schedule ($J_{\mathrm{cool}} = 500$, dark violet) reaches the
		optimum, while the faster quenches (grey, $J_{\mathrm{cool}} = 150, 50, 15$) freeze onto
		progressively higher local minima, showing that the cooling \emph{rate} is decisive in the
		turbulent regime.}
	\label{fig:wheatstone_da_cooling_turb}
\end{figure}

\paragraph{$\bullet$ Influence of measurement noise.}

We now repeat the noise study of Sec.~\ref{sec:laminar_fluid_flow} on the turbulent manifold. The procedure uses the same $100$ realizations per pair $(\eta, \mathcal{N}_m)$. Every statistic is again computed over the realizations solved admissibly by all methods, with the Brute Force serving as the per-realization lower bound. The aggregated results are summarized in Figs.~\ref{fig:wheatstone_conv_nonlinear_gaussian} and \ref{fig:wheatstone_conv_nonlinear_uniform} for Gaussian and Uniform noise, respectively.

\begin{figure}[h!]
	\centering
	\svgfig{1.0\columnwidth}{481.23544222}{wheatstone_conv_nonlinear_gaussian}
	\caption{Statistical performance comparison of solver approaches for the \emph{Wheatstone bridge} network under turbulent flow conditions and Gaussian noise perturbations. (a) Median objective function value $\mathcal{J}$ with interquartile range, computed over the realizations solved admissibly by every method, as data quality improves. (b) Empirical cumulative distribution functions (ECDFs) of $\mathcal{J}$ over the same realizations, aggregated across all problem configurations. (c)--(f) Representative examples of the synthetic turbulent datasets at increasing data quality: (c) $\eta = 0.05$, $\mathcal{N}_m = 20$; (d) $\eta = 0.025$, $\mathcal{N}_m = 40$; (e) $\eta = 0.0125$, $\mathcal{N}_m = 80$; (f) $\eta = 0.00625$, $\mathcal{N}_m = 160$.}
	\label{fig:wheatstone_conv_nonlinear_gaussian}
\end{figure}

\begin{figure}[h!]
	\centering
	\svgfig{1.0\columnwidth}{476.77943388}{wheatstone_conv_nonlinear_uniform}
	\caption{Statistical performance comparison of solver approaches for the \emph{Wheatstone bridge} network under turbulent flow conditions and Uniform noise perturbations. (a) Median objective function value $\mathcal{J}$ with interquartile range, computed over the realizations solved admissibly by every method, as data quality improves. (b) Empirical cumulative distribution functions (ECDFs) of $\mathcal{J}$ over the same realizations, aggregated across all problem configurations. (c)--(f) Representative examples of the synthetic turbulent datasets at increasing data quality: (c) $\eta = 0.05$, $\mathcal{N}_m = 20$; (d) $\eta = 0.025$, $\mathcal{N}_m = 40$; (e) $\eta = 0.0125$, $\mathcal{N}_m = 80$; (f) $\eta = 0.00625$, $\mathcal{N}_m = 160$.}
	\label{fig:wheatstone_conv_nonlinear_uniform}
\end{figure}

Panels (a) reveal a fundamental difference from the laminar regime. In the laminar case, the Null-space and Random strategies still improved monotonically as data quality rose, staying within the same order of magnitude as the baseline. Here, they improve only marginally. Their median decays at a slope of $\approx -1$ compared to the $\approx -1.5$ slope of the Brute Force reference. Consequently, even at the highest density and lowest noise, they remain more than $1.5$ orders of magnitude above the optimum. This demonstrates that the local trapping identified in the noise-free analysis persists under stochastic perturbation. Better data narrows the gap but does not close it, proving unable to compensate for a poor initialization when the manifold is strongly non-convex. In contrast, the ADM-CMR and DA medians track the Brute Force bound closely with slopes between $-1.5$ and $-1.7$. They maintain narrow interquartile ranges, just as in the laminar case. The only difference is the higher absolute error levels caused by the more challenging optimization landscape. Both noise distributions yield identical trends.

A second contrast concerns thermodynamic admissibility. In the laminar case, the Null-space and Random initializations often failed completely, yielding no admissible solution in up to half of the coarsest realizations. Here, however, almost every run is admissible. The only exception is the Random initialization at $\mathcal{N}_m = 20$, which is inadmissible in about $8$ to $10$\% of cases. Admissibility requires $g_k q_k \le 0$ on every edge, per Eq.~\eqref{eq:clausius_duhem}. Noise can only flip this product near the boundary where $g_k q_k = 0$. The mild laminar boundary conditions leave one arm nearly stagnant ($\mathrm{Re} \approx 43$), right next to this boundary, so small perturbations easily push it across. The stronger turbulent conditions keep every edge at $\mathrm{Re} \sim 10^{4}$, where the same noise cannot possibly change the sign. Therefore, the poor showing of the Null-space and Random strategies is not a thermodynamic failure but an optimization one: they settle on a valid but strongly suboptimal local minimum.

The ECDFs in Figs.~\ref{fig:wheatstone_conv_nonlinear_gaussian}(b) and \ref{fig:wheatstone_conv_nonlinear_uniform}(b) show the same split distributionally. The Brute Force, ADM-CMR, and DA curves overlap in the low-objective region, with the Brute Force bounding the others from the left. Meanwhile, the Null-space and Random strategies form a distinct cluster roughly two orders of magnitude to the right, centered between $\mathcal{J} \sim 10^{-1}$ and $10^{0}$. In the laminar ECDFs, all four strategies overlapped within about a single order of magnitude. The stark separation seen here is the direct signature of the non-convexity.

The previous sweep varies resolution and noise together. To isolate the noise effect, we fix an intermediate dataset of $\mathcal{N}_m = 90$ data pairs and sweep the noise level from $\eta = 0.2$ down to $\eta = 0$. We test 100 realizations per level for both distributions, reporting accuracy and admissibility side by side in Fig.~\ref{fig:wheatstone_noise_level} (omitting the Brute Force reference). We must note a caveat on how the medians of this figure are formed. Unlike the earlier sweeps, they are taken over the realizations \emph{each specific} solver solves admissibly, rather than over those solved admissibly by all of them. Because the Null-space and Random initializations fail most often at the highest noise levels, their curves in those regions are conditioned only on the easier realizations. Consequently, the performance gap reported below is likely understated.

\begin{figure}[h!]
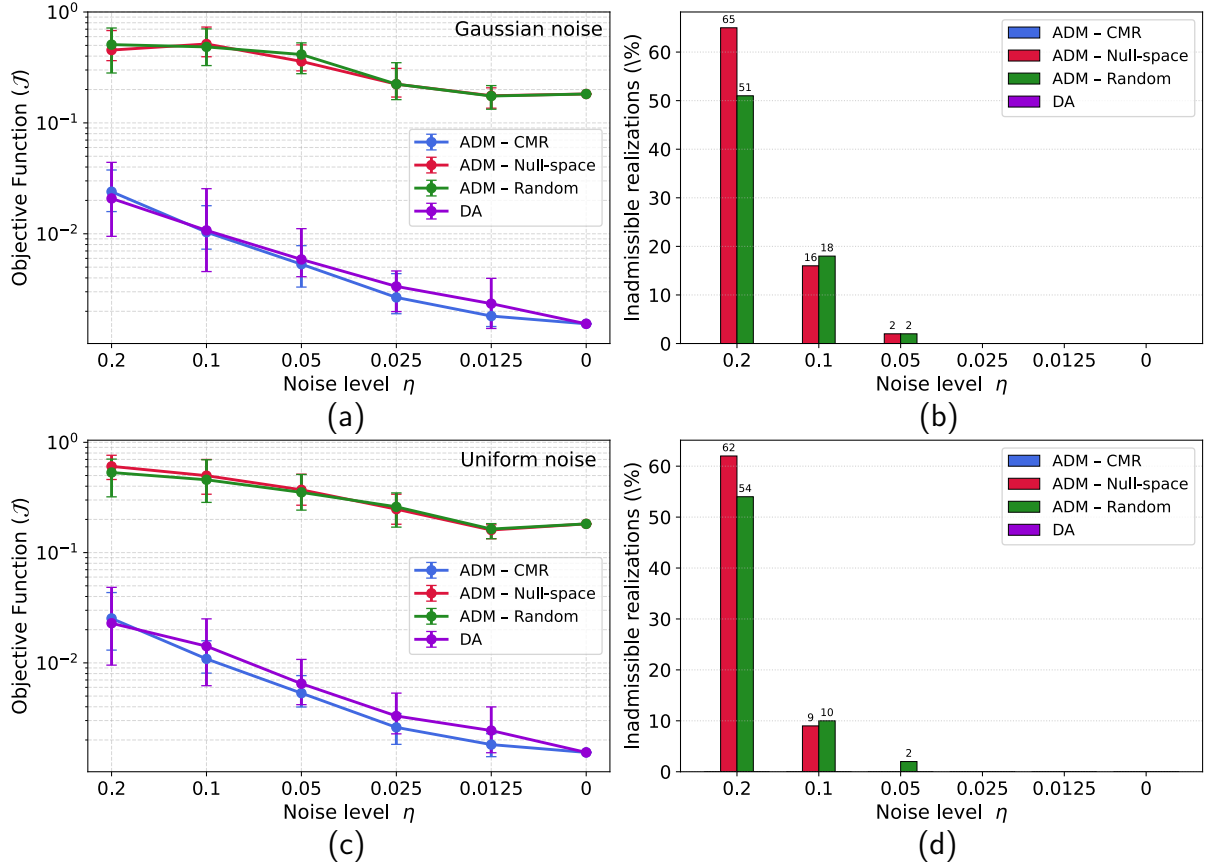

	\centering
	\svgfig{1.0\columnwidth}{452.01821953}{wheatstone_noise_level}
	\caption{Noise-level sensitivity of the solvers on the turbulent \emph{Wheatstone bridge} at a fixed dataset size ($\mathcal{N}_m = 90$), as the noise level $\eta$ decreases from $0.2$ to $0$ (left to right). Panels (a),(b) correspond to Gaussian noise and (c),(d) to Uniform noise. Panels (a),(c) report the median objective function $\mathcal{J}$ with its interquartile range over 100 realizations, computed over the realizations each solver solves admissibly; panels (b),(d) report the fraction of realizations whose final solution is thermodynamically inadmissible, as a bar chart per solver. The Brute Force reference is omitted.}
	\label{fig:wheatstone_noise_level}
\end{figure}

The ADM-CMR and DA solvers degrade smoothly and almost identically. Their median objective rises with the noise level, while their solutions remain thermodynamically admissible at every level tested. In contrast, the Null-space and Random initializations remain trapped between one and two orders of magnitude above the well-initialized solvers. Their objective values are essentially insensitive to $\eta$. Furthermore, beyond a moderate noise threshold ($\eta \approx 0.05$), they begin to return thermodynamically inadmissible solutions. This inadmissible fraction climbs steeply to between $50$ and $65$\% at $\eta = 0.2$. Thus, in the non-convex turbulent regime, a poor initialization is doubly penalized under heavy noise: it not only stagnates at a high objective but also increasingly violates the second law. Conversely, the initialization-free DA matches the robustness and accuracy of the best-initialized ADM across the entire noise range.

\paragraph{$\bullet$ Influence of the reconstruction effect.}

The CMR has so far used the linear least-squares fit in the laminar regime and the $C^1$-continuous Hermite projection in the turbulent one. Both successfully guided the ADM to the global optimum in their respective settings. However, it remains unclear whether the linear fit can survive a nonlinear dataset. We settle this question by running the two variants against each other on the turbulent bridge.

\begin{figure}[h!]
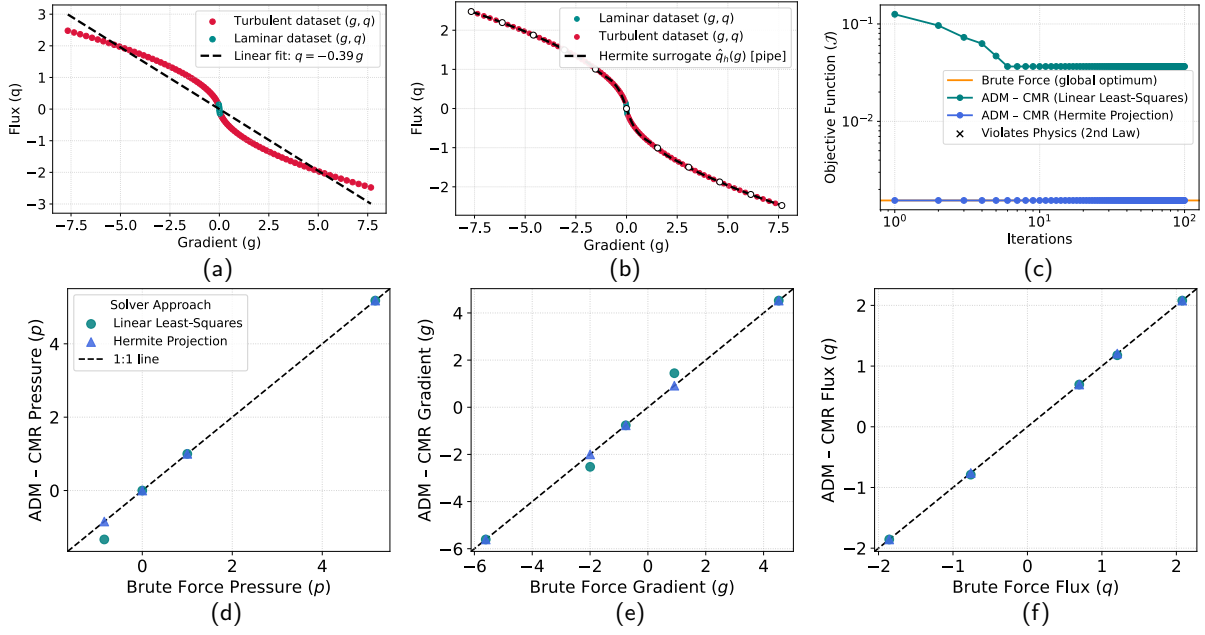

	\centering
	\svgfig{1.0\columnwidth}{597.15834214}{wheatstone_exp_cmr_comparison}
	\caption{Comparison of the two CMR initialization strategies for the turbulent \emph{Wheatstone bridge}. (a) Turbulent dataset with the linear least-squares fit $q = -0.39\,g$; the laminar dataset is shown for scale. (b) Turbulent dataset with the Hermite finite-element projection $\hat{q}_h(g)$, the points colored by flow regime and the open circles marking the nodes of the underlying finite-element mesh. (c) Objective $\mathcal{J}$ against the ADM iteration count (logarithmic axis, 100 iterations) for both CMR variants alongside the Brute Force global optimum; no iterate of either variant violates the second law. (d)--(f) Field parity against the Brute Force reference for nodal pressures ($p$), edge gradients ($g$) and edge fluxes ($q$), respectively.}
	\label{fig:wheatstone_exp_cmr_comparison}
\end{figure}

The linear fit $q = -0.39\,g$ captures only the average slope of the dataset (Fig.~\ref{fig:wheatstone_exp_cmr_comparison}(a)). Because the S-curve is steepest near the origin and flattens at the extremes, the linear fit underestimates the flux magnitude over most of the gradient range. This underestimation is most severe for $|g|$ between $1$ and $2$. It only overestimates the flux for $|g| \gtrsim 5$. In contrast, the Hermite surrogate $\hat{q}_h(g)$ accurately follows the curve throughout the entire range (Fig.~\ref{fig:wheatstone_exp_cmr_comparison}(b)).

The consequence for the ADM is immediate (Fig.~\ref{fig:wheatstone_exp_cmr_comparison}(c)). The Hermite initialization lands on the Brute Force optimum ($\mathcal{J}_{\text{cmr,hp}} = 1.54 \times 10^{-3}$) at the very first iteration. Conversely, the linear fit starts at $\mathcal{J}_{\text{cmr,lls}} \approx 1.26 \times 10^{-1}$ and stalls at $3.66 \times 10^{-2}$ from iteration 6 onward. It becomes trapped in a similar manner to the Null-space and Random initializations discussed in Sec.~\ref{sec:turbulent_fluid_flow}. However, it falls into a shallower basin because it still manages to recover the two high-gradient edges. Neither variant ever violates the second law. Therefore, what the linear fit suffers is a failure of optimization, not of thermodynamics. Ultimately, what decides the outcome is the fidelity of the surrogate employed to initialize the ADM-CMR. 

Table~\ref{tab:cmr_comparison} confirms this behavior on the recovered fields. The Hermite projection reproduces the optimum to numerical precision. The linear fit misses only Edges 2 and 3, whose operating gradients ($|g| \approx 0.9$ and $2.0$) fall squarely in the poorly fitted region. Edges 0 and 1 operate at $|g| > 4.5$ where the fit remains accurate, so both variants recover them. Both variants also recover Edge 4 because it bridges the two pressure-prescribed nodes, leaving its gradient fixed purely by the boundary conditions. The damage is also confined to two of the three fields. Its worst single entries are the pressure of Node 3 and the gradient of Edge 2, which are off by roughly $58\%$. Evaluated against the exact solution, the linear fit departs by $8.9\%$ in pressure and $9.7\%$ in gradient. This error lies a full order of magnitude above the $\approx 0.8\%$ floor established by the Brute Force optimum itself. Its flux field, by contrast, stays at that $0.6\%$ floor. It is pinned by mass conservation exactly as it was for the Null-space and Random initializations. The failure is thus attributable strictly to the surrogate and not to the sampling. This is especially telling because the DA reaches this exact same optimum using no surrogate at all.

\begin{table}[h!]
	\centering
	\caption{Recovered fields for the turbulent \emph{Wheatstone bridge} under the two CMR initialization strategies, compared against the Brute Force global optimum. Values highlighted in red indicate a relative absolute error greater than $10\%$ with respect to that optimum; the deviation norms in the last row of each block are likewise taken with respect to it. The \emph{Exact} column is the solution of the classical nonlinear problem, reproduced from Table~\ref{tab:turb_adm_initialization_results}. The gradient block doubles as the list of operating gradients, which highlights the correlation between fitting accuracy and solution error.}
	\label{tab:cmr_comparison}
	\vspace{0.2cm}
	\resizebox{0.72\textwidth}{!}{%
		\begin{tabular}{l >{\columncolor{exactshade}}c ccc}
			\toprule
			\textbf{Node} & \exhead{Exact} & \textbf{Brute Force} & \textbf{Hermite} & \textbf{Linear Least-Squares} \\
			\midrule
			0 & \exval{$\phantom{-}1.000$} & $\phantom{-}1.000$ & $\phantom{-}1.000$ & $\phantom{-}1.000$ \\
			1 & \exval{$\phantom{-}5.141$} & $\phantom{-}5.178$ & $\phantom{-}5.178$ & $\phantom{-}5.178$ \\
			2 & \exval{$\phantom{-}0.000$} & $\phantom{-}0.000$ & $\phantom{-}0.000$ & $\phantom{-}0.000$ \\
			3 & \exval{$-0.861$}           & $-0.844$           & $-0.844$           & \errval{$-1.334$} \\
			\midrule
			\textbf{Objective} $\mathcal{J}$ & \exval{--} & $1.54\times10^{-3}$ & $1.54\times10^{-3}$ & $3.66\times10^{-2}$ \\
			$\|\Delta\mathbf{p}\|_2$         & \exval{--} & --                  & $5.4\times10^{-8}$  & $4.9\times10^{-1}$ \\
			\bottomrule
		\end{tabular}
	}

	\vspace{0.35cm}

	\resizebox{0.72\textwidth}{!}{
		\begin{tabular}{l >{\columncolor{exactshade}}c ccc}
			\toprule
			\textbf{Edge} & \exhead{Exact} & \textbf{Brute Force} & \textbf{Hermite} & \textbf{Linear Least-Squares} \\
			\midrule
			\multicolumn{5}{@{}l}{\textit{Gradient} ($g$)} \\
			0\ ($0\!\to\!1$) & \exval{$\phantom{-}4.484$} & $\phantom{-}4.524$ & $\phantom{-}4.524$ & $\phantom{-}4.524$ \\
			1\ ($1\!\to\!2$) & \exval{$-5.567$}           & $-5.607$           & $-5.607$           & $-5.607$ \\
			2\ ($3\!\to\!2$) & \exval{$\phantom{-}0.932$} & $\phantom{-}0.914$ & $\phantom{-}0.914$ & \errval{$\phantom{-}1.444$} \\
			3\ ($0\!\to\!3$) & \exval{$-2.015$}           & $-1.997$           & $-1.997$           & \errval{$-2.527$} \\
			4\ ($0\!\to\!2$) & \exval{$-0.766$}           & $-0.766$           & $-0.766$           & $-0.766$ \\
			$\|\Delta\mathbf{g}\|_2$ & \exval{--} & -- & $3.5\times10^{-7}$ & $7.5\times10^{-1}$ \\
			\addlinespace[0.15cm]
			\multicolumn{5}{@{}l}{\textit{Flux} ($q$)} \\
			0\ ($0\!\to\!1$) & \exval{$-1.852$}           & $-1.856$           & $-1.856$           & $-1.856$ \\
			1\ ($1\!\to\!2$) & \exval{$\phantom{-}2.083$} & $\phantom{-}2.079$ & $\phantom{-}2.079$ & $\phantom{-}2.079$ \\
			2\ ($3\!\to\!2$) & \exval{$-0.775$}           & $-0.761$           & $-0.761$           & $-0.789$ \\
			3\ ($0\!\to\!3$) & \exval{$\phantom{-}1.192$} & $\phantom{-}1.206$ & $\phantom{-}1.206$ & $\phantom{-}1.179$ \\
			4\ ($0\!\to\!2$) & \exval{$\phantom{-}0.694$} & $\phantom{-}0.696$ & $\phantom{-}0.696$ & $\phantom{-}0.696$ \\
			$\|\Delta\mathbf{q}\|_2$ & \exval{--} & -- & $1.9\times10^{-7}$ & $3.9\times10^{-2}$ \\
			\bottomrule
		\end{tabular}
	}
\end{table}

\subsubsection{Network Scale and Topology}\label{sec:scalability}

We now assess how network size affects the DA and the ADM under its three initializations. An $M\times M$ Cartesian grid of square cells is laid on a fixed one-metre square domain. This grid contains $(M+1)^2$ nodes and $2M(M+1)$ edges. Refining the grid shrinks the individual pipe length to $L_{\mathrm{pipe}} = 1/M$ while leaving the overall domain size unchanged. With the boundary conditions and the inlet-to-outlet path length held fixed, this procedure cleanly isolates the effect of the network dimension. The flux splits unevenly across the network, causing the pipes to span a wide range of Reynolds numbers. Specifically, the peripheral pipes experience turbulent flow while the central ones remain laminar. Consequently, a single \emph{Churchill}-based dataset covering $\mathrm{Re} \in [0, \mathrm{Re}_{\max}]$ is used to describe every edge, regardless of the grid size $M$. The problem is solved for $\mathcal{N}_m \in \{50, 100, 200, 400, 800, 1600\}$ and $M \in \{2, 3, 4, 5\}$. The \emph{Churchill}-based dataset is sampled up to $\mathrm{Re}_{\max} = 3\times10^{4}$. Every solver receives a $500$-iteration budget, and the DA cools from $\beta_0 = 0.02$ to $\beta_{\max} = 10^{7}$. Table~\ref{tab:cartesian_network_parameters} details the geometry, the fluid properties (water at $20\,^\circ\mathrm{C}$), and the boundary conditions. Figure~\ref{fig:cartesian_grid_grids} illustrates the topology for the first two resolutions. Notably, this is the first benchmark where Brute Force enumeration is computationally impossible. At the finest dataset resolution, the number of candidate assignments reaches approximately $10^{38}$ on the $2 \times 2$ grid and $10^{192}$ on the $5 \times 5$ grid. Because a certified global optimum is unavailable, we can only compare the iterative solvers against one another.

\begin{figure}[h!]
	\centering
	\begin{minipage}[c]{0.525\textwidth}
		\centering
		\svgfig{1.0\columnwidth}{447.65472382}{cartesian_grid_grids}
		\caption{Schematic illustration of the Cartesian graph topology for the first two resolutions, (a) $M = 2$ and (b) $M = 3$, together with the nodes where the boundary conditions are imposed.}
		\label{fig:cartesian_grid_grids}
	\end{minipage}\hfill
	\begin{minipage}[b]{0.45\textwidth}
		\centering
		\captionof{table}{Geometric properties, fluid physical parameters, and boundary conditions for the Cartesian grid hydraulic network.}
		\label{tab:cartesian_network_parameters}
		\vspace{0.15cm}
		\resizebox{\linewidth}{!}{%
			\begin{tabular}{llcl}
				\toprule
				\textbf{Parameter} & \textbf{Symbol} & \textbf{Value} & \textbf{Unit} \\
				\midrule
				\multicolumn{4}{c}{\textit{Graph Topology \& Geometry}} \\
				\midrule
				Grid dimensions & $M \times M$ & 2 to 5 & -- \\
				Domain size     & $L \times H$ & $1.0 \times 1.0$ & $\mathrm{m}$ \\
				Pipe length     & $L_{\mathrm{pipe}}$ & $1/M$ & $\mathrm{m}$ \\
				Pipe diameter   & $D$   & $1.0 \times 10^{-2}$ & $\mathrm{m}$ \\
				Pipe roughness  & $\varepsilon$ & $1.0 \times 10^{-5}$ & $\mathrm{m}$ \\
				\midrule
				\multicolumn{4}{c}{\textit{Fluid Properties (Water at $20\,^\circ\mathrm{C}$)}} \\
				\midrule
				Fluid density     & $\rho$ & $998.0$ & $\mathrm{kg/m^3}$ \\
				Dynamic viscosity & $\mu_f$  & $1.0 \times 10^{-3}$ & $\mathrm{Pa \cdot s}$ \\
				\midrule
				\multicolumn{4}{c}{\textit{Boundary Conditions}} \\
				\midrule
				Inlet pressure (Bottom-left)  & $p_{\mathrm{in}}$  & $103.30 \times 10^3$ & $\mathrm{Pa}$ \\
				Outlet pressure (Top-right)   & $p_{\mathrm{out}}$ & $101.30 \times 10^3$ & $\mathrm{Pa}$ \\
				Flux (Bottom-right)           & $f_{\mathrm{in}}$  & $2.0 \times 10^{-4}$  & $\mathrm{m^3/s}$ \\
				Flux (Top-left)               & $f_{\mathrm{out}}$ & $-1.0 \times 10^{-4}$ & $\mathrm{m^3/s}$ \\
				\bottomrule
			\end{tabular}
		}
	\end{minipage}
\end{figure}

Figure~\ref{fig:cartesian_grid_network_size_influence} reports the mean per-edge objective $\bar{\mathcal{J}} = \mathcal{J}/n_e$ against the dataset size. This value is normalized so that grids of different sizes share a common axis. For every solver and every grid, $\bar{\mathcal{J}}$ decays monotonically as the dataset densifies. The primary difference between the methods is their baseline error level rather than their decay rate. The DA solver maintains a convergence slope of approximately $-2$ on every grid (specifically $-1.96$, $-1.97$, $-2.02$, and $-2.00$ for $M = 2$ through $5$; only the ADM-CMR slope is annotated in the figure). This demonstrates that its accuracy is essentially unaffected by the network dimension. In contrast, the convergence slope for the ADM-CMR method degrades from $-1.80$ at $M = 2$ down to $-1.43$ at $M = 5$. The Null-space and Random initializations decay at comparable rates, but from a persistently higher offset. They end roughly an order of magnitude above the DA at the finest dataset. They are also the least reliable, violating the second law at every resolution tested for $M = 3$ and $M = 4$. The crosses show that the ADM-CMR and the DA are not completely immune to failures either, predominantly occurring on the sparser datasets. Ultimately, the DA retains both the lowest objective and the most size-robust convergence, remarkably without requiring any specific initialization.

\begin{figure}[h!]
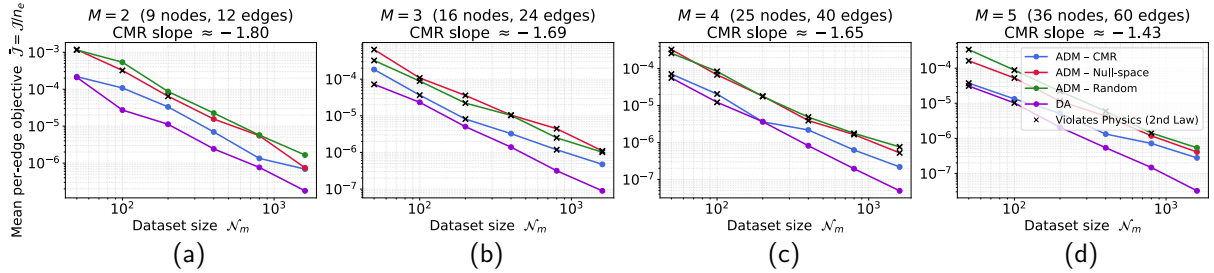

	\centering
	\svgfig{1.0\columnwidth}{547.94963698}{cartesian_grid_network_size_influence}
	\caption{Mean per-edge objective $\bar{\mathcal{J}} = \mathcal{J}/n_e$ versus the dataset size $\mathcal{N}_m$ for Cartesian grids of increasing dimension: (a) $M=2$ ($12$ edges), (b) $M=3$ ($24$ edges), (c) $M=4$ ($40$ edges) and (d) $M=5$ ($60$ edges). One curve per solver (ADM-CMR, ADM-Null-space, ADM-Random and DA); colored circles denote admissible solutions and black crosses ($\times$) second-law violations. Each panel is annotated with the ADM-CMR slope.}
	\label{fig:cartesian_grid_network_size_influence}
\end{figure}

Finally, Fig.~\ref{fig:cartesian_grid_da_field_regime} shows the DA solution at $\mathcal{N}_m = 1600$ for $M = 2$ and $M = 3$. Column (a) displays the nodal pressure, column (b) the edge flux magnitude, and column (c) the per-edge Reynolds number. The Reynolds number map explicitly illustrates the coexistence of the two flow regimes. Because the color scales are shared down each column, the two grids can be read against each other directly.

\begin{figure}[h!]
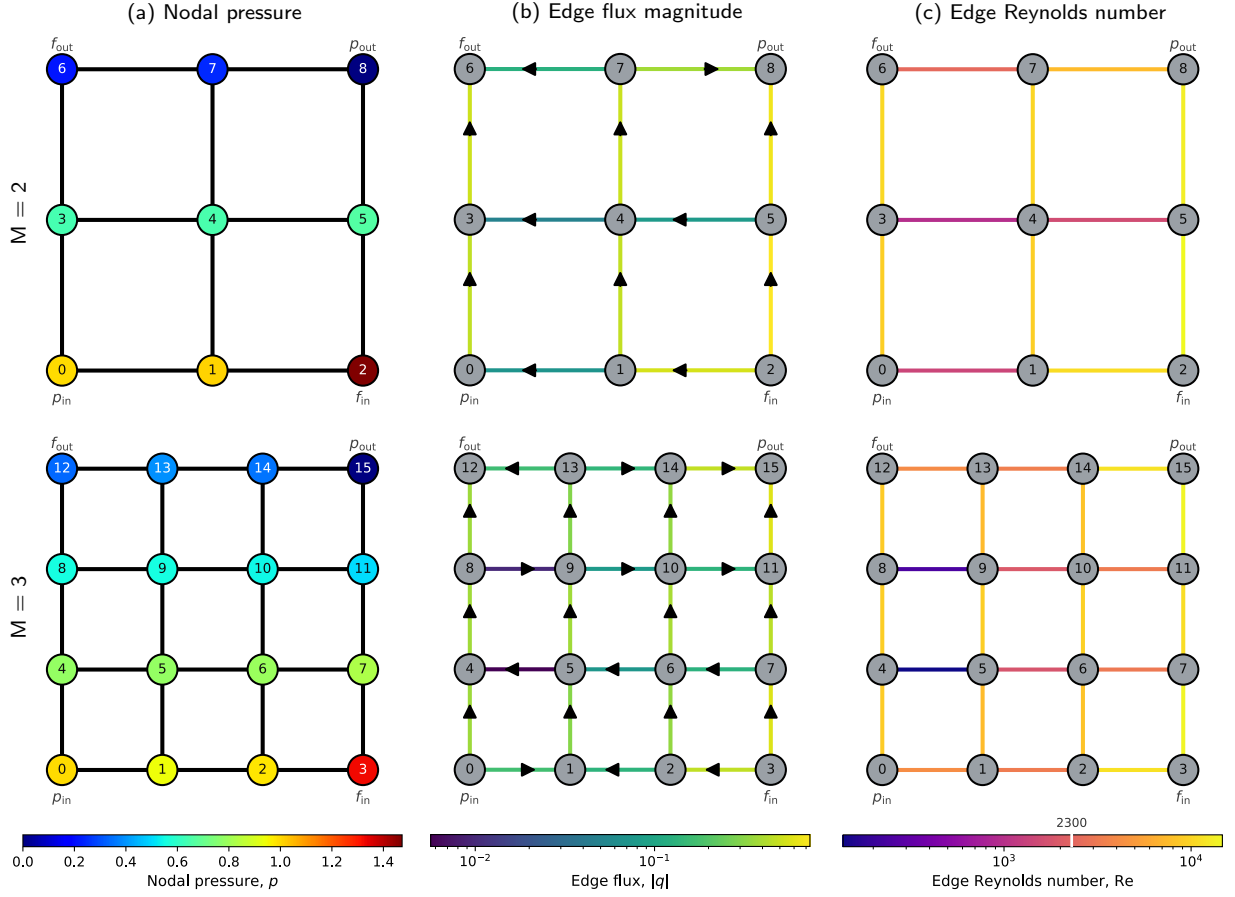

	\centering
	\svgfig{1.0\columnwidth}{700}{cartesian_grid_da_field_regime}
	\caption{DA solution on the Cartesian grids at $\mathcal{N}_m = 1600$, for $M = 2$ (top row) and $M = 3$ (bottom row). (a) Nodal pressure. (b) Edge flux magnitude on a logarithmic scale, the triangles giving the flow direction. (c) Per-edge Reynolds number on a logarithmic scale, with the laminar--turbulent transition $\mathrm{Re} = 2300$ marked on the colorbar. Nodes carry their index, and the four corners the boundary conditions of Table~\ref{tab:cartesian_network_parameters}. Color scales are shared down each column, so the two grids are directly comparable.}
	\label{fig:cartesian_grid_da_field_regime}
\end{figure}


\subsubsection{Networks with Lumped Elements: Elbows}\label{sec:lumped_elements}

Real hydraulic networks connect their ducts through various fittings, such as valves, bends, elbows, tees, and expansions. The flow separation and mixing within these components cause localized minor losses. A fitting is characterized experimentally by a loss coefficient $K_L \geq 0$ \cite{Crane2009}, giving a pressure drop
\begin{equation}
	\Delta p_{\text{fitting}} = K_L\,\frac{\rho}{2A^2}\,|q|\,q,
	\label{eq:minor_loss}
\end{equation}
with $|q|q$ enforcing the correct sign on flow reversal and $A$ taken constant through the fitting.

To incorporate a fitting into the graph, we model it as a short \emph{arm} sub-edge of length $L_{\text{arm}} = 2D$ with a uniform pressure gradient. This length is short relative to a standard pipe, yet long enough to keep the arm well-conditioned within the discrete numerical operators. The arm length is tied directly to the vessel calibre, ensuring that all arms meeting at a specific hub share the same diameter. Its resistance is
\begin{equation}
	R_{\text{fitting}} = \frac{8\rho}{\pi^2 D_k^4}\,\frac{K_{L}}{L_{\text{arm}}}\,|q_k|,
	\label{eq:fitting_resistance}
\end{equation}
which has the \emph{same} structure as the \emph{Darcy--Weisbach} resistance~\eqref{eq:hyd_resistance} under $f_k/D_k \mapsto K_L/L_{\text{arm}}$, so a fitting needs no machinery of its own. Each fitting node becomes a \emph{hub}. An arm is inserted on every incident edge, and its length is subtracted from the connected pipe to preserve the overall geometry (Fig.~\ref{fig:lumped_elements}). In this study, we consider only $90^{\circ}$ elbows. Because an elbow is symmetric, it is modeled using two equal arms. Each arm carries a loss coefficient of $K_{L}/2$ because they experience the same flux and their pressure drops add in series. The arm law is thus perfectly odd about the origin.

\begin{figure}[h!]
	\centering
	\begin{minipage}[c]{0.455\textwidth}
		\centering
		\svgfig{1.0\columnwidth}{290.42194431}{lumped_elements}
		\caption{Hub-and-arm decomposition of lumped elements. (a) Original elbow graph and (b) its mutated hub-and-arm form (two arms, $K_L/2$ each).}
		\label{fig:lumped_elements}
	\end{minipage}\hfill
	\begin{minipage}[c]{0.515\textwidth}
		\centering
		\captionof{table}{Geometric properties, fitting constants, fluid physical parameters, boundary conditions and solver settings for the square loop with $90^{\circ}$ elbows. The node and edge counts are those of the mutated graph, obtained by applying the hub-and-arm decomposition of Fig.~\ref{fig:lumped_elements} at each of the four corners; node indices follow the loop counter-clockwise from the inlet corner.}
		\label{tab:lumped_network_parameters}
		\vspace{0.15cm}
		\resizebox{\linewidth}{!}{%
		\begin{tabular}{llcl}
			\toprule
			\textbf{Parameter} & \textbf{Symbol} & \textbf{Value} & \textbf{Unit} \\
			\midrule
			\multicolumn{4}{c}{\textit{Graph Topology \& Geometry}} \\
			\midrule
			Loop side length                     & --                    & $1.0$                 & $\mathrm{m}$ \\
			Number of nodes (hubs $+$ arm nodes) & $n_v$                 & $12$                  & -- \\
			Number of edges (4 pipes $+$ 8 arms) & $n_e$                 & $12$                  & -- \\
			Arm length                           & $L_{\mathrm{arm}}$    & $2D = 2.0\times10^{-2}$ & $\mathrm{m}$ \\
			Pipe length, after shortening        & $L_{\mathrm{pipe}}$   & $0.96$                & $\mathrm{m}$ \\
			Pipe and arm diameter                & $D$                   & $1.0 \times 10^{-2}$  & $\mathrm{m}$ \\
			Pipe roughness                       & $\varepsilon$         & $1.0 \times 10^{-5}$  & $\mathrm{m}$ \\
			\midrule
			\multicolumn{4}{c}{\textit{Fitting (flanged $90^{\circ}$ elbow, Hooper two-$K$)}} \\
			\midrule
			Low-Reynolds constant                & $K_1$                 & $800$                 & -- \\
			Fully turbulent constant             & $K_{\infty}$          & $0.25$                & -- \\
			Per-arm loss coefficient, $\mathrm{Re}\!\to\!\infty$ & $K_L/2$ & $0.443$             & -- \\
			\midrule
			\multicolumn{4}{c}{\textit{Fluid Properties (Water at $20\,^\circ\mathrm{C}$)}} \\
			\midrule
			Fluid density                        & $\rho$                & $998.0$               & $\mathrm{kg/m^3}$ \\
			Dynamic viscosity                    & $\mu_f$               & $1.0 \times 10^{-3}$  & $\mathrm{Pa \cdot s}$ \\
			\midrule
			\multicolumn{4}{c}{\textit{Boundary Conditions}} \\
			\midrule
			Inlet pressure (Node 0)              & $p_{\mathrm{in}}$     & $105.30 \times 10^3$  & $\mathrm{Pa}$ \\
			Outlet pressure (Node 2)             & $p_{\mathrm{out}}$    & $101.30 \times 10^3$  & $\mathrm{Pa}$ \\
			Source flux (Node 1)                 & $f_{\mathrm{in}}$     & $2.0 \times 10^{-4}$  & $\mathrm{m^3/s}$ \\
			Sink flux (Node 3)                   & $f_{\mathrm{out}}$    & $-2.0 \times 10^{-4}$ & $\mathrm{m^3/s}$ \\
			\midrule
			\multicolumn{4}{c}{\textit{Dataset and Solver Settings}} \\
			\midrule
			Samples per dataset key              & $\mathcal{N}_m$       & $200$                 & -- \\
			Sampling Reynolds bound              & $\mathrm{Re}_{\max}$  & $7.0 \times 10^{4}$   & -- \\
			Iteration budget                     & $J_{\max}$            & $900$                 & -- \\
			Annealing range                      & $\beta_0 \to \beta_{\max}$ & $0.02 \to 10^{6}$ & -- \\
			\bottomrule
		\end{tabular}}
	\end{minipage}
\end{figure}

The loss coefficient itself follows the two-constant correlation of Hooper \cite{Hooper1981},
\begin{equation}
	K_L(D, \mathrm{Re}) = \frac{K_1}{\mathrm{Re}} + K_\infty\left(1 + \frac{1\,\text{in}}{D}\right),
	\label{eq:hooper_2k}
\end{equation}
closed-form in the inside diameter and free of range restrictions, with $K_\infty$ the coefficient of a large fitting in the fully turbulent limit and $K_1$ its value at $\mathrm{Re} = 1$. For the flanged $90^{\circ}$ elbow adopted here, $K_1 = 800$ and $K_\infty = 0.25$. This yields $K_L = 0.885$ at $D = 1\,\mathrm{cm}$ in the turbulent limit. Substituting~\eqref{eq:hooper_2k} with $\mathrm{Re} = 4\rho|q_k|/(\pi D_k \mu_f)$ separates the flux dependence exactly,
\begin{equation}
	R_{\text{fitting}}(q_k) = \underbrace{\frac{2\mu_f K_1}{\pi D_k^3 L_{\text{arm}}}}_{\text{viscous, constant}}
	\;+\; \underbrace{\frac{8\rho\,K_\infty\!\left(1 + 1\,\text{in}/D_k\right)}{\pi^2 D_k^4 L_{\text{arm}}}\,|q_k|}_{\text{inertial}, \;\propto\, |q_k|},
	\label{eq:fitting_resistance_hooper}
\end{equation}
The arm law $g_k = -R_{\text{fitting}}(q_k)\,q_k$ represents the quadratic term from Eq.~\eqref{eq:minor_loss} plus a linear viscous correction that a simple constant-$K$ table omits. This forms the entire model. Arms sharing a specific key (type and $D$) draw from a single dataset. We generate this dataset via \emph{Local Constitutive Sampling} exactly as for the pipes, using the same dimensionless scales. Figure~\ref{fig:results_lumped_elements}(b) displays the pipe and elbow datasets alongside their Hermite surrogates.

The benchmark itself is a square loop with a side length of $1\,\mathrm{m}$. Its four pipes meet at $90^{\circ}$ corners, meaning every corner acts as a fitting. The hub-and-arm expansion transforms the original four-node, four-edge graph into an expanded network of $12$ nodes and $12$ edges. This expanded graph consists of four pipes (each shortened by one arm at either end) and eight fitting arms. Pressures are prescribed at two opposite corners, and nodal fluxes are set at the remaining two, forcing the two branches of the loop to carry different fluxes. Table~\ref{tab:lumped_network_parameters} collects the geometry, the elbow constants, the fluid properties, the boundary conditions, and the dataset and annealing settings used throughout this section.

This is also the first benchmark where the per-dataset metric weight of Eq.~\eqref{eq:c_choice} deserves attention. The two dataset keys populate the phase space with substantially different aspect ratios. Measured on the dimensionless clouds, the ratio $c^{(\mu)} = \sigma_q^{(\mu)}/\sigma_g^{(\mu)}$ equals $0.534$ for the standard pipe and just $0.056$ for the elbow arm. These values sit nearly a decade apart because the short, high-resistance arm spreads over a much wider gradient range at the same flux. If we repeat the experiment using a plain isotropic metric $\mathbf{C} = \mathbf{I}$, the ADM-CMR and the DA solvers no longer agree. The DA solver settles at an objective $31$ times higher than the ADM-CMR. When using the key-dependent weights, both methods successfully reach the exact same solution.

The Brute Force method is not evaluated in Fig.~\ref{fig:results_lumped_elements}(a). Enumeration would require checking $N_{\text{comb}} = \mathcal{N}_{\text{pipe}}^{4}\,\mathcal{N}_{\text{arm}}^{8} = 4.1\times10^{27}$ assignments, so no certified global optimum is available. The reference value becomes the best admissible objective reached by any solver, which itself acts as an upper bound on the true minimum.

The ADM--Random method becomes completely trapped in a state that violates the second law, meaning it is never admissible. Conversely, the ADM--Null-space reaches a thermodynamically consistent but highly suboptimal solution of $\mathcal{J}_{\text{ns}} = 2.098\times10^{-2}$. The ADM-CMR and the DA methods converge to the exact same objective: $\mathcal{J}_{\text{cmr}} = \mathcal{J}_{\text{da}} = 1.4747\times10^{-4}$. Their recovered fields agree to within tight numerical tolerances: $\|\Delta \uH\| = 4.6\times10^{-9}$, $\|\Delta \eH\| = 1.4\times10^{-9}$, and $\|\Delta \sH\| = 1.6\times10^{-10}$. Having two completely independent solvers land on the same assignment provides the strongest possible evidence of success when a certified optimum is absent. When split by data key, the eight arms carry $76\%$ of the total objective ($1.13\times10^{-4}$, compared to just $3.48\times10^{-5}$ on the four pipes). Thus, the fittings dominate the residual rather than the pipes. Their operating states anticipate this: over the same flux range $|q| \leq 0.499$, the arms sustain gradients $|g|$ up to $2.887$, while the pipes only reach $0.354$. Finally, Figures~\ref{fig:results_lumped_elements}(c)-(e) illustrate the pressure and flux fields for the DA, ADM-CMR, and ADM--Null-space solutions, respectively.

\begin{figure}[h!]
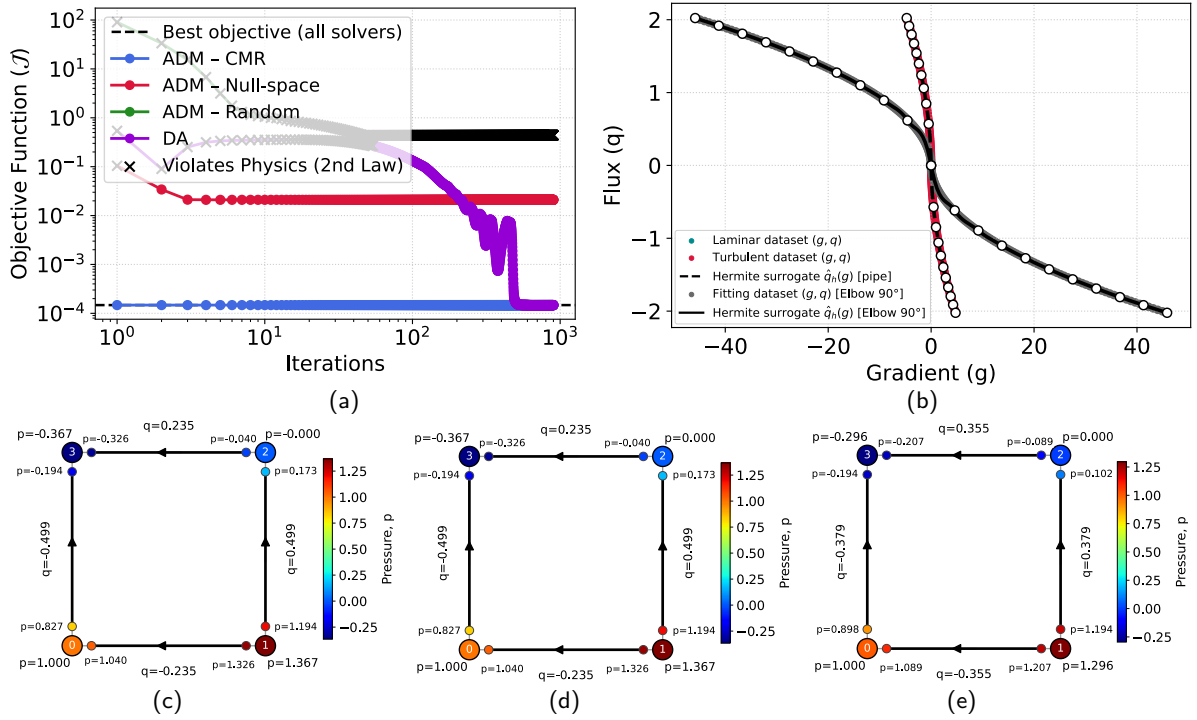

	\centering
	\svgfig{1.0\columnwidth}{596.22928169}{results_lumped_elements}
	\caption{(a) Evolution of the objective function $\mathcal{J}$ across the iterations. (b) Laminar and turbulent datasets for the pipe elements together with the Hermite finite-element projection $\hat{q}_{h}(g)$; Fitting dataset for the elbow $90^\circ$ along with its Hermite-finite element projection. (c)-(e) Nodewise pressure field and edgewise flux field solutions for the DA, ADM-CMR and ADM--Null-space, respectively. ADM-Random is not shown, since it never reaches a thermodynamically admissible solution.}
	\label{fig:results_lumped_elements}
\end{figure}

Figure~\ref{fig:lumped_elements_da_square_elbow} details the annealing process on the square loop. Every arm and pipe successfully collapses to a single assigned datum by the end of the cooling schedule, yielding $N_{\mathrm{eff}} = 1.000$ on all twelve edges. The individual assignment weights at the coldest iteration confirm that the DA recovers the exact same hard assignment as the ADM method.

However, the two dataset keys do not harden simultaneously. The fitting arms collapse into a hard assignment at $\beta \approx 8.8\times10^{2}$, while the pipes only collapse at $\beta \approx 8.5\times10^{3}$, which is an order of magnitude colder. The objective function stops changing at $\beta \approx 8.2\times10^{3}$, coinciding directly with the hardening of the pipe data. Therefore, the minor oscillation of $\mathcal{J}_{\text{da}}$ during the middle and late iterations reflects the pipe assignments redistributing while the arms are already completely frozen. This two-timescale hardening behavior is determined by the per-key nearest-neighbor gaps described in Proposition~\ref{prop:dda_limits}. The dense, high-resistance arm data clouds resolve their assignments at much warmer temperatures than the pipe data. Therefore, a single global cooling schedule must reach a terminal temperature cold enough to accommodate the slowest-hardening dataset key in the entire network.

\begin{figure}[h!]
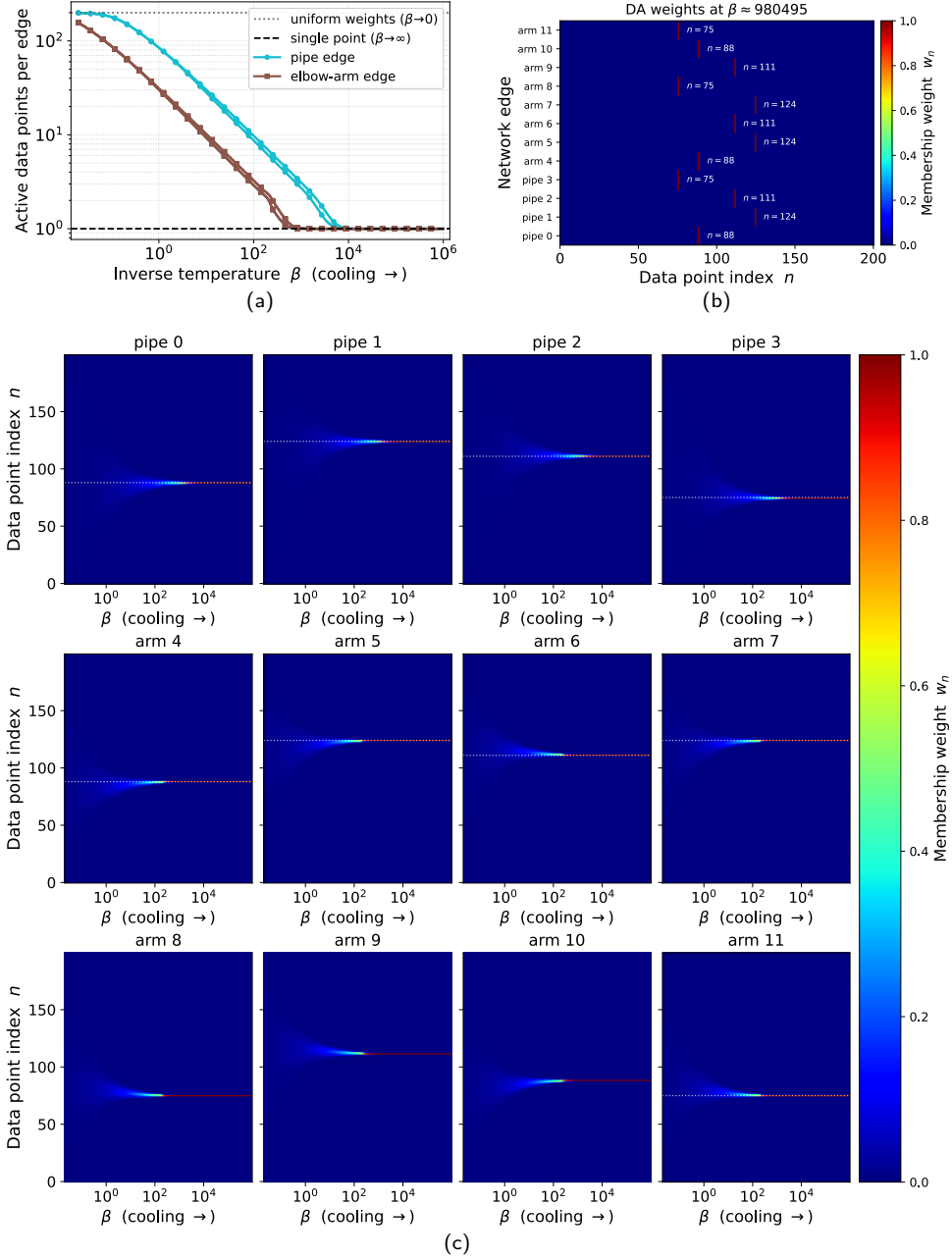

	\centering
	\svgfig{0.8\columnwidth}{562.52975824}{lumped_elements_da_square_elbow}
	\caption{Mechanistic evolution of the DA solver on the square elbow lumped elements. (a) Active data points per edge, $N_{\mathrm{eff}} = \exp\mathcal{H}$, versus inverse temperature $\beta$, showing the contraction from uniform assignment ($\beta \to 0$) to a single datum ($\beta \to \infty$). (b) Final membership weights $w_n$ at $\beta \approx 980495$, illustrating data localization. (c) Weight evolution $w_n(\beta)$ for pipes 0--3 and arms 4--11, where the diffuse probability mass concentrates onto the target datum (horizontal dotted lines) as $\beta$ increases.}
	\label{fig:lumped_elements_da_square_elbow}
\end{figure}

Finally, we test the benchmark on noisy data. We add Gaussian noise (standard deviation $\eta = 0.05$) to both components of every dataset point (Fig.~\ref{fig:results_lumped_elements_noisy}). The behavior observed on noise-free data persists. ADM-Random remains inadmissible throughout the iterations. ADM-Null-space stays trapped in a local minimum ($\mathcal{J}_{\text{ns}} = 5.389\times10^{-2}$). The ADM-CMR solver also degrades because its constitutive manifold reconstruction fits the corrupted cloud. It settles at $\mathcal{J}_{\text{cmr}} = 1.605\times10^{-2}$. This is approximately twice the objective reached by the DA method ($\mathcal{J}_{\text{da}} = 8.115\times10^{-3}$), which attains the best objective among the four solvers. Nevertheless, the noise penalizes all methods. The perturbed samples no longer lie on the manifold, creating a residual distance that sets a floor for the objective function. As a result, the best noisy objective is higher than the noise-free reference ($\mathcal{J} = 1.4747\times10^{-4}$).

\begin{figure}[h!]
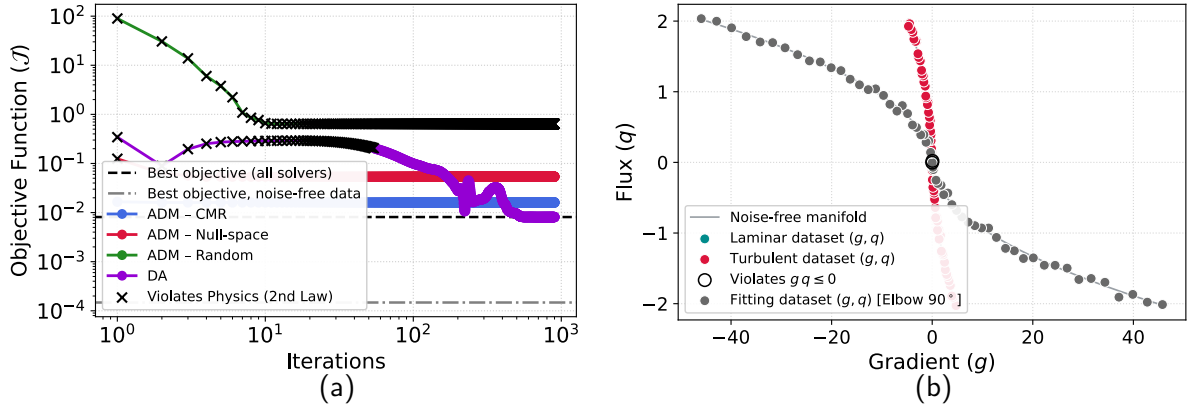

	\centering
	\svgfig{1.0\columnwidth}{500}{lumped_elements_noisy}
	\caption{Square loop with 90° elbows under Gaussian noise ($\eta = 0.05$). (a) Evolution of the objective function $\mathcal{J}$ across the iterations; the dashed line marks the best objective reached on the noisy data and the dash-dotted line the noise-free reference of Fig.~\ref{fig:results_lumped_elements}(a). (b) Perturbed pipe and elbow-arm datasets against their noise-free manifolds; circled markers are the samples driven across $g\,q \leq 0$ by the noise.}
	\label{fig:results_lumped_elements_noisy}
\end{figure}


\subsection{Application to Hemodynamics}\label{sec:hemodynamics}

Hemodynamics is a natural application for data-driven approaches because \textit{in vivo} microcirculation measurements of pressure and flow are notoriously difficult to obtain. Moreover, non-Newtonian viscosity representations like the \emph{Carreau--Yasuda} model are empirical fits based on experimental blood rheology data. We apply our framework to a small arteriovenous bed modeled as an arterial tree rooted at a single feeding arteriole. This tree bifurcates over four generations following \emph{Murray's law} \cite{Murray1926}. We adopt a branching exponent of $k = 3.0$ for both the arterial and venous trees.

\begin{equation}
	D_{\text{parent}}^k = D_{\text{major}}^k + D_{\text{minor}}^k, \qquad k = 3.0
\end{equation}

We assume a continuous asymmetry factor of $\alpha = 0.85$ between the two daughter branches at every bifurcation down to the sixteen terminal arterioles. Substituting $D_{\text{minor}} = \alpha\,D_{\text{major}}$ yields the branching relations directly.

\begin{equation}
	D_{\text{major}} = \frac{D_{\text{parent}}}{(1 + \alpha^k)^{1/k}}, \quad D_{\text{minor}} = \alpha\,D_{\text{major}}
\end{equation}

We construct a mirrored venous tree rooted at a single collecting venule. This tree branches outward to sixteen terminal venules corresponding directly to the terminal arterioles. A \emph{capillary fan} containing $28$ parallel paths bridges each arteriole--venule leaf pair. Each path splits into two series sub-edges at a shared midpoint node. The resulting graph (Fig.~\ref{fig:hemodynamics_network}) comprises $512$ nodes and $958$ edges, including $896$ capillary sub-edges.

\begin{figure}[h!]
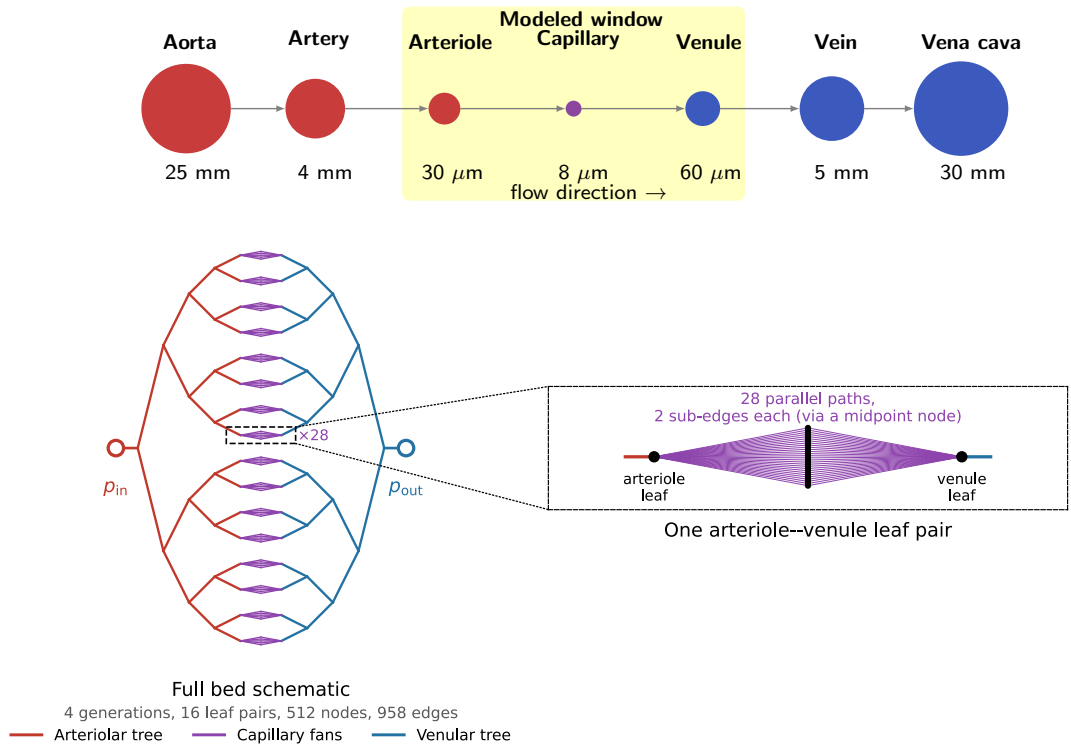

	\centering
	\svgfig{0.975\columnwidth}{702.99230149}{hemodynamics_network}
	\caption{Arteriovenous bed used in the hemodynamic experiment (not to scale). \emph{Top:} The modeled window (highlighted) spans the arteriole--capillary--venule segment of the circulation, with the larger upstream and downstream vessels drawn for context. \emph{Bottom:} The resulting network graph.}
	\label{fig:hemodynamics_network}
\end{figure}

The two trees are rooted at a feeding arteriole ($D = 30\,\mu\mathrm{m}$) and a collecting venule ($D = 60\,\mu\mathrm{m}$) rather than the aorta and vena cava. We set the boundary pressures accordingly: an inlet pressure of $p_{\mathrm{in}} = 60~\mathrm{mmHg}$ and an outlet of $p_{\mathrm{out}} = 15~\mathrm{mmHg}$. These values match those measured by Zweifach and Lipowsky \cite{Zweifach1977} across comparable microvascular modules, yielding a total pressure drop of $45~\mathrm{mmHg}$. Whole blood exhibits shear-thinning properties; red blood cells aggregate into rouleaux at low shear rates and disperse as the shear rate increases. We apply the \emph{Carreau--Yasuda} non-Newtonian model to the capillary edges using the whole-blood parameters established by Cho and Kensey \cite{ChoKensey1991, Gijsen1999}. The arterioles and venules use a simpler Newtonian closure. Their higher wall shear rates keep the apparent fluid viscosity near the plateau $\mu_{f,\infty}$, making the two models virtually indistinguishable. 

Both models provide continuum descriptions that serve as an idealization at the capillary scale. Because the capillary diameter $D_{\text{cap}}$ is comparable to a single red blood cell, the actual \emph{in vivo} resistance below approximately $15\,\mu\mathrm{m}$ is substantially higher than predictions from viscosity laws calibrated in standard glass tubes \cite{FahraeusLindqvist1931, Pries1990}. Consequently, the closure acts only as a generator of synthetic data pairs (as defined in Remark~\ref{rem:role_closures}) rather than a definitive medical claim about human capillary rheology. Table~\ref{tab:hemodynamic_parameters} summarizes the geometry, fluid parameters, and boundary conditions.

Before applying the data-driven solution, we characterize the classical problem. The physical system is fundamentally nonlinear because the capillary law depends on the edge's flux through $\mu_{f,\mathrm{app}}(q_k)$. The Picard scheme outlined in Algorithm~\ref{alg:picard_solver} converges in $12$ iterations. The exact solution $\mathbf{z}_{\mathrm{ex}} = (\eH_{\mathrm{ex}}, \sH_{\mathrm{ex}})$ provides a gold-standard reference for the recovered fields and maps the capillary operating points onto the constitutive curve in Fig.~\ref{fig:hemodynamics_cy_curve}. The built-in asymmetry of the branching trees causes the capillary wall shear rate to span $\dot\gamma \approx 183$ to $612~\mathrm{s}^{-1}$, with a mean of $349~\mathrm{s}^{-1}$. This places roughly $40\%$ of the sub-edges within the highly shear-thinning range. The mean apparent viscosity is $3.79\times10^{-3}~\mathrm{Pa\,s}$. This value is $14.8$ times lower than the zero-shear limit $\mu_{f,0}$ but remains $10.0\%$ above the high-shear plateau $\mu_{f,\infty}$.

\begin{figure}[h!]
	\centering
	\begin{minipage}[c]{0.46\textwidth}
		\centering
		\svgfig{1.0\columnwidth}{457.15264893}{hemodynamics_cy_curve}
		\caption{Apparent viscosity of the \emph{Carreau--Yasuda} closure, with the zero-shear and high-shear plateaus indicated and the shaded band marking the range over which the departure from the high-shear plateau exceeds $10\%$. The capillary operating points of the exact nonlinear solution are superimposed.}
		\label{fig:hemodynamics_cy_curve}
	\end{minipage}\hfill
	\begin{minipage}[b]{0.51\textwidth}
		\centering
		\captionof{table}{Geometric properties, fluid and Carreau--Yasuda parameters, and boundary conditions for the arteriovenous tree.}
		\label{tab:hemodynamic_parameters}
		\vspace{0.15cm}
		\resizebox{\linewidth}{!}{%
		\begin{tabular}{llcl}
			\toprule
			\textbf{Parameter} & \textbf{Symbol} & \textbf{Value} & \textbf{Unit} \\
			\midrule
			\multicolumn{4}{c}{\textit{Topology \& Geometry}} \\
			\midrule
			Bifurcation generations & -- & 4 & -- \\
			Segment length-to-diameter ratio & -- & 20 & -- \\
			Root arteriole diameter & -- & 30 & $\mu\mathrm{m}$ \\
			Root venule diameter & -- & 60 & $\mu\mathrm{m}$ \\
			Murray exponent (both trees) & $k$ & 3.0 & -- \\
			Daughter-branch asymmetry & $\alpha$ & 0.85 & -- \\
			Terminal arteriole diameters & -- & 8.3--15.8 & $\mu\mathrm{m}$ \\
			Terminal venule diameters & -- & 16.5--31.7 & $\mu\mathrm{m}$ \\
			Capillary diameter & $D_{\text{cap}}$ & 8 & $\mu\mathrm{m}$ \\
			Capillary path length & $L_{\text{cap}}$ & 1.2 & $\mathrm{mm}$ \\
			Capillary paths per leaf pair & -- & 28 & -- \\
			Branch half-angles (gen. 1--4) & -- & 35,\ 28,\ 22,\ 16 & $^{\circ}$ \\
			\midrule
			\multicolumn{4}{c}{\textit{Fluid Properties (Blood)}} \\
			\midrule
			Blood density & $\rho$ & 1060.0 & $\mathrm{kg/m^3}$ \\
			Newtonian viscosity (arterioles and venules) & $\mu_f$ & $3.5\times10^{-3}$ & $\mathrm{Pa \cdot s}$ \\
			\midrule
			\multicolumn{4}{c}{\textit{Carreau--Yasuda Parameters (capillary edges)}} \\
			\midrule
			Zero-shear viscosity & $\mu_{f,0}$ & $5.6\times10^{-2}$ & $\mathrm{Pa \cdot s}$ \\
			High-shear asymptotic viscosity & $\mu_{f,\infty}$ & $3.45\times10^{-3}$ & $\mathrm{Pa \cdot s}$ \\
			Relaxation time & $\lambda$ & 1.902 & $\mathrm{s}$ \\
			Transition sharpness & $a$ & 1.25 & -- \\
			Power-law index & $n$ & 0.22 & -- \\
			\midrule
			\multicolumn{4}{c}{\textit{Boundary Conditions}} \\
			\midrule
			Root arteriole pressure & $p_{\text{in}}$ & 60 & $\mathrm{mmHg}$ \\
			Root venule pressure & $p_{\text{out}}$ & 15 & $\mathrm{mmHg}$ \\
			\bottomrule
		\end{tabular}}
	\end{minipage}
\end{figure}

The asymmetric bifurcations give almost every generation a different pair of daughter diameters, resulting in $31$ dataset keys. This comprises $30$ macro-vessel keys (one per calibre, with $15$ per tree) and a single capillary key shared by all $896$ capillary sub-edges. We sample the macro-vessel keys with $150$ points each and the more complex capillary key with $600$ points. This yields a pooled dataset of $|\mathcal{D}| = 5100$ dimensionless $(g,q)$ pairs scaled by the harmonic mean. We size each key's flux range dynamically based on the operating flux of its edges, adding a $2\times$ headroom margin. Figure~\ref{fig:hemodynamics_manifold} displays the generated datasets. The $30$ macro-vessel datasets fan out by calibre to form straight \emph{Hagen--Poiseuille} loci. The capillary dataset curves according to the shear-thinning law. This curvature remains mild over the sampled range, meaning the physiological bed lacks the pronounced non-convexity seen in the turbulent manifolds of Sec.~\ref{sec:turb_adm_da_baseline}.

\begin{figure}[h!]
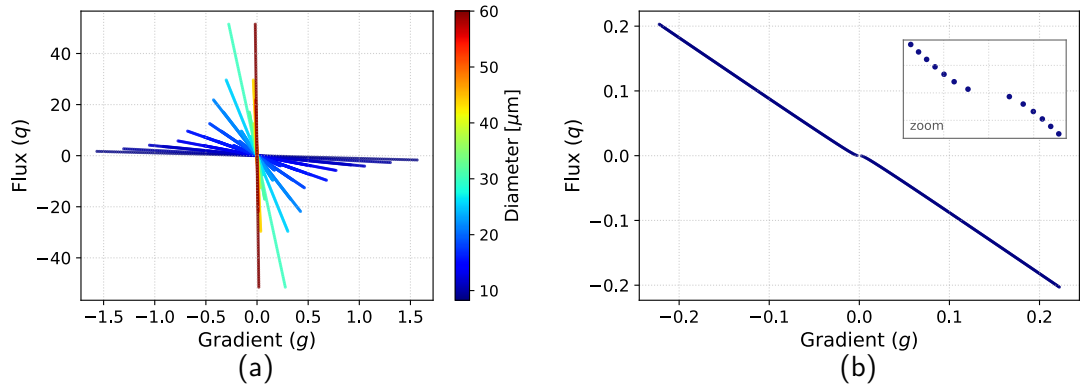

	\centering
	\svgfig{0.9\columnwidth}{450}{hemodynamics_manifold}
	\caption{Constitutive datasets of the arteriovenous bed, obtained by \emph{Local Constitutive Sampling}. (a) The $30$ macro-vessel keys, one per calibre and following the \emph{Hagen--Poiseuille} closure, colored by vessel diameter. (b) The single capillary key, following the \emph{Carreau--Yasuda} closure, with a zoom inset resolving the individual sampled points.}
	\label{fig:hemodynamics_manifold}
\end{figure}

All subsequent numerical runs cool from $\beta_0 = 0.02$ to a terminal $\beta_{\max} = 10^{7}$ over $900$ iterations. We build the CMR initialization using the $C^1$-continuous Hermite projection rather than a linear fit. Because the capillary key curves continuously under the shear-thinning closure, a single effective linear conductance misrepresents its mathematical nature. This matches our observations on the turbulent manifold in Sec.~\ref{sec:turbulent_fluid_flow}. The geometric curvature of the manifold, rather than the fluid flow regime, dictates the necessary surrogate variant. While this biological bed remains laminar on every edge, it still contains one highly nonlinear data key. The $30$ macro-vessel keys act as straight \emph{Hagen--Poiseuille} loci, which the cubic Hermite surrogate reproduces exactly.

Table~\ref{tab:hemo_solvers} compares the DA solver against the three ADM initializations based on the achieved objective and the field error against the exact reference $\mathbf{z}_{\mathrm{ex}}$. The DA solver attains the lowest objective value, with the ADM-CMR settling just $3\%$ higher. However, the ADM-CMR returns the smallest field errors, keeping all three metrics below $0.6\%$. The Null-space and Random initializations perform poorly. They finish an order of magnitude higher in objective and produce field errors four to six times larger than the well-initialized methods. 

This behavior mirrors the pattern established on the \emph{Wheatstone bridge}. A faithful reconstruction guides ADM to the optimum, while DA reaches the same neighborhood without needing a prior initialization or a complex $31$-key surrogate model. Finally, all four solvers are thermodynamically admissible across all $958$ edges. For DA, this admissibility holds even at the coldest iterate. Its fully converged state satisfies all physical constraints to machine precision ($\|\mathbf{D}\sH - \mathbf{f}\|_2 = 4.8\times10^{-15}$, $\|\eH - \mathbf{L}^{-1}\mathbf{G}\uH\|_2 = 9.3\times10^{-16}$).

\begin{table}[h!]
	\centering
	\caption{Comparison of the DA and ADM approximations for the arteriovenous bed. The field errors are relative $\ell_2$ deviations with respect to the exact nonlinear solution $\mathbf{z}_{\mathrm{ex}}$.}
	\label{tab:hemo_solvers}
	\vspace{0.2cm}
	\resizebox{0.85\textwidth}{!}{%
		\begin{tabular}{@{}lcccc@{}}
			\toprule
			\textbf{Solver} & $\mathcal{J}$
			& $\|\Delta\eH\|/\|\eH_{\mathrm{ex}}\|$ & $\|\Delta\sH\|/\|\sH_{\mathrm{ex}}\|$
			& $\|\Delta\uH\|/\|\uH_{\mathrm{ex}}\|$ \\
			\midrule
			DA               & $7.06\times10^{-4}$ & $6.09\times10^{-3}$ & $4.14\times10^{-3}$ & $6.24\times10^{-3}$ \\
			ADM-CMR          & $7.28\times10^{-4}$ & $5.62\times10^{-3}$ & $3.87\times10^{-3}$ & $4.37\times10^{-3}$ \\
			ADM-Null-space   & $9.50\times10^{-3}$ & $2.50\times10^{-2}$ & $2.56\times10^{-2}$ & $3.95\times10^{-2}$ \\
			ADM-Random       & $8.90\times10^{-3}$ & $2.98\times10^{-2}$ & $1.91\times10^{-2}$ & $3.93\times10^{-2}$ \\
			\bottomrule
	\end{tabular}}
\end{table}

Figure~\ref{fig:hemodynamics_solver_comparison}(a) visualizes the convergence histories. The DA solver plateaus early until the cooling schedule crosses the soft-to-hard transition, at which point it drops to the lowest recorded objective. The three ADM runs settle within a few tens of iterations. The poorly performing Null-space and Random runs briefly violate the second law (indicated by cross markers) before the alternating projection mechanism restores their admissibility. Figure~\ref{fig:hemodynamics_solver_comparison}(b) details the pressure-drop budget for the recovered DA solution. Of the total $45~\mathrm{mmHg}$, the dense arteriolar tree absorbs $32.27~\mathrm{mmHg}$ ($71.7\%$). The capillary bed absorbs $5.95~\mathrm{mmHg}$ ($13.2\%$), and the venular tree takes the remaining $6.78~\mathrm{mmHg}$ ($15.1\%$). This arteriole-dominated pressure split aligns with the biological role of arterioles as the primary resistance vessels in tissue beds.

\begin{figure}[h!]
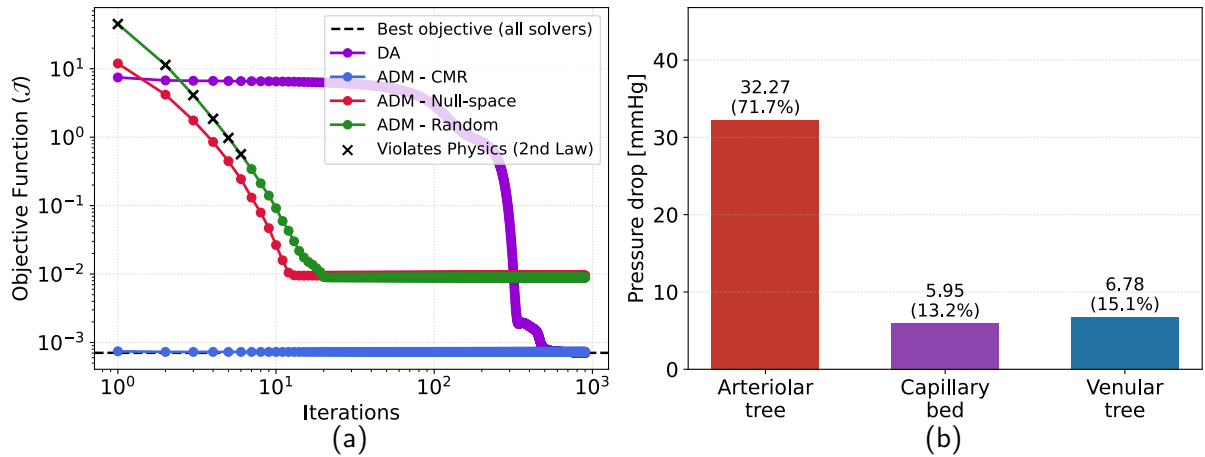

	\centering
	\svgfig{1.0\columnwidth}{505.50090926}{hemodynamics_solver_comparison}
	\caption{Solver behavior on the arteriovenous bed. (a) Evolution of the objective function $\mathcal{J}$ across the iterations; the DA solver cools from $\beta_0 = 0.02$ to $\beta_{\max} = 10^{7}$, all solvers are allocated the same budget of $900$ iterations, and cross markers ($\times$) indicate states that violate the second law of thermodynamics. (b) Pressure-drop budget of the recovered DA solution across the bed, per vessel class, as a fraction of the $45~\mathrm{mmHg}$ imposed between the root arteriole and the root venule.}
	\label{fig:hemodynamics_solver_comparison}
\end{figure}

Figure~\ref{fig:hemodynamics_da_fields} displays the fields recovered by the solver framework. The nodal pressure decreases monotonically from arteriole to venule, experiencing its steepest drop along the highly resistant arteriolar tree. The edge flux spans nearly three orders of magnitude, ranging from the main feeding trunk down to individual capillary sub-edges.

\begin{figure}[h!]
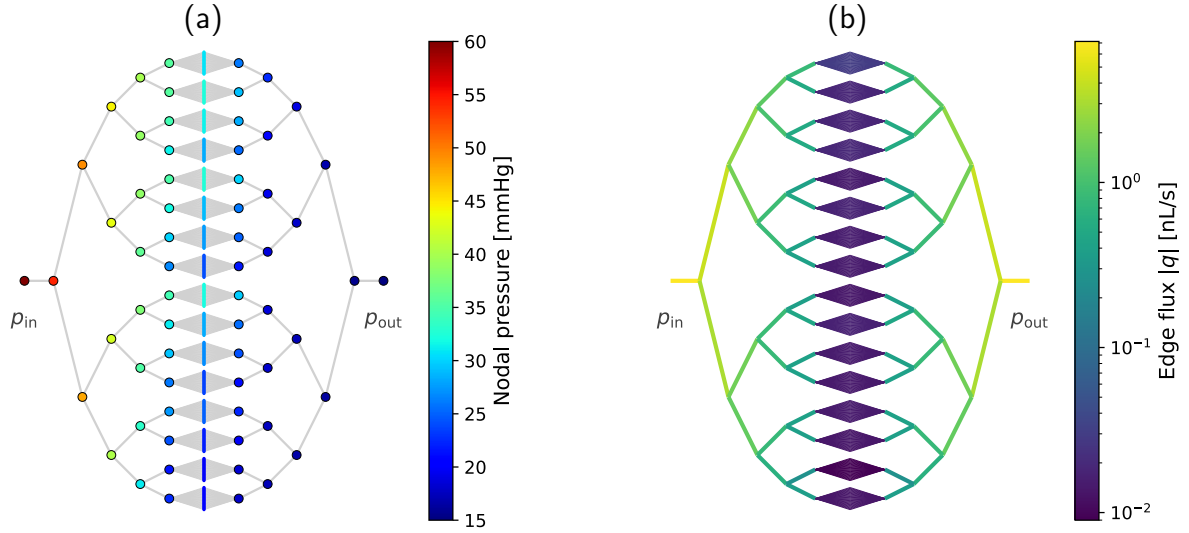

	\centering
	\svgfig{1.0\columnwidth}{450}{hemodynamics_da_fields}
	\caption{Recovered DA solution on the arteriovenous bed. (a) Nodal pressure field, in $\mathrm{mmHg}$. (b) Edge flux magnitude, in $\mathrm{nL/s}$, on a logarithmic scale, with the macro-vessel edges drawn thicker than the capillary sub-edges. Both trees are laid out as dendrograms, with the generation on the horizontal axis and the leaf pair on the vertical axis, because the fractal coordinates employed for the physics self-intersect at four generations. The node and edge counts, the generation depth and the capillary fan pairing are those of the actual network.}
	\label{fig:hemodynamics_da_fields}
\end{figure}

Table~\ref{tab:hemo_physiology} summarizes the recovered physiological solution by vessel class in dimensional terms. The capillaries operate at $0.34~\mathrm{mm/s}$ with a wall shear stress of $\tau_w = 1.32~\mathrm{Pa}$ and a microscopic Reynolds number of $\mathrm{Re} \approx 7.7\times10^{-4}$, placing them squarely in the creeping-flow regime. The arteriolar class sustains the highest wall shear stress in the bed at $11.17~\mathrm{Pa}$. This occurs because each narrow $8.3~\mu\mathrm{m}$ terminal arteriole acts as a high-pressure manifold feeding a $28$-capillary fan. Because this test bed is an idealized construction, these values demonstrate the algorithmic framework's output rather than serving as a calibrated physiological validation.

\begin{table}[h!]
	\centering
	\caption{Recovered solution by vessel class, expressed in dimensional terms. The upper block reports the recovered per-class quantities; the lower one the pressure-drop budget, as fractions of the total $45~\mathrm{mmHg}$ drop.}
	\label{tab:hemo_physiology}
	\vspace{0.2cm}
	\resizebox{0.8\textwidth}{!}{%
		\begin{tabular}{@{}lcccccc@{}}
			\toprule
			\textbf{Class} & $n_e$ & $D~[\mu\mathrm{m}]$ & $\overline{v}~[\mathrm{mm/s}]$
			& $\mathrm{Re}$ & $\tau_w~[\mathrm{Pa}]$ & $\mu_{f,\mathrm{app}}~[\mathrm{mPa\,s}]$ \\
			\midrule
			Arteriolar & $31$  & $8.3$--$30.0$ & $5.64$ & $2.70\times10^{-2}$ & $11.17$ & $3.500$ \\
			Capillary  & $896$ & $8.0$         & $0.34$ & $7.70\times10^{-4}$ & $1.32$  & $3.794$ \\
			Venular    & $31$  & $16.5$--$60.0$& $1.57$ & $1.42\times10^{-2}$ & $1.67$  & $3.500$ \\
			\midrule
			\multicolumn{7}{@{}l}{\textit{Pressure-drop budget}} \\
			Arteriolar & \multicolumn{6}{l}{$32.27~\mathrm{mmHg}$ \quad $71.7\%$} \\
			Capillary  & \multicolumn{6}{l}{$5.95~\mathrm{mmHg}$ \quad\ \ $13.2\%$} \\
			Venular    & \multicolumn{6}{l}{$6.78~\mathrm{mmHg}$ \quad\ \ $15.1\%$} \\
			\bottomrule
	\end{tabular}}
\end{table}

	\section{Conclusions}\label{sec:conclusion}

We have formulated a data-driven framework for incompressible flow in hydraulic networks. This framework is posed directly on the topological structure of a graph and covers both the laminar (linear) and turbulent (nonlinear) regimes within a single setting. By removing the constitutive equation, we end up with a mixed-integer quadratic optimization (OP), whose vector unknowns are associated with the nodal pressure field, the edgewise gradient and flux fields, whereas the integer (or binary) variable corresponds to the data assignment through the edges. In particular, we also state a saddle-point system associated with the sub-optimization problem (SOP), which arises once the data pairs have already been assigned through the edges of the graph. To solve the OP, three solution algorithms are proposed: Brute Force, Alternating Direction Method (ADM), and Deterministic Annealing (DA). These are then exposed to a wide range of numerical experiments to assess accuracy, robustness, and reliability. 

For small networks and moderate dataset sizes, the GPU-accelerated Brute Force method ensures the global optimum solution. This enables a quantitative comparison of the iterative solvers, an approach that has not been reported so far. In addition, we show that the ADM algorithm strongly depends on the initialization strategy and reaches the optimum only when starting from a faithful surrogate. For the convex datasets of the laminar regime, for instance, the ADM-CMR initialization with a linear least-squares fit is the only one of the three to reach the global optimum at every dataset resolution; the Null-space initialization only approaches it once the data are dense enough, and the Random one falls short throughout. The same pattern occurs for non-convex datasets in the turbulent regime, where the ADM-CMR reaches the global optimum in the very first iteration by initializing with a piecewise-cubic Hermite reconstruction. In this way, even when the initialization comes from the solution of the classical problem with the manifold reconstructed by a surrogate, it can bias the outcome: a poor manifold reconstruction provides a poor initialization, causing ADM to settle into a local minimum.

This dependence on the initialization is then structurally removed by the DA. Since it starts from the data centroid, it reaches the same optima without any surrogate by performing an unsupervised clustering of the measured pairs onto the network edges. Moreover, its two temperature limits (the centroid as $\beta \to 0^{+}$ and the hard nearest-neighbor assignment as $\beta \to \infty$) are approached at a rate governed by the gap between the two closest candidate data points. This leads to two practical consequences confirmed by our experiments. First, the terminal temperature must be lowered as the dataset densifies. Second, the cooling rate only matters for non-convex manifolds; it is decisive in the turbulent regime but almost insensitive in the laminar one. Despite this, the DA has a larger iteration budget, since the soft projection weighs every datum, and the cooling schedule must be slow enough for the trajectory to reach the global basin before the weights sharpen.

We also assessed these solvers under noisy datasets. As the data are refined through denser sampling and lower noise, the median objective falls monotonically for every solver. In this case, both DA and ADM-CMR track the Brute Force baseline, while the Null-space and Random initializations perform poorly. We reported accuracy and admissibility separately to distinguish optimization failures from violations of the second law. On the coarsest and noisiest laminar datasets, the Null-space and Random initializations return inadmissible solutions in about half of the realizations. In contrast, DA fails in about $18\%$ of cases, and ADM-CMR in $12\%$ to $16\%$. In the turbulent regime, almost every run is admissible, but the poorly initialized strategies remain trapped one to two orders of magnitude above the optimum. For the realizations solved admissibly, DA matches the Brute Force lower bound throughout, achieving the same accuracy as the best-initialized ADM without needing a surrogate.

Turning to more challenging scenarios, the framework was further tested beyond small benchmarks on graphs mimicking Cartesian grids of increasing dimension. In these cases, DA maintains its rate of convergence under dataset refinement regardless of the network size. Conversely, the convergence rate of ADM-CMR degrades as the grid grows. These tests demonstrate the framework's applicability to more realistic hydraulic systems where computing the global optimum is infeasible.

Finally, the last two examples applied the framework to practical scenarios. The first example assembles fittings and pipes on a single graph. Here, multiple constitutive types highlight the annealing behavior. Because different datasets harden at different temperatures, the cooling schedule must be cold enough to accommodate the slowest dataset. For example, in a square loop with elbows, the tight arm data hardens much earlier than the pipe data. The per-dataset metric weight is also decisive in this case. The two datasets have vastly different aspect ratios, causing the solvers to disagree under a plain isotropic metric. The second example explores a biomedical context using an arteriovenous bed of $958$ edges and $31$ dataset keys, featuring a Carreau--Yasuda non-Newtonian model for the capillaries. The recovered solution successfully captures an arteriole-dominated pressure budget (accounting for $71.7\%$ of the $45~\mathrm{mmHg}$ imposed pressure). Although this network is an idealized construction rather than a calibrated physiological model, it demonstrates the framework operating successfully at scale.

\textbf{Future work.} Three limitations frame the natural continuation of this work. First, thermodynamic consistency is currently monitored rather than imposed. Enforcing the Clausius--Duhem inequality as a constraint, instead of detecting its violation \emph{a posteriori}, is an immediate extension. Second, the certified optimum is only available for small networks where enumeration is affordable. On larger networks, the reference is the best admissible objective rather than a proven global minimum. Third, the datasets used here are synthetic and sampled from known equations. Using experimental data with their own confidence levels and error bounds would better test the paradigm in its intended operating environment. Other natural extensions include modeling transient flows, non-standard edge geometries like curved edges, and complex lumped elements such as bifurcations or valves. We could also explore piecewise-linear representations of the gradient and flux fields along each edge, as well as adaptive annealing schedules.

	\section*{Acknowledgments}

	\begin{sloppypar}
	The authors gratefully acknowledge the financial support from the São Paulo Research Foundation (FAPESP), under grants 2013/07375-0, 2023/14427-8 and 2025/00975-9 (RFA) and grant 2025/27460-9 (PBB, direct-doctorate scholarship); from the National Council for Scientific and Technological Development -- CNPq, under grant 308704/2022-3 (RFA); from the Coordenação de Aperfeiçoamento de Pessoal de Nível Superior -- Brazil (CAPES), Finance Code 88887.975792/2024-00 (PBB); and from the European Research Council, through the ERC Consolidator Grant ``DATA-DRIVEN OFFSHORE'', Project ID 101083157 (CGG). The numerical experiments reported here were run on the Euler cluster at ICMC-USP: research carried out using the computational resources of the Center for Mathematical Sciences Applied to Industry (CeMEAI) funded by FAPESP (grant 2013/07375-0).
	\end{sloppypar}

	\appendix
	\section{Nomenclature}

\noindent The following symbols and operators are used consistently throughout this work.

\footnotesize

\vspace{0.4em}
\noindent \textbf{Graph and Topology}
\nopagebreak
\begin{longtable}{@{}>{\raggedright\arraybackslash}p{2.6cm} >{\raggedright\arraybackslash}p{4.7cm} @{\hspace{0.5cm}} >{\raggedright\arraybackslash}p{2.6cm} >{\raggedright\arraybackslash}p{4.7cm}@{}}
	$G = (\mathcal{V}, \mathcal{E})$ & Connected graph with node set $\mathcal{V}$ and edge set $\mathcal{E}$, each edge carrying an arbitrary reference orientation & $n_v, n_e$           & Number of nodes and edges \\
	$\mathcal{V}_D$, $\mathcal{V}_N$ & Subsets of nodes with prescribed pressure and flux & $\mathbf{D} \in \mathbb{R}^{n_v \times n_e}$ & Discrete divergence (incidence) operator \\
	$\mathbf{G} \in \mathbb{R}^{n_e \times n_v}$ & Discrete gradient operator; $\mathbf{G} = -\mathbf{D}^\top$ & $\mathbf{L} \in \mathbb{R}^{n_e \times n_e}$ & Diagonal matrix of edge lengths \\
	$\mathbf{B} \in \mathbb{R}^{|\mathcal{V}_D| \times n_v}$ & Selection matrix for Dirichlet nodes & $\mathbf{I}_D \in \mathbb{R}^{n_v \times n_v}$ & Diagonal mask matrix for Dirichlet boundary conditions \\
	$\mathbf{f} \in \mathbb{R}^{n_v}$ & Prescribed nodal source/sink term; $f_i > 0$ injects fluid at node $\mathbf{v}_i$ & $\uH_D \in \mathbb{R}^{|\mathcal{V}_D|}$ & Prescribed nodal pressures on $\mathcal{V}_D$ \\
\end{longtable}

\noindent \textbf{Primal Fields and Constitutive Parameters}
\nopagebreak
\begin{longtable}{@{}>{\raggedright\arraybackslash}p{2.6cm} >{\raggedright\arraybackslash}p{4.7cm} @{\hspace{0.5cm}} >{\raggedright\arraybackslash}p{2.6cm} >{\raggedright\arraybackslash}p{4.7cm}@{}}
	$\uH \in \mathbb{R}^{n_v}$  & Vector of nodal pressures & $\eH \in \mathbb{R}^{n_e}$  & Vector of edge-wise pressure gradients \\
	$\sH \in \mathbb{R}^{n_e}$  & Vector of edge-wise volumetric fluxes & $\etH, \stH$                & Data fields (assigned pressure-gradient and flux pairs) \\
	$\mathbf{K} \in \mathbb{R}^{n_e \times n_e}$ & Diagonal matrix of hydraulic conductances & $K_k$                       & Hydraulic conductance of edge $e_k$ \\
	$R_k$                       & Hydraulic resistance per unit length of edge $e_k$ & $f_k$                       & Darcy friction factor on edge $e_k$ \\
	$\Theta_k$, $\Phi_k$        & Auxiliary coefficients of the \emph{Churchill} correlation on edge $e_k$ & $\mathrm{Re}_k$             & Reynolds number on edge $e_k$ \\
	$\mathrm{Re}_{\max}$        & Target maximum Reynolds number of the sampling & $q_{\max}$ & Largest sampled flux, $(\pi/4)\,(\mu_f/\rho)\,D\,\mathrm{Re}_{\max}$ \\
\end{longtable}

\noindent \textbf{Data-Driven Formulation}
\nopagebreak
\begin{longtable}{@{}>{\raggedright\arraybackslash}p{2.6cm} >{\raggedright\arraybackslash}p{4.7cm} @{\hspace{0.5cm}} >{\raggedright\arraybackslash}p{2.6cm} >{\raggedright\arraybackslash}p{4.7cm}@{}}
	$\mathcal{M}$               & Set of distinct dataset keys present in the network & $\mu(k)$                   & Dataset key assigned to edge $e_k$; $\mu(k) \in \mathcal{M}$ (distinct from the multiplier $\mH$ and the viscosity $\mu_f$) \\
	$\mathcal{E}_\mu$          & Subset of edges sharing dataset key $\mu$; $\mathcal{E} = \bigsqcup_{\mu} \mathcal{E}_\mu$ & $\Gamma^{(\mu)}$           & Constitutive manifold of key $\mu$; $\mathcal{D}^{(\mu)}$ is a finite sample of it \\
	$\mathcal{D}^{(\mu)}$      & Dataset of pressure-gradient/flux pairs for dataset key $\mu$ & $\mathcal{D}$               & Pooled dataset, $\mathcal{D} = \bigsqcup_{\mu} \mathcal{D}^{(\mu)}$ \\
	$\mathcal{N}_m^{(\mu)}$    & Number of data pairs in $\mathcal{D}^{(\mu)}$ & $\mathcal{N}_m$             & Number of data pairs sampled per dataset key; equals $|\mathcal{D}| = \sum_{\mu} \mathcal{N}_m^{(\mu)}$ when $|\mathcal{M}| = 1$ \\
	$\chi_n^{(\mu)}$           & Binary data-assignment function for pair $n$ of key $\mu$ & $\mathcal{S}$               & Admissible set of data-assignment variables \\
	$N_{\mathrm{comb}}$         & Number of candidate data assignments over the network, $|\mathcal{S}|$ & $\mathbf{C} \in \mathbb{R}^{n_e \times n_e}$ & Symmetric positive-definite diagonal weight matrix in $\mathcal{J}$ \\
	$C_k$                       & $k$-th diagonal entry of $\mathbf{C}$, $C_k = c^{(\mu(k))}$ & $c^{(\mu)}$                & Per-dataset metric weight; scalar $c$ when $|\mathcal{M}|=1$ \\
	$\sigma_{g}^{(\mu)}, \sigma_{q}^{(\mu)}$ & Standard deviations of the $g$ and $q$ coordinates over $\mathcal{D}^{(\mu)}$ & $\|\cdot\|_{\mathbf{C}}^2$  & $\mathbf{C}$-weighted squared distance \\
	$\mathcal{A}$               & Physically admissible set (affine subspace) & $\mathcal{Z}$               & Constitutive-data set (discrete, non-convex) \\
	$\mathbf{z} = (\eH,\sH)$    & Primal state in phase space & $\mathbf{\tilde{z}} = (\etH,\stH)$ & Data state assigned to the edges \\
	$\mathcal{J}$               & Proximity functional (objective function), weighted by $\mathbf{C}$ & $\mathcal{L}$               & Lagrangian \\
	$\lH \in \mathbb{R}^{n_v}$  & Lagrange multiplier enforcing mass balance & $\mH \in \mathbb{R}^{n_e}$  & Lagrange multiplier enforcing compatibility \\
	$P_{\mathcal{A}},\, P_{\mathcal{Z}}$ & Projection operators onto $\mathcal{A}$ and $\mathcal{Z}$ & $\hat{q}_h$                 & Piecewise-cubic Hermite surrogate of the constitutive manifold, fitted to $\mathcal{D}$ by the CMR initialization \\
	$\mathbf{M}^\dagger$        & Moore--Penrose pseudo-inverse of $\mathbf{M}$ &  &  \\
\end{longtable}

\noindent \textbf{Fluid and Geometric Properties}
\nopagebreak
\begin{longtable}{@{}>{\raggedright\arraybackslash}p{2.6cm} >{\raggedright\arraybackslash}p{4.7cm} @{\hspace{0.5cm}} >{\raggedright\arraybackslash}p{2.6cm} >{\raggedright\arraybackslash}p{4.7cm}@{}}
	$\rho$                      & Fluid density & $\mu_f$                       & Dynamic viscosity \\
	$\eta(\dot\gamma)$          & Shear-rate-dependent viscosity of a generalized Newtonian fluid (distinct from the noise level $\eta$) & $\mu_{f,\mathrm{app}}$      & Apparent viscosity of the \emph{Carreau--Yasuda} closure \\
	$\dot\gamma_k$              & Wall shear rate on edge $e_k$ & $D_k$, $r_k$, $L_k$        & Diameter, inner radius, and length of edge $e_k$ \\
	$\varepsilon_k$             & Absolute roughness of edge $e_k$ & $A_k$                       & Cross-sectional area of edge $e_k$ \\
	$\Delta p_c$, $q_c$, $K_c$, $L_c$ & Characteristic pressure drop, flux, conductance, and length & $D_c$, $v_c$                & Characteristic diameter and velocity \\
	$g_c$                       & Characteristic pressure gradient, $g_c = \Delta p_c / L_c$ & $f_{\mathrm{ref}}$          & Representative friction factor for the turbulent flux scale \\
	$\tau_w$                    & Wall shear stress on edge $e_k$, $\tau_w = (D_k/4)\,|g_k|$ & $K_L$, $K_1$, $K_\infty$    & Loss coefficient of a fitting and the two constants of the \emph{Hooper} correlation \\
	$L_{\mathrm{arm}}$          & Length of a fitting arm sub-edge, $L_{\mathrm{arm}} = 2D$ & $\alpha$                    & Daughter-branch asymmetry of the arteriovenous tree \\
\end{longtable}

\vspace{0.4em}
\noindent \textbf{Solution Algorithms}
\nopagebreak
\begin{longtable}{@{}>{\raggedright\arraybackslash}p{2.6cm} >{\raggedright\arraybackslash}p{4.7cm} @{\hspace{0.5cm}} >{\raggedright\arraybackslash}p{2.6cm} >{\raggedright\arraybackslash}p{4.7cm}@{}}
	$J_{\max}$                  & Maximum number of iterations & $k_{\mathrm{fit}}$          & Effective conductance of the linear least-squares CMR fit \\
	$J_{\mathrm{cool}}$         & Number of iterations over which $\beta$ is raised from $\beta_0$ to $\beta_{\max}$ & $\bar{\mathcal{J}}$ & Mean per-edge objective, $\bar{\mathcal{J}} = \mathcal{J}/n_e$ \\
	$\overline{\mathcal{S}}$, $\overline{\mathcal{Z}}$ & Convex hulls of $\mathcal{S}$ and $\mathcal{Z}$ & $w_n(e_k)$                  & Soft membership weight of data pair $n$ on edge $e_k$ \\
	$d_{k,n}$                   & $\mathbf{C}$-weighted distance from the state on edge $e_k$ to data pair $n$ & $d_{(1)}, d_{(2)}$          & Two smallest such distances on a given edge \\
	$\beta$, $\beta_0$, $\beta_{\max}$ & Inverse temperature, and its initial and terminal values & $\gamma$                    & Geometric cooling ratio of the annealing schedule (distinct from the shear rate $\dot\gamma$) \\
	$\mathcal{F}_k^{\beta}$     & Free-energy functional minimized by the soft projection on edge $e_k$ & $\mathcal{H}$               & Shannon entropy of the edge weight vector \\
	$N_{\mathrm{eff}}$          & Effective number of active data points per edge, $N_{\mathrm{eff}} = \exp(\mathcal{H})$ & $P_{\mathcal{Z}}^{\beta}$   & Soft (temperature-dependent) data projection \\
\end{longtable}

\vspace{0.4em}
\noindent \textbf{Synthetic Data Perturbation}
\nopagebreak
\begin{longtable}{@{}>{\raggedright\arraybackslash}p{2.6cm} >{\raggedright\arraybackslash}p{4.7cm} @{\hspace{0.5cm}} >{\raggedright\arraybackslash}p{2.6cm} >{\raggedright\arraybackslash}p{4.7cm}@{}}
	$\mathcal{D}^{\xi}$         & Noise-perturbed dataset & $\xi_n^{g}, \xi_n^{q}$      & Perturbations added to the $g$ and $q$ coordinates \\
	$\eta$                      & Noise level (standard deviation of the perturbations) &  &  \\
\end{longtable}

\vspace{0.5em}
\noindent A few letters are reused across contexts. Their specific meaning is always clear from the argument or the section. For example, $\lambda$ is a Lagrange multiplier in Secs.~\ref{sec:data_driven} and~\ref{sec:solution_algorithms}, but represents the \emph{Carreau--Yasuda} relaxation time in Secs.~\ref{sec:hyd_net_model} and~\ref{sec:numerical_experiments}. The letter $n$ indexes a data pair throughout the data-driven formulation and serves as the power-law index in the constitutive closures. Finally, $k$ indexes an edge everywhere, except in the \emph{Murray} relation (where it is the branching exponent) and in $k_{\mathrm{fit}}$ (where it is the fitted conductance).

\vspace{0.5em}
\noindent Section~\ref{sec:hyd_net_model} uses a superscript~$(\cdot)^*$ to denote dimensionless quantities alongside their dimensional counterparts. From Sec.~\ref{sec:data_driven} onward, all quantities are dimensionless, and the asterisk is dropped for readability.

\normalsize

\section{Data Generation and Nonlinear Solver}
\label{appendix:algorithms}

This appendix gives the two algorithms built on the \emph{Churchill} closure: the generation of the synthetic constitutive datasets of Sec.~\ref{sec:local_constitutive_sampling}, and the Picard scheme used for the classical nonlinear reference solutions.

\begin{algorithm}[H]
	\caption{Synthetic Constitutive Data Generation (Churchill Model)}
	\label{alg:data_generation_churchill}
	\footnotesize
	\begin{algorithmic}[1]
		\Require Pipe properties $\{(D,\varepsilon)\}$, one per key $\mu \in \mathcal{M}$; density $\rho$ and viscosity $\mu_f$; $\mathrm{Re}_{\max}$; samples per key $\mathcal{N}_m^{(\mu)}$; scales $q_c, g_c$.
		\Ensure Dimensionless dataset $\mathcal{D} = \{(\tilde{g}_n, \tilde{q}_n)\}_{n=1}^{\mathcal{N}_m}$.
		\State $\mathcal{D} \gets \emptyset$
		\For{each pipe configuration $(D, \varepsilon)$}
		\State $q_{\max} \gets \frac{\mathrm{Re}_{\max} \pi D \mu_f}{4 \rho}$ \Comment{physical flow limit}
		\State Generate a uniform grid of $\mathcal{N}_m^{(\mu)}$ flow rates $q \in [-q_{\max}, q_{\max}]$
		\For{each $q$ in the grid}
		\If{$q = 0$}
		\State $g \gets 0$ \Comment{zero-gradient condition}
		\Else
		\State $\mathrm{Re} \gets \frac{4 \rho |q|}{\pi D \mu_f}$
		\State $f \gets$ Churchill friction factor, Eq.~\eqref{eq:churchill}, at $(\mathrm{Re}, \varepsilon/D)$
		\State $g \gets -f \left( \frac{8 \rho}{\pi^2 D^5} \right) |q| q$ \Comment{\emph{Darcy--Weisbach}}
		\EndIf
		\State $\tilde{g} \gets g/g_c$, \quad $\tilde{q} \gets q/q_c$ \Comment{nondimensionalize}
		\State $\mathcal{D} \gets \mathcal{D} \cup \{(\tilde{g}, \tilde{q})\}$
		\EndFor
		\EndFor
		\State \Return $\mathcal{D}$
	\end{algorithmic}
\end{algorithm}

\begin{algorithm}[H]
	\caption{Picard Iterative Scheme for Nonlinear Hydraulic Networks}
	\label{alg:picard_solver}
	\footnotesize
	\begin{algorithmic}[1]
		\Require Topology $(\mathbf{D}, \mathbf{G}, \mathbf{L})$; boundary data $(\uH_D, \mathbf{f})$; closure $q = \hat{q}(g)$; tolerance $\epsilon_{\mathrm{tol}}$; $J_{\mathrm{max}}$; threshold $\delta$.
		\Ensure Converged primal fields $(\uH, \eH, \sH)$.
		\State Set initial uniform conductance $K_0 > 0$
		\State Solve $(\mathbf{D}(K_0\mathbf{I})\mathbf{L}^{-1}\mathbf{D}^{\top})\uH^{(0)} = \mathbf{f}$ s.t. $\mathbf{B}\uH^{(0)} = \uH_D$
		\State $\eH^{(0)} \gets \mathbf{L}^{-1}\mathbf{G}\,\uH^{(0)}$; \quad $j \gets 0$; \quad $\mathrm{error} \gets \infty$
		\While{$\mathrm{error} > \epsilon_{\mathrm{tol}}$ \textbf{and} $j < J_{\mathrm{max}}$}
		\For{each edge $e \in \mathcal{E}$}
		\If{$|\eH_e^{(j)}| < \delta$}
		\State $K_{ee}^{(j)} \gets -\lim_{g \to 0} \frac{d\hat{q}}{dg}(g)$ \Comment{analytic limit of the secant}
		\Else
		\State $K_{ee}^{(j)} \gets -\hat{q}(\eH_e^{(j)})\,/\,\eH_e^{(j)}$ \Comment{secant approximation}
		\EndIf
		\EndFor
		\State $\mathbf{A}^{(j)} \gets \mathbf{D}\mathbf{K}^{(j)}\mathbf{L}^{-1}\mathbf{D}^{\top}$ \Comment{$\mathbf{K}^{(j)} = \mathrm{diag}(K_{ee}^{(j)})$}
		\State Solve $\mathbf{A}^{(j)}\uH^{(j+1)} = \mathbf{f}$ s.t. $\mathbf{B}\uH^{(j+1)} = \uH_D$
		\State $\eH^{(j+1)} \gets \mathbf{L}^{-1}\mathbf{G}\,\uH^{(j+1)}$; \quad $\sH^{(j+1)} \gets -\mathbf{K}^{(j)}\eH^{(j+1)}$
		\State $\mathrm{error} \gets \|\uH^{(j+1)} - \uH^{(j)}\|_2 \,/\, \|\uH^{(j+1)}\|_2$; \quad $j \gets j + 1$
		\EndWhile
		\State \Return $(\uH^{(j)}, \eH^{(j)}, \sH^{(j)})$
	\end{algorithmic}
\end{algorithm}

\section{Moore--Penrose Pseudo-inverse}
\label{appendix:pseudoinverse}

The initialization strategies of Sec.~\ref{sec:adm} rest on two full-rank linear problems -- the least-squares fit of Sec.~\ref{sec:cmr}(a) and the minimum-norm state of Sec.~\ref{sec:null_space_init} -- whose unique solutions are expressed through the \emph{Moore--Penrose pseudo-inverse}, in its overdetermined (left) and underdetermined (right) forms respectively.

\begin{definition}[Moore--Penrose Pseudo-inverse]\label{def:pseudo_inverse}
	Let $\mathbf{M} \in \mathbb{R}^{m \times n}$ be of full rank. Its Moore--Penrose pseudo-inverse $\mathbf{M}^{\dagger}$ is the \emph{left} inverse when $m \geq n$ and the \emph{right} inverse when $m \leq n$:
	\[
	\mathbf{M}^{\dagger} := (\mathbf{M}^{\top}\mathbf{M})^{-1}\mathbf{M}^{\top}
	\quad (m \geq n),
	\qquad\qquad
	\mathbf{M}^{\dagger} := \mathbf{M}^{\top}(\mathbf{M}\,\mathbf{M}^{\top})^{-1}
	\quad (m \leq n).
	\]
\end{definition}

\begin{theorem}[Invertibility of the Gram matrix]\label{theo:gram_spd}
	Let $\mathbf{M} \in \mathbb{R}^{m \times n}$. If $\rank(\mathbf{M}) = n$, then $\mathbf{M}^{\top}\mathbf{M}$ is symmetric positive-definite; if $\rank(\mathbf{M}) = m$, then $\mathbf{M}\mathbf{M}^{\top}$ is symmetric positive-definite. In either case the corresponding Gram matrix is invertible, so $\mathbf{M}^{\dagger}$ is well defined.
\end{theorem}

\begin{proof}
	We argue for $\mathbf{M}^{\top}\mathbf{M}$; the case of $\mathbf{M}\mathbf{M}^{\top}$ is identical upon replacing $\mathbf{M}$ by $\mathbf{M}^{\top}$. Symmetry is immediate, and for any $\mathbf{v}$ we have $\mathbf{v}^{\top}(\mathbf{M}^{\top}\mathbf{M})\mathbf{v} = \|\mathbf{M}\mathbf{v}\|_2^2 \geq 0$. If $\rank(\mathbf{M}) = n$, the columns of $\mathbf{M}$ are linearly independent, so $\ker(\mathbf{M}) = \{\mathbf{0}\}$ and $\|\mathbf{M}\mathbf{v}\|_2^2 = 0$ only for $\mathbf{v} = \mathbf{0}$; hence $\mathbf{M}^{\top}\mathbf{M}$ is positive-definite and, in particular, invertible.
\end{proof}

Theorem~\ref{theo:gram_spd} settles both initialization problems. The least-squares objective is a strictly convex quadratic whose Hessian is the SPD Gram matrix $\mathbf{M}^{\top}\mathbf{M}$, so it has the unique minimizer $\mathbf{M}^{\dagger}\mathbf{q}_{\mathcal{D}}$ -- covering both the Hermite projection of Eq.~\eqref{eq:hermite_normal_system} and, up to the sign carried by the model $q = -k\,g$, the scalar fit of Eq.~\eqref{eq:normal_system}. The minimum-norm objective is likewise strictly convex under an affine constraint; its stationarity conditions read $\mathbf{x} = \mathbf{M}^{\top}\boldsymbol{\nu}$ and $\mathbf{M}\mathbf{x} = \mathbf{r}$, whence $\mathbf{M}\mathbf{M}^{\top}\boldsymbol{\nu} = \mathbf{r}$, and the invertibility of $\mathbf{M}\mathbf{M}^{\top}$ yields the unique solution $\mathbf{x}_0 = \mathbf{M}^{\dagger}\mathbf{r}$.

	\bibliographystyle{siam}
	\bibliography{bibliography}
	
\end{document}